\documentclass[10pt]{amsart}
\usepackage{amssymb,amsfonts,amsthm,amscd,stmaryrd,dsfont,esint,upgreek}
\usepackage[numbers]{natbib}
\usepackage[mathcal]{euscript}
\usepackage[body={5.5in,8in}, top=1.5in,left=1.50in]{geometry}
\theoremstyle{plain}
\newtheorem{thm}{Theorem}
\newtheorem{cor}{Corollary}
\newtheorem{lem}[cor]{Lemma}
\newtheorem{prop}[cor]{Proposition}
\theoremstyle{definition}

\newtheorem{remark}[cor]{Remark}

\numberwithin{cor}{section}
\numberwithin{equation}{section}

\DeclareMathOperator{\tr}{tr}
\DeclareMathOperator{\diam}{diam}
\DeclareMathOperator{\dist}{dist}
\DeclareMathOperator*{\osc}{osc}
\DeclareMathOperator*{\supt}{supp}

\DeclareMathOperator*{\esssup}{ess\,sup}
\DeclareMathOperator*{\essinf}{ess\,inf}
\DeclareMathOperator{\USC}{USC}
\DeclareMathOperator{\LSC}{LSC}
\DeclareMathOperator{\BUC}{BUC}
\DeclareMathOperator{\Var}{Var}

\renewcommand{\d}{d}
\newcommand{\ep}{\varepsilon}
\newcommand{\R}{\ensuremath{\mathds{R}}}
\newcommand{\Q}{\ensuremath{\mathds{Q}}}
\newcommand{\Rd}{\ensuremath{\mathds{R}^\d}}

\newcommand{\V}{\ensuremath{V}}

\newcommand{\N}{\ensuremath{\mathds{N}}}
\newcommand{\M}{\ensuremath{\mathcal{M}}}

\newcommand{\Sy}{\ensuremath{\mathcal{S}^\d}}
\newcommand{\Prob}{\ensuremath{\mathds{P}}}
\newcommand{\E}{\ensuremath{\mathds{E}}}

\newcommand{\iden}{I_\d}
\newcommand{\ms}{{m}}

\usepackage{tikz}
\usepackage{subfigure}
\tikzset{
  dots/.style={mark=*,only marks,mark size=.1pt,mark options={draw=#1!60,fill=#1!40}},
  dots/.default=blue
}

\newcommand{\drawaxes}{
  \draw (0,0) rectangle (1cm,1cm)
}

 \numberwithin{equation}{section}

\begin{document}

%

\title[Stochastic homogenization of Hamilton-Jacobi equations]{Stochastic homogenization of Hamilton-Jacobi \\ and degenerate Bellman equations in \\ unbounded environments}

\author{Scott N. Armstrong}
\address{Department of Mathematics\\ The University of Chicago\\ 5734 S. University Avenue
Chicago, Illinois 60637.}
\email{armstrong@math.uchicago.edu}
\author{Panagiotis E. Souganidis}
\address{Department of Mathematics\\ The University of Chicago\\ 5734 S. University Avenue
Chicago, Illinois 60637.}
\email{souganidis@math.uchicago.edu}
\date{\today}

\begin{abstract}
We consider the homogenization of Hamilton-Jacobi equations and degenerate Bellman equations in stationary, ergodic, unbounded environments. We prove that, as the microscopic scale tends to zero, the equation averages to a deterministic Hamilton-Jacobi equation and study some properties of the effective Hamiltonian. We discover a connection between the effective Hamiltonian and an eikonal-type equation in exterior domains. In particular, we obtain a new formula for the effective Hamiltonian. To prove the results we introduce a new strategy to obtain almost sure homogenization, completing a program proposed by Lions and Souganidis that previously yielded homogenization in probability. The class of problems we study is strongly motivated by Sznitman's study of the quenched large deviations of Brownian motion interacting with a Poissonian potential, but applies to a general class of problems which are not amenable to probabilistic tools.
\end{abstract}

\maketitle

\section{Introduction}
\label{IN}

The primary goal of this paper is the study of the behavior, as $\ep \to 0$, of the solution $u^\ep=u^\ep(x,t,\omega)$ of the initial value problem
\begin{equation} \label{HJ}
\left\{ \begin{aligned}
&u^\ep_t - \ep \tr \!\left( A(\tfrac x\ep, \omega) D^2u^\ep \right) + H(Du^\ep, \tfrac x\ep, \omega) = 0 & \mbox{in} & \ \Rd \times \R_+, \\
& u^\ep = u_0 & \mbox{on} & \ \Rd \times \{ 0 \}.
\end{aligned} \right.
\end{equation}
Here $A=A(y,\omega)$ and $H=H(p,y,\omega)$ are random processes as they depend on $\omega$, an element of an underlying probability space $(\Omega, \mathcal F,\Prob)$, and the initial datum $u_0$ is a bounded, uniformly continuous function on $\Rd$. Precise notation and hypotheses on the coefficients are given in the following section, but we mention here that the diffusion matrix $A$ is degenerate elliptic and the Hamiltonian $H$ is convex in $p$ and, in general, unbounded from below, both in $y$ for fixed $(p,\omega)$ and in $\omega$ for fixed $(p,y)$. The random coefficients $A$ and $H$ are required to be stationary and ergodic, and, furthermore, the unboundedness of $H$ is controlled by a random variable which is assumed to be strongly mixing.

Model equations satisfying our hypotheses include the viscous Hamilton-Jacobi equation
\begin{equation} \label{PCP}
u_t^\ep - \ep \Delta u^\ep + |Du^\ep|^\gamma - V(\tfrac x\ep, \omega) = 0,
\end{equation}
and the first-order Hamilton-Jacobi equation
\begin{equation} \label{PCPnv}
u_t^\ep + |Du^\ep|^\gamma - V(\tfrac x\ep, \omega) = 0,
\end{equation}
where $\gamma > 1$ and $V$ is, for example, a Poissonian potential.

\medskip

The first of the main results (Theorem~\ref{MAIN}, which is stated in Section~\ref{MR}) of the paper is the identification of a deterministic effective nonlinearity $\overline H=\overline H(p)$ combined with the assertion that, as $\ep \to 0$, the solutions $u^\ep$ of \eqref{HJ} converge, locally uniformly in $(x,t)$ and almost surely in $\omega$, to a deterministic function $u=u(x,t)$, the unique solution of the initial-value problem
\begin{equation} \label{HJ-eff}
\left\{ \begin{aligned}
&u_t  +\overline H(Du) = 0 & \mbox{in} & \ \Rd \times \R_+, \\
& u = u_0 & \mbox{on} & \ \Rd \times \{ 0 \}.
\end{aligned} \right.
\end{equation}
To obtain almost sure convergence, we study the asymptotic behavior of the special problem
\begin{equation} \label{meteqn}
\left\{ \begin{aligned}
& -\ep \tr\!\left(A(\tfrac x{\ep},\omega)D^2m_\mu^\ep\right) + H(p+Dm_\mu^\ep,\tfrac x{\ep},\omega) = \mu & \mbox{in} & \ \Rd \! \setminus \! D_\ep, \\
& m_\mu^\ep = 0 & \mbox{on} & \ \partial D_\ep, \\
& \lefteqn{0 \leq \liminf_{|x|\to\infty} |x|^{-1} m_\mu^\ep(x,\omega)}
\end{aligned} \right.
\end{equation}
where, depending on some parameters in the equation (e.g., the growth of $H$), either $D_\ep=\{ 0\}$ or $D_\ep = B_\ep$. The problem \eqref{meteqn}, which we call the \emph{metric problem}, is of independent interest. We prove that it has a unique solution for each $\mu > \overline H(p)$ and no solution for $\mu < \overline H(p)$, thereby characterizing the effective Hamiltonian $\overline H$ in a new way. We also provide a new characterization of the $\min \overline H$. The functions $m^\ep_\mu$ possess a subadditivity property sufficient to apply the subadditive ergodic theorem and obtain, in the limit $\ep \to 0$, the almost sure convergence of $m_\mu^\ep$ to the unique solution $\overline m_\mu$ of
\begin{equation} \label{meteqnbar}
\left\{ \begin{aligned}
& \overline H(p+D\overline m_\mu) = \mu  \quad \mbox{in}  \ \Rd \! \setminus \! \{ 0 \}, \\
& \overline m_\mu(0) = 0, \\
& 0 \leq \liminf_{|y|\to\infty} |y|^{-1} \overline m_\mu(y).
\end{aligned} \right.
\end{equation}
It turns out that the almost sure behavior of the solutions $u^\ep$ of \eqref{HJ} are controlled by $m^\ep_\mu$ and the characterization of $\min \overline H$, and this allows us to deduce the convergence of $u^\ep$ almost surely.

\medskip

There has been much recent interest in the homogenization of partial differential equations in stationary ergodic random environments. While the linear case was settled long ago by Papanicolaou and Varadhan \cite{PV1,PV2} and Kozlov \cite{K}, and general variational problems were studied by Dal Maso and Modica \cite{DM2,DM1} (see also Zhikov, Kozlov, and Ole{\u\i}nik \cite{ZKO}), it was only relatively recently that nonlinear problems were considered (in bounded environments). Results for stochastic homogenization of Hamilton-Jacobi equations were first obtained by Souganidis \cite{S} (see also Rezakhanlou and Tarver \cite{RT}), and for viscous Hamilton-Jacobi equations by Lions and Souganidis \cite{LS2} and Kosygina, Rezakhanlou, and Varadhan \cite{KRV}. The homogenization of these equations in spatio-temporal media was studied by Kosygina and Varadhan \cite{KV} and Schwab \cite{Sch}. Davini and Siconolfi~\cite{DS} also proved, using some connections to Mather theory, some interesting results for Hamilton-Jacobi equations in bounded environments. We also mention the work of Caffarelli, Souganidis and Wang \cite{CSW} on the stochastic homogenization of uniformly elliptic equations of second-order, Caffarelli and Souganidis \cite{CS} who obtained a rate of convergence for the latter in strongly mixing environments, and Schwab~\cite{Sch2} on the homogenization of nonlocal equations. A new proof of the results of \cite{KRV,LS2,RT,S}, which yields only convergence in probability but does not rely on formulae, has recently been found by Lions and Souganidis \cite{LS3}. We adapt the approach of \cite{LS3} to our setting and upgrade the convergence to almost sure, using the ``metric problem'' introduced in Section~\ref{Q}.

\medskip

As far as unbounded environments are concerned, Sznitman \cite{Sz1,Sz2,Sz3, Szb} studied the behavior of Brownian motions in the presence of Poissonian obstacles and obtained very elegant and complete results concerning the asymptotic behavior and large deviations of such processes, both in the quenched and annealed settings 
Although the results of Sznitman were not stated in terms of homogenization theory, some of them may be formulated in terms of the stochastic homogenization of \eqref{PCP} for $\gamma =2$. Part of our work can be seen as an extension of Sznitman's results to more general equations, including degenerate Bellman equations and Hamilton-Jacobi equations, as well as to more general random environments. In particular, we consider equations which cannot be written as a linear Schr\"odinger equation via the Hopf-Cole transformation, and general mixing rather than i.i.d. environments.

More recently, and after this paper was accepted, Rassoul-Agha, Sepp\"al\"ainen and Yilmaz posted a paper \cite{RSY} to the arXiv which considers quenched free energy and large deviations for a random walk in a random (unbounded) potential. This corresponds to a ``discrete'' version of the work of Sznitman, but with more general potentials, and to ours, in the case $\gamma=2$. The potential in \cite{RSY} is assumed to satisfy a strong mixing condition in order to compensate for its unboundedness, which is very similar (although not identical) to some of our hypotheses below.

\medskip

We remark that problems depending on the macroscopic variable~$x$ such as
\begin{equation} \label{HJx}
\left\{ \begin{aligned}
&u^\ep_t - \ep \tr \!\left( A(x,\tfrac x\ep, \omega) D^2u^\ep \right) + H(Du^\ep,x, \tfrac x\ep, \omega) = 0 & \mbox{in} & \ \Rd \times \R_+, \\
& u^\ep = u_0 & \mbox{on} & \ \Rd \times \{ 0 \},
\end{aligned} \right.
\end{equation}
can be handled by the methods in this paper, leading to an effective equation depending also on $x$. This requires proving a continuous dependence estimate (which is more sophisticated than Proposition~\ref{CDE-noxd}), so that we can work with a countable number of subsets of $\Omega$. This can be achieved by combining standard viscosity theoretic arguments (c.f. \cite{CIL}) with our Lemma~\ref{uppbndosc}. However, for clarity, we confine our attention in this paper to \eqref{HJ}, and pursue such generalizations in a future work.

\medskip

While in this work we focus on equations such as \eqref{PCP} and \eqref{PCPnv} with $V$ unbounded, our proof strategy can be applied to viscous and non-viscous Hamilton-Jacobi equations with other kinds of ``degenerate coercivity." For example, with minor modifications, our techniques can handle equations such as
\begin{equation} \label{perc}
u^\ep_t - \ep \delta \Delta u^\ep + a\big( \tfrac x\ep, \omega) |Du^\ep|^\gamma = 1,
\end{equation}
where $\delta\in \{ 0,1\}$, and the function $a(y,\omega) > 0$ is stationary ergodic, but not necessarily bounded below. This lack of coercivity can be compensated for by adding a strong mixing hypothesis to $a$ as well as the assumption that a large moment of $a^{-1}$ is bounded, hypotheses which are similar to those imposed on $V$ in our model equation \eqref{PCP} and \eqref{PCPnv}. The problem \eqref{perc} with $\delta = 0$ corresponds to a continuum model for \emph{first passage percolation}, where the medium imposes a cost $a^{-1}(y,\omega)$ to travel near $y$, which may be arbitrarily large. See Kesten~\cite{Ke}.

\medskip

Finally we note that our proofs also yield the homogenization of of time independent versions of \eqref{HJ} and \eqref{HJx}, i.e., to equations of the general form
\begin{equation}\label{hjx}
u^\ep - \ep \tr \!\left( A(x,\tfrac x\ep, \omega) D^2u^\ep \right) + H(Du^\ep,x, \tfrac x\ep, \omega) = 0  \quad \mbox{in} \ U\end{equation}
for a domain $U$ with appropriate boundary conditions.

\medskip

The paper is organized as follows. In the next section, we review the notation, introduce the precise assumptions, and consider some motivating examples. In Section~\ref{MR} we state the main result and give an overview of its proof, which is completed in Section~\ref{PH}. A key ingredient of the proof is the study of the macroscopic (cell) problem, which is the content of Section~\ref{AMP}. The effective Hamiltonian is identified and its basic properties are studied in Section~\ref{EH}. In Section~\ref{Q} we consider the metric problem \eqref{meteqn} and obtain almost sure homogenization with the help of the subadditive ergodic theorem. In Section~\ref{P} we compare our results for \eqref{PCP} and $\gamma=2$ with those of Sznitman, explain some of the connections to probability theory, and give a new characterization of $\overline H$ in terms of the solvability metric problem. Several auxiliary lemmata are recorded in Appendix~\ref{A}.

\section{Preliminaries}
\label{PL}

We briefly review the notation and state precisely the assumptions on \eqref{HJ} before discussing some motivating examples.

\subsection{Notation and conventions}

The symbols $C$ and $c$ denote positive constants, which may vary from line to line and, unless otherwise indicated, do not depend on $\omega$. We work in the $\d$-dimensional Euclidean space $\Rd$ with $\d \geq 1$, and we write $\R_+:=(0,\infty)$. The sets of rational numbers and positive integers are denoted respectively by $\Q$ and $\N$, while  $\M^{m\times \d}$, and $\Sy \subseteq \M^{\d\times \d}$ are respectively the sets of $m$-by-$\d$ matrices and  $\d$-by-$\d$ symmetric matrices. If $v,w\in \Rd$, then $v\otimes w \in \Sy$ is the symmetric tensor product which is the matrix with entries $\frac12(v_iw_j+ v_jw_i)$. For $y \in \Rd$, we denote the Euclidean norm of $y$ by $|y|$, while if $M\in \M^{d\times \d}$, $M^t$ is the transpose of $M$, $\tr(M)$ is the trace of $M$, and we write $|M| := \tr(M^t M)^{1/2}$. The identity matrix is $\iden$. If $U \subseteq\Rd$, then $|U|$ is the Lebesgue measure of $U$. Open balls are written $B(y,r): = \{ x\in \Rd : |x-y| < r\}$, and we set $B_r : = B(0,r)$. The distance between two subsets $U,V\subseteq \Rd$ is denoi ted by $\dist(U,V) = \inf\{ |x-y|: x\in U, \, y\in V\}$. If $U\subseteq \Rd$ is open, then $\USC(U)$, $\LSC(U)$ and $\BUC(U)$ are the sets of upper semicontinuous, lower semicontinuous and bounded and uniformly continuous functions $U\to \R$, respectively. If $f:U\to \R$ is integrable, then $\fint_U f \, dy$ is  the average of $f$ over $U$. If $f:U \to \R$ is measurable, then we set $\osc_U f:= \esssup_U f - \essinf_U f$. The Borel $\sigma$-field on $\Rd$ is denoted by $\mathcal{B}(\Rd)$. If $s,t\in \R$, we write $s\wedge t : = \min\{ s,t\}$.

\medskip

We emphasize that, throughout this paper, all differential inequalities involving functions not known to be smooth are assumed to be satisfied in the viscosity sense. Wherever we refer to ``standard viscosity solution theory" in support of a claim, the details can always be found in the standard reference \cite{CIL}.

We abbreviate the phrase \emph{almost surely in} $\omega$ by ``a.s. in $\omega$." To keep the bookkeeping simple, $\Omega_1\supseteq\Omega_2 \supseteq \Omega_3 \supseteq \cdots$ is a decreasing sequence of subsets of $\Omega$ of full probability used to keep track of almost sure statements. Roughly, $\Omega_j$ is the subset of $\Omega$ of full probability on which all the almost sure statements appearing in the paper prior to the introduction of $\Omega_j$ hold.

\subsection{The random medium}
The random environment is described by a probability space $(\Omega, \mathcal F, \mathds P)$. A particular ``medium" is an element $\omega \in\Omega$. The probability space is endowed with an ergodic group $(\tau_y)_{y\in \Rd}$ of $\mathcal F$-measurable, measure-preserving transformations $\tau_y:\Omega\to \Omega$. Here \emph{ergodic} means that, if $D\subseteq \Omega$ is such that $\tau_z(D) = D$ for every $z\in \Rd$, then either $\Prob[D] = 0$ or $\Prob[D] = 1$. An $\mathcal F$-measurable process $f$ on $\Rd \times \Omega$ is said to be \emph{stationary} if the law of $f(y,\cdot)$ is independent of $y$. This is quantified in terms of $\tau$ by the requirement that
\begin{equation*}
f(y,\tau_z \omega) = f(y+z,\omega) \quad \mbox{for every} \ y,z\in \Rd.
\end{equation*}
Notice that if $\phi:\Omega \to S$ is a random process, then $\tilde \phi(y,\omega) : = \phi(\tau_y\omega)$ is stationary. Conversely, if $f$ is a stationary function on $\Rd \times \Omega$, then $f(y,\omega) = f(0,\tau_y\omega)$.

The expectation of a random variable $f$ with respect to $\mathds P$ is written $\E f$, and we denote the variance of $f$ by $\Var(f): = \E(f^2) - (\E f)^2$. If $E \in \mathcal F$, then $\mathds{1}_E$ is the indicator random variable for $E$; i.e., $\mathds{1}_E(\omega) = 1$ if $\omega\in E$, and $\mathds{1}_E(\omega) = 0$ otherwise.

\medskip

Many times in this paper, we rely on the following multiparameter ergodic theorem, a proof of which can be found in Becker~\cite{B}.

\begin{prop} \label{ergthm}
Suppose that $f:\Rd \times \Omega \to \R$ is stationary and $\E |f(0,\cdot)| < \infty$. Then there is a subset $\widetilde \Omega \subseteq \Omega$ such that ${\mathds P}[\widetilde \Omega] =1$ and , for each bounded domain $V \subseteq \Rd$ and $\omega \in \widetilde \Omega$,
\begin{equation*}
\lim_{t\to \infty} \fint_{tV} f(y,\omega) \, dy = \E f.
\end{equation*}
\end{prop}

We also make use of a subadditive ergodic theorem, the statement of which requires some additional notation. Let $\mathcal{I}$ denote the class of subsets of $[0,\infty)$ which consist of finite unions of intervals of the form $[a,b)$. Let $\{\sigma_t\}_{t\geq 0}$ be a semigroup of measure-preserving transformations on $\Omega$. A \emph{continuous subadditive process} on $(\Omega, \mathcal F, \Prob)$ with respect to $\sigma_y$ is a map
\begin{equation*}
Q: \mathcal I \rightarrow L^1(\Omega,\Prob)
\end{equation*}
which satisfies the following three conditions:
\begin{enumerate}
\item[(i)] $Q(I)(\sigma_t\omega) = Q( y+I)(\omega)$ for each $t>0$, $I \in \mathcal I$ and a.s. in $\omega$.
\item[(ii)] There exists a constant $C> 0$ such that, for each $I\in \mathcal I$,
\begin{equation*}
\E \big| Q(I) \big| \leq C |I|.
\end{equation*}
\item[(iii)] If $I_1,\ldots I_k \in \mathcal I$ are disjoint and $I=\cup_j I_j$, then 
\begin{equation*}
Q(I) \leq \sum_{j=1}^k Q(I_j).
\end{equation*}
\end{enumerate}
A proof of the following version of the subadditive ergodic theorem can be found in Akcoglu and Krengel~\cite{AK}.

\begin{prop} \label{SAergthm}
Suppose that $Q$ is a continuous subadditive process. Then there is a random variable $a(\omega)$ such that
\begin{equation*} 
\frac1t Q([0,t))(\omega) \rightarrow a(\omega) \quad \mbox{a.s. in} \ \omega.
\end{equation*}
If, in addition, $\{ \sigma_t \}_{t>0}$ is ergodic, then $a$ is constant. 
\end{prop}

\subsection{Assumptions on the coefficients}
The following hypotheses are in force throughout this article. The environment is given by a probability space $(\Omega, \Prob,\mathcal F)$, and
\begin{equation} \label{erghyp}
\tau_y : \Omega \to \Omega \mbox{is an ergodic group of measure-preserving transformations}
\end{equation}
as described above. 

The following hypotheses we impose on the matrix~$A$ and the Hamiltonian~$H$ are taken to hold for \emph{every} $\omega \in \Omega$, rather than almost surely in $\omega$, since we lose no generality by initially removing an event of probability zero. We require $A$ and $H$ to be functions
\begin{equation*}
A:\Rd\times\Omega \to \R \qquad \mbox{and} \qquad H:\Rd\times\Rd\times\Omega\to\R,
\end{equation*}
and we require that, for every $p \in \Rd$,
\begin{equation} \label{AHstat}
(y,\omega) \mapsto A(y,\omega) \quad \mbox{and} \quad (y,\omega) \mapsto H(p,y,\omega) \quad \mbox{are stationary.}
\end{equation}
We assume that, for each $\omega \in \Omega$,
\begin{equation} \label{AC11}
A(\cdot,\omega)  \in C^{1,1}_\mathrm{loc}(\Rd;\Sy),
\end{equation}
and that $A$ has the form $A  = \Sigma \Sigma^t$, where
\begin{equation}\label{Alip}
\Sigma(y, \omega) \in \M^{\d \times m} \quad \mbox{satisfies} \quad \| \Sigma(\cdot,\omega) \|_{C^{0,1}(\Rd)} \leq C.
\end{equation}
As far as the Hamiltonian $H$ is concerned, we assume that, for each $(y,\omega)\in \Rd\times \Omega$,
\begin{equation} \label{Hconv}
p \mapsto H(p,y,\omega) \quad \mbox{is convex,}
\end{equation}
and that there exists $1 < \gamma < \infty$ such that, for every $(p,y,\omega) \in\Rd\times\Rd\times\Omega$,
\begin{equation} \label{Hbound}
H(p,y,\omega) \leq C \left(1+|p|^\gamma \right).
\end{equation}
The coercivity and Lipschitz regularity in $y$ of the Hamiltonian $H$ are governed by a nonnegative stationary process $V=V(y,\omega)$, as follows. We assume that there exist constants $c_0, C_0>0$ such that, for every $(p,y,\omega) \in \Rd\times \Rd \times \Omega$,
\begin{equation} \label{Hcoer}
H(p,y,\omega) \geq c_0 |p|^\gamma - V(y,\omega) - C_0,
\end{equation}
and we require that, for every $(p_1,y_1), (p_2, y_2) \in \Rd \times \Rd$ and $\omega \in \Omega$,
\begin{multline} \label{Hlip}
\left| H(p_1,y_1,\omega) - H(p_2,y_2,\omega) \right|\\  \leq C \Big( \left(1+ |p_1|+|p_2| \right)^{\gamma-1} |p_1-p_2| + \Big( \sup_{[y_1,y_2]} V(\cdot,\omega)\Big) |y_1-y_2| \Big).
\end{multline}
For fixed positive constants $\alpha, \beta > 0$, satisfying
\begin{equation} \label{alphabeta}
\alpha > \d  \qquad \mbox{and} \qquad \beta > \d + \frac{4\d^2}{\gamma\alpha},
\end{equation}
the process $V=V(y,\omega) \geq 0$ is assumed to possess the following two properties:
\begin{equation} \label{Vmoment}
\E[\sup_{y\in B_1} V(y,\cdot)^\alpha] < \infty,
\end{equation}
and
\begin{equation} \label{Vmix}
 V \ \mbox{is strongly mixing with respect to} \ \tau \ \mbox{with an algebraic rate of} \  \beta.
\end{equation}
The mixing condition is an assumption on the rate of decay of the correlation between the random variables $V(y,\omega)$ and $V(z,\omega)$ as $|y-z|$ becomes large. To state this precisely, for each open set $U\subseteq \Rd$, we denote by $\mathcal{G}(U)$ the $\sigma$-algebra generated by the sets $\{ \omega : V(\omega,y) \in [0,a)\}$ ranging over $y\in U$ and $a>0$. The assumption \eqref{Vmix} is then written precisely as
\begin{equation} \label{mixing}
\sup\left\{ \left| \Prob(E\cap F) - \Prob(E)\Prob(F) \right| :\ \dist(U,V) \geq r, \ E\in \mathcal{G}(U), \ F\in \mathcal{G}(V) \right\} \leq Cr^\beta.
\end{equation}

Observe that our hypotheses in the case of the model equations \eqref{PCP} and \eqref{PCPnv} allow the potential $V$ to be unbounded from above, but require it to be bounded from below. There is a very good reason for this; see Remark~\ref{invispot-easy} below.

\medskip

We emphasize that the assumptions \eqref{erghyp}, \eqref{AHstat}, \eqref{AC11}, \eqref{Alip}, \eqref{Hconv}, \eqref{Hbound}, \eqref{Hcoer}, \eqref{Hlip}, \eqref{alphabeta}, \eqref{Vmoment}, and \eqref{mixing} are in force throughout this article.

\medskip

As usual, measurability is an burdensome issue. To avoid an overly pedantic presentation, we are going to suppress demonstrations of the measurability of the various functions and processes we encounter in this paper. For the benefit of the concerned reader, however, we sketch here a proof that all of our random variables are measurable. First, the functions of which we need to check the measurability solve equations with random coefficients for which we have uniqueness. In fact, these solutions depend continuously on the coefficients, and the coefficients are measurable functions of $\omega$ combined with the other variables. Therefore, as continuous functions of measurable functions, they are measurable.

\subsection{Cloud point potentials}

We present a simple example of a model equation that fits our framework. Take $H$ to be a Hamiltonian of the form \begin{equation*}H(p,y,\omega) = \widetilde H(p) - V(y,\omega)\end{equation*} where $\widetilde H$ is an appropriate deterministic function, for instance $\widetilde H(p) =  |p|^\gamma$ with $\gamma \geq 1$, and $V$ is a stationary potential satisfying \eqref{Vmoment} and \eqref{Vmix}.

\medskip

We construct some examples of unbounded potentials by first specifying a random process which generates a cloud of points (see Figure~1), and then attach to each point a smooth nonnegative function of compact support. We first briefly recall the notion of a point process on $\Rd$, referring to the books \cite{DVJ1,DVJ2} for more details.

\begin{figure}\label{cloudpic}
\centering

\subfigure[A Poisson cloud] {
  \label{poissonfig}

  \begin{tikzpicture}[scale=6]
    \drawaxes;

    \draw plot[dots=black] file{poisson.dat};
  \end{tikzpicture}
}
\subfigure[A cluster point cloud] {
  \label{clusterfig}

  \begin{tikzpicture}[scale=6]
    \drawaxes;

    \draw plot[dots=black] file{cluster.dat};
  \end{tikzpicture}
}

\caption{Point processes}
\end{figure}
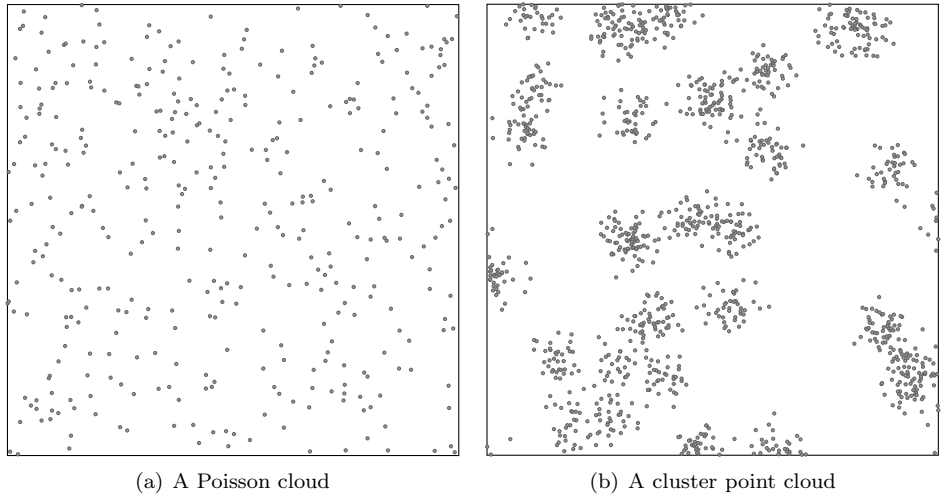

Consider the set $\Omega$ of locally finite, simple pure point measures on $\Rd$. An element of $\Omega$ has the form
\begin{equation} \label{omeform}
\omega = \sum_{j=1}^\infty \updelta_{y_j}
\end{equation}
where the points $y_j$ are distinct and the set $\{ y_j\} \cap B_R$ is finite for each $R> 0$. Here $\updelta_y$ denotes the Dirac measure at $y \in \Rd$. Let $\mathcal F$ be the $\sigma$-field generated by the maps $\omega \mapsto \omega(U)$ ranging over all $U \in \mathcal{B}(\Rd)$. It follows that $E + y :=\{ \omega(\cdot -y) : \omega\in E \}\in \mathcal F$ for any $E\in \mathcal F$.

Fix a smooth $W \in C(\Rd ;{\bar \R}_+)$ with compact support such that $W(0)=1$ and define $V(y,\omega)$, for each $\omega \in \Omega$ of the form \eqref{omeform}, by
\begin{equation} \label{Vexam}
V(y,\omega) : = \int_{\Rd} W(y-z) \, d\omega(z) = \sum_{j=1}^\infty W(y-y_j).
\end{equation}
The law of $V$ is inherited from the probability measure $\Prob$ we attach to the measurable space $(\Omega,\mathcal F)$. The canonical choice of the probability measure $\Prob$ is to take it to be the Poisson law $\Prob_\nu$
which is characterized uniquely by the properties:
\begin{enumerate}
\item[(i)] for every $U \in \mathcal B(\Rd)$, $\E_\nu[\omega(U) ] = \nu |U|$,
\item[(ii)] if $k\in \N$ and $U_1, \ldots, U_k \in \mathcal{B}(\Rd)$ are disjoint, then the random variables $\omega \mapsto \omega(U_j)$ are independent, and
\item[(iii)] $\Prob_\nu[ E ] = \Prob_\nu[ E+y]$ for each $E \in \mathcal F$.
\end{enumerate}
Under $\Prob_\nu$, the canonical process $\omega$ is called a \emph{Poisson point process with intensity} $\nu$. It is easy to see that, with respect to $\Prob_\nu$, the potential $V$ given by \eqref{Vexam} is unbounded in $\omega$ but satisfies \eqref{Vmoment} and \eqref{Vmix} for any $\alpha,\beta > 0$. In fact, it follows from (ii) that $V(y,\cdot)$ and $V(z,\cdot)$ are independent, provided that $|y-z| > \diam (\supt(W))$.

There are many examples of stationary point processes one may consider giving rise to potentials $V$ satisfying our hypotheses, and the reader is invited to consult \cite{DVJ1,DVJ2} for many more. We mention one other type, a \emph{cluster process}, in which \emph{cluster center points} are chosen according to a Poisson law and around each cluster center point a certain (possibly random) number of points are placed according to a predetermined distribution. For instance, in Figure~\ref{clusterfig} we have placed a Poisson number of points around each cluster center point according to a Gaussian distribution. If the clusters are not uniformly bounded, that is, points may land arbitrarily far from their cluster center with positive probability, then the process will not be i.i.d. but may be strongly mixing. While mixing conditions are difficult to check in practice, sufficient conditions for a cluster point process to satisfy a mixing condition such as \eqref{mixing} can be found in Laslett \cite{L}. Likewise, if the function $W$ in \eqref{Vexam} has unbounded support, then the potential $V$ is not i.i.d. but may nonetheless satisfy strong mixing hypothesis, depending on the rate of decay of $W(y)$ for large $|y|$.

\section{The main results and proof overview}
\label{MR}

With the hypotheses stated in the previous section in force, we now state the two main results.

\begin{thm} \label{MAIN}
There exist $\overline H:\Rd\to \R$ and $\Omega_0 \subseteq \Omega$ of full probability such that, for each $\omega \in \Omega_0$ and $u_0\in \BUC(\Rd)$, the unique solution $u^\ep = u^\ep(\cdot,\omega)$ of \eqref{HJ} given in Proposition~\ref{epWP} converges locally uniformly in $\Rd \times \R_+$, as $\ep \to 0$, to the unique solution $u$ of \eqref{HJ-eff} which belongs to $\BUC(\Rd\times [0,T])$ for every $T > 0$.
\end{thm}

\begin{thm}\label{MAIN1}
There exists a subset $\Omega_0 \subseteq \Omega$ of full probability such that, for each $\omega \in \Omega_0$ and $\mu > {\overline H}(p)$, with  ${\overline H}:\Rd\to \R$ as in Theorem~\ref{MAIN}, the unique solution $m^\ep_\mu=m^\ep_\mu(\cdot,\omega)$ of \eqref{meteqn} given in Proposition~\ref{metexistence}
converges, as $\ep \to 0$, locally uniformly in $\Rd\setminus\{0\}$ to the unique solution ${\overline m}_\mu$ of \eqref{meteqnbar}. Moreover, ${\overline H}(p)$ is the unique constant such that  \eqref{meteqnbar} has a unique solution for all $\mu > {\overline H}(p)$ and no solution if
 $\mu < {\overline H}(p)$.
\end{thm}

The effective nonlinearity $\overline H$ is identified in Proposition~\ref{mainstep} below. Some of its basic properties are summarized in Proposition~\ref{effHam}. We remark that  these properties yield that \eqref{HJ-eff} has a unique bounded uniformly continuous solution on $\Rd\times [0,T]$ for any $T> 0$. The well-posedness of \eqref{HJ} is discussed briefly in Section~\ref{PH}.

\medskip

As the proofs of Theorems~\ref{MAIN} and~\ref{MAIN1} are rather lengthy and involved, we devote the remainder of this section to summarizing the primary obstacles and main ideas needed to overcome them. To simplify our exposition, we consider the time-independent problem
\begin{equation} \label{stat}
u^\ep - \ep \tr \!\left( A(\tfrac x\ep, \omega) D^2u^\ep \right) + H(Du^\ep, \tfrac x\ep, \omega) = 0 \quad \mbox{in}  \ \Rd. \\
\end{equation}
Assume that $u^\ep$ admits the asymptotic expansion
\begin{equation} \label{asyexp}
u^\ep(x,\omega) = u(x) + \ep w(\tfrac x\ep, \omega) + O(\ep^2).
\end{equation}
Inserting this into \eqref{stat} and performing a formal computation, we arrive at
\begin{equation}\label{inserhr}
u - \tr \!\left( A (y, \omega) D^2_y w \right) + H(D_xu + D_yw,y, \omega) = 0.
\end{equation}
If the expression on the left of \eqref{inserhr} is independent of $(y,\omega)$, then we obtain an equation for~ $u$. That is, we suppose that, for each $p\in \Rd$, there exists a constant $\overline H = \overline H(p)$ and a function $w=w(y,\omega)$ such that
\begin{equation} \label{cell}
- \tr \!\left( A(y, \omega) D^2_y w \right) + H(p + D_yw,y, \omega) = \overline H(p) \quad \mbox{a.s. in} \ \omega.
\end{equation}
Substituting into \eqref{inserhr}, we thereby obtain the effective equation
\begin{equation} \label{effeq}
u + \overline H(Du) = 0 \quad \mbox{in} \ \Rd.
\end{equation}
Returning to \eqref{asyexp}, we see that the convergence $u^\ep \to u$ locally uniformly in $\Rd$, as $\ep \to 0$, is formally equivalent to $\ep w(\tfrac x\ep, \omega) \to0$, that is, $w$ must be strictly sublinear at infinity:
\begin{equation} \label{sslatinf}
|y|^{-1} w(y,\omega) \rightarrow 0 \quad \mbox{as} \ |y| \to \infty \quad \mbox{a.s. in} \ \omega.
\end{equation}

The PDE approach to homogenization lies in reversing this analysis: it is hoped that by studying the \emph{macroscopic problem} \eqref{cell}, one can justify the convergence of $u^\ep$ to the solution $u$ of the effective equation \eqref{effeq}. A solution $w$ of \eqref{cell} satisfying \eqref{sslatinf} is typically called a \emph{corrector}, since it permits one to remove the influence of microscopic oscillations in the analysis of \eqref{stat}, as predicted by the asymptotic expansion \eqref{asyexp}. The effective Hamiltonian $\overline H$ is thereby identified via a compatibility condition, namely, the solvability of \eqref{cell} subject to condition \eqref{sslatinf}.

\medskip

In the periodic setting, \eqref{cell} is typically referred to as the \emph{cell problem} since it suffices to solve the equation on the unit cube (a single cell) with periodic boundary conditions. In this case, the condition \eqref{sslatinf} is obviously redundant since periodic functions are bounded.

\medskip

Solving the macroscopic problem requires considering an approximate problem, which typically is
\begin{equation} \label{aux}
\delta v^\delta - \tr\!\left( A(y,\omega) D^2_yv^\delta \right) + H(p+D_yv^\delta,  y, \omega) = 0.
\end{equation}
Here $\delta > 0$ and $p\in \Rd$ are fixed. The auxiliary macroscopic problem \eqref{aux} is not merely an \emph{ad hoc} approximation to \eqref{cell}. Indeed, rescaling the latter by setting $u_\ep(y) : = \ep^{-1} u^\ep(\ep y) - p\cdot y$ and substituting $y=x/\ep$ for the fast variable in \eqref{cell}, one obtains \eqref{stat} with $\delta=\ep$.

In the case that $V$ is bounded, the fact that \eqref{aux} is \emph{proper} (i.e., strictly increasing in its dependence on $v^\delta$) permits one to obtain, from standard viscosity solution theory, a unique bounded solution $v^\delta\in \BUC(\Rd)$ satisfying
\begin{equation} \label{locestout}
 \| \delta v^\delta \|_{L^\infty(\Rd)} + \| Dv^\delta \|_{L^\infty(\Rd)} \leq C.
\end{equation}
The stationarity of $v^\delta$ is immediate from uniqueness. We then define $w^\delta (y,\omega) : = v^\delta(y,\omega) - v^\delta(0,\omega)$ and attempt to pass to the limits $-\delta v^\delta(0,\omega) \rightarrow \overline H$ and $w^\delta(y,\omega) \rightarrow w(y,\omega)$ as $\delta \to 0$, in the hope of obtaining \eqref{cell} from \eqref{aux}. The uniqueness of $\overline H$ then follows from an application of the comparison principle.

\medskip

A periodic environment possesses sufficient compactness to make this argument rigorous, and, as previously mentioned,  \eqref{sslatinf} comes for free since the limit function $w$ is periodic and thus bounded.

To avoid messy measurability issues, in the random setting we must pass to limits in the variables $(y,\omega)$ together. Unfortunately we possess insufficient compactness in $\omega$ to accomplish this without further ado. What is more, unlike the periodic case, we cannot obtain \eqref{sslatinf} from \eqref{locestout}. These are not merely technical issues. In fact, Lions and Souganidis \cite{LS1} have shown that correctors do not exist in the general stationary ergodic case.

\medskip

Hence in the random setting we must concede our attempt to solve the macroscopic problem \eqref{cell} and instead return to the auxiliary problem \eqref{aux}. To obtain the homogenization result, it turns out to be sufficient to find a (necessarily unique) deterministic constant $\overline H$ such that, almost surely, the sequence $\delta v^\delta$ converges to $-\overline H$ uniformly in balls of radius $\sim1/\delta$ as $\delta \to 0$. In Proposition~\ref{bigstep} we show that, for every $r> 0$,
\begin{equation} \label{cplim}
\limsup_{\delta \to 0} \sup_{y\in B_{r/\delta}} \left| \delta v^\delta(y,\omega) + \overline H \right| = 0 \quad \mbox{a.s. in} \ \omega.
\end{equation}

An important idea in the proof of \eqref{cplim}, recently introduced in \cite{LS3}, is the observation that to homogenize it is nearly enough to construct a \emph{subcorrector} $w$ of \eqref{inserhr}, i.e., a subsolution of \eqref{inserhr} which is strictly sublinear at infinity. The subcorrector can be constructed by passing to weak limits in \eqref{aux} along a subsequence $\delta_j \to 0$ and using the convexity of $H$. It then follows from the ergodic theorem and the stationarity of the gradients $Dv^\delta$ and their weak limit that $w$ satisfies \eqref{sslatinf}. The comparison principle permits us to compare $w$ with the full sequence $v^\delta$. Combined with some elementary measure theory this yields that
\begin{equation} \label{inprob}
\E \big| \delta v^\delta(0,\cdot) + \overline H \big| \rightarrow 0 \quad \mbox{as} \ \delta \to 0.
\end{equation}
Of course, this yields almost sure convergence along some subsequence-- and eventually homogenization almost surely-- but only along this particular subsequence. 

Obtaining \eqref{cplim} from \eqref{inprob} along the full sequence $\delta \to 0$, without relying on explicit formulae, is nontrivial and requires some additional estimates and, more importantly, some new ideas.

We begin with the estimates. The first controls the oscillation of $\delta v^\delta$ in balls of radius $\sim1/\delta$. We show that there exists a deterministic constant $C> 0$, depending on $|p|$, such that, for each $y\in \Rd$ and $r > 0$,
\begin{equation} \label{preosc}
\limsup_{\delta \to 0} \osc_{B(y/\delta,r/\delta)} \delta v^\delta(\cdot,\omega) \leq Cr \quad \mbox{a.s. in} \ \omega.
\end{equation}
The second controls the size of $\delta v^\delta$ on balls of radius $\sim 1/\delta$, so that for each $y\in \Rd$,
\begin{equation}\label{uppbndosc1}
\sup_{R>0}\limsup_{\delta \to 0} \osc_{B(y/\delta,R/\delta)} \delta v^\delta(\cdot,\omega) \leq C  \quad \mbox{a.s. in} \ \omega.
\end{equation}
The estimate \eqref{preosc} is useful for $r> 0$ small, while \eqref{uppbndosc1} is typically applied for $R>0$ very large. In the case of bounded $V$, we have global Lipschitz estimates on $v^\delta$ and $L^\infty$-bounds on the $\delta v^\delta$. Therefore \eqref{preosc} and \eqref{uppbndosc1} are immediate. In the unbounded setting, as we will see, proving \eqref{preosc} and \eqref{uppbndosc1} is a more delicate matter. 

The next step to obtain \eqref{cplim} is to prove that
\begin{equation} \label{cplim0}
\delta v^\delta(0,\omega) \rightarrow -\overline H \quad \mbox{a.s. in} \ \omega.
\end{equation}
Once this is done, and we describe how it is proved below, we can get the convergence in balls of radius $\sim 1/\delta$ instead of just at the origin by relying again on \eqref{preosc}, a second application of the ergodic theorem and some elementary measure theory.

\medskip

Proving \eqref{cplim0} requires studying the behavior, as $\ep \to 0$, of the solutions $m_\mu^\ep$ of the metric problem \eqref{meteqn}. Incidentally, that $\mu > \overline H(p)$ is necessary and sufficient for the well-posedness of \eqref{meteqn} is itself another new result in the theory of viscosity solutions. As mentioned in the introduction, we use the subadditive ergodic theorem to conclude, after some work, that the functions $m^\ep_\mu$ converge a.s. in $\omega$ and locally uniformly in $\Rd$ to solutions $\overline m_\mu$ of \eqref{meteqnbar}, for every $\mu > \overline H(p)$. With this in place, we can prove \eqref{cplim0}, using a new \emph{reverse} perturbed test function argument: if \eqref{cplim0} does not hold along a subsequence, then the convergence of the $m^\ep_\mu$ must fail for every $p \in \Rd$ for which $\overline H(p) = \mu$. This argument works for all $p$ for which $H(p) > \min \overline H$, since we need $\mu > \overline H(q)$ for some $q\in \Rd$ for the existence of the functions $m^\ep_\mu$. To conclude we need a separate argument for the case $\overline H(p) = \min \overline H$. For this we prove a new characterization of $\min \overline H$ which involves constructing subsolutions of \eqref{aux} which are permitted to grow linearly instead of strictly sublinearly at infinity.

\medskip

Several additional difficulties arise in unbounded environments. Their resolution has new and interesting implications for the theory of viscosity solutions. Firstly, obtaining stationary solutions $v^\delta$ of the auxiliary problem \eqref{aux} as well as the ``metric problem'' is nontrivial, as we see later. In particular, we cannot expect \eqref{aux} to have bounded solutions, and we must also prove a comparison principle which as far as we know is new in the context of viscosity solutions. Secondly, and this is a more serious problem, it is necessary, as already explained earlier, to have an independent of $\delta$ control over the modulus of continuity of $v^\delta$. We show in Lemma~\ref{impest} that
the mixing hypothesis on $V$ yields control of the oscillation of $\delta v^\delta$ on balls of radius $\sim 1/\delta$.
Finally, the absence of uniform Lipschitz estimates on $v^\delta$ causes an additional difficulty in the proof of \eqref{inprob}, using the comparison principle. This is overcome by relying again on the convexity of $H$.

\section{The auxiliary macroscopic problem}
\label{AMP}

We study here the auxiliary macroscopic problem \eqref{aux} setting the stage for the definition of $\overline H$ in the next section. The main goal is to show that \eqref{aux} has a unique (and therefore stationary) solution $v^\delta$, and that the function $x\mapsto \delta v^\delta(x/\delta)$ is uniformly Lipschitz on scales of order $O(1)$ as $\delta \to 0$. If $V$ is bounded, this is straightforward: a unique bounded solution $v^\delta$ exists for each $\delta > 0$, and we can easily obtain global Lipschitz estimates which are uniform in $\delta$; see \cite{LS2}.

The unbounded setting presents difficulties requiring us to utilize hypotheses \eqref{alphabeta}, \eqref{Vmoment} and \eqref{Vmix} on the potential $V$. Firstly, we cannot expect solutions $v^\delta$ of \eqref{aux} to be bounded from above. This complicates the well-posedness of \eqref{aux}, and, in particular, leads to difficulties with uniqueness, which we handle by using the convexity of $H$.

Secondly, the unboundedness of $V$ means that global Lipschitz estimates on $v^\delta$ and $L^\infty$ bounds on $\delta v^\delta$ cannot hold, and therefore obtaining \eqref{preosc} and \eqref{uppbndosc1} is a nontrivial matter. To deal with this difficulty, we use the mixing condition \eqref{Vmix} and some probability. We proceed by obtaining a local estimate on $|Dv^\delta|$ in terms of $V$ using Bernstein's method. By Morrey's inequality, this reduces the issue to controlling the average of a power of $V$ on large balls. The latter is precisely what a mixing condition allows us to estimate.

Finally, since we may only make countable intersects of subsets of $\Omega$, the proof of the homogenization theorem requires an estimate for the dependence of $v^\delta$ on the parameter $p$ in \eqref{aux}. This is dealt with in Lemma~\ref{CDE-noxd}, a by-product of which is the continuity of effective Hamiltonian $\overline H$, as we will see in Section~\ref{EH}. In the case of general $x$-dependent problems like \eqref{HJx}, this issue becomes much more complicated to resolve in the unbounded environment and requires a strengthening of the hypothesis on $V$; in particular we must take $\alpha$ and $\beta$ to be much larger than in \eqref{alphabeta}.

\subsection{The well-posedness of the auxiliary macroscopic problem} \label{AMP-wp}

We prove that, a.s. in $\omega$, \eqref{aux} has a unique bounded from below solution on $\Rd$. The first issue is to obtain strictly sublinear decay on subsolutions of \eqref{aux}.

The following lemma provides a local upper bound (depending on $\omega$) for subsolutions of \eqref{aux}. The proof is based on  a barrier construction, following \cite{LL,LS2}.

\begin{lem} \label{locest}
Fix $\delta > 0$ and $(p,\omega) \in \Rd\times \Omega$. There exists a deterministic constant $C> 0$, independent of $\delta$, such that if $v\in \USC(\bar B_1)$ is a subsolution of \eqref{aux} in $B_1$, then
\begin{equation} \label{locesteq1}
\sup_{y\in B_{1/2}} \delta v(y) \leq \sup_{y\in B_1} V(y,\omega) + C(1+\delta).
\end{equation}
\end{lem}
\begin{proof}
We construct a simple barrier. To the extent that we rely on \eqref{Hcoer}, we may assume that $\gamma \leq 2$. Indeed, if \eqref{Hcoer} holds for $\gamma > 2$, then, for every $(p,y,\omega)$,
\begin{equation*}
H(p,y,\omega) \geq c_0|p|^2 - V(y,\omega) - (c_1+c_0).
\end{equation*}
With $a,b\geq 0$ to be selected below and $\eta := \frac{2-\gamma}{\gamma-1}$, consider the smooth function $w \in C^\infty(B_1)$ defined by
\begin{equation} \label{barr}
w(y) : = \begin{cases}
a + b(1-|y|^2 )^{-\eta} & \mbox{if} \ \gamma < 2, \\
a - b\log(1-|y|^2) & \mbox{if} \ \gamma=2.
\end{cases}
\end{equation}
It follows that
\begin{equation*}
|Dw(y)|^\gamma \geq c b^\gamma |y|^\gamma \!\left(1-|y|^2\right)^{-(\eta+2)},
\end{equation*}
and
\begin{equation*}
\tr\!\left( A(y,\omega) D^2w \right) \leq Cb \!\left( 1 - |y|^2\right)^{-(\eta+2)} \leq Cb \left(  |y|^\gamma \!\left( 1 - |y|^2\right)^{-(\eta+2)} + C\right).
\end{equation*}
Inserting $w$ into the left side of \eqref{aux} yields
\begin{multline*}
\delta w - \tr\!\left( A(y,\omega) D^2w\right) + H(p+Dw, y, \omega)  \\
\geq (a \delta -C(b+1) - V(y,\omega)) + (cb^\gamma - Cb)  |y|^\gamma \!\left( 1 - |y|^2\right)^{-(\eta+2)} \geq 0 \quad \mbox{in} \ B_1,
\end{multline*}
provided that $b> 0$ is chosen sufficiently large in terms of the constants in \eqref{Alip} and \eqref{Hcoer}, and that $a$ is taken to be the random variable
\begin{equation*}
a(\omega): =  \frac1\delta\! \left( \sup_{y\in B_1} V(y,\omega) + C(b+1) \right).
\end{equation*}
Since $w$ is smooth, it follows from the definition of viscosity subsolution that $v-w$ cannot have a local maximum at any point in the set $\{ v> w\}$. Since $w(y) \to +\infty$ as $|y|\to 1$, we deduce that $\{ v > w \}$ is empty and, therefore, $v \leq w$ in $B_{1}$. The bound \eqref{locesteq1} now follows.
\end{proof}

The hypothesis $\alpha > \d$ is needed to show that $V(\cdot,\omega)$ is strictly sublinear at infinity, a.s. in $\omega$. In light of Lemma~\ref{locest}, this ensures that any subsolution of \eqref{aux} is bounded from above, a.s. in $\omega$, by a function growing strictly sublinearly at infinity. The latter fact is needed to obtain a comparison principle for \eqref{aux}.

\begin{lem} \label{potests}
There exist a set $\Omega_1\subseteq \Omega$ of full probability and constants $\sigma < 1$ and $C_1> 0$, such that, for all $\omega \in \Omega_1$,
\begin{equation} \label{sublin}
\limsup_{R \to \infty} R^{-\sigma} \sup_{y\in B_R} V(y,\omega) = 0.
\end{equation}
\end{lem}
\begin{proof}
Using the stationarity of $V$ and \eqref{Vmoment} and by covering $B_R$ with at most $C R^\d$ balls of radius 1, we see that, for any $\mu > 0$,
\begin{equation} \label{crudecov}
\Prob\left[ \sup_{y\in B_R} V(y,\cdot) \geq \mu \right] \leq C R^\d \mu^{-\alpha}.
\end{equation}
Taking $R=2^k$ and $\mu = R^\sigma\ep$ for fixed $\ep > 0$ and $\d/\alpha < \sigma < 1$, we find
\begin{equation} \label{crudity}
\Prob\left[ \sup_{y\in B_{2^k}} V(y,\cdot) \geq 2^{\sigma k}\ep \right] \leq C_\ep 2^{(\d-\sigma\alpha) k}.
\end{equation}
Since $d-\sigma\alpha < 0$, we obtain
\begin{equation*}
\sum_{k=1}^\infty \Prob\left[  \sup_{y\in B_{2^k}} V(y,\cdot) \geq 2^{\sigma k}\ep \right] \leq C_\ep \sum_{k=1}^\infty 2^{(\d-\sigma\alpha) k} < \infty.
\end{equation*}
Applying the Borel-Cantelli lemma, we get
\begin{equation*}
\limsup_{R\to \infty} R^{-\sigma} \sup_{y\in B_{R}} V(y,\omega) \leq 2^\sigma \limsup_{k\to \infty} \, 2^{-\sigma k} \sup_{y\in B_{2^k}} V(y,\omega)\leq 2^\sigma \ep \quad \mbox{a.s. in} \ \omega.
\end{equation*}
Disposing of $\ep >0$ yields \eqref{sublin} for all $\omega$ in a set of full probability.
\end{proof}
The following is an immediate consequence of Lemmata~\ref{locest} and~\ref{potests}.

\begin{cor} \label{globest}
Fix $\delta > 0$, $p\in\Rd$, $\omega\in \Omega_1$, and $\d / \alpha < \sigma < 1$. If $v\in \USC(\Rd)$ is a subsolution of \eqref{aux} in $\Rd$, then
\begin{equation} \label{globesteq1}
\limsup_{|y|\to\infty} |y|^{-\sigma} v(y) \leq 0.
\end{equation}
\end{cor}

\medskip

Due to the unboundedness of $H$, we cannot apply standard comparison results from the theory of viscosity solutions to \eqref{aux}. However, the previous corollary, the convexity of $H$ and the one-sided bound \eqref{Hbound} suffice to prove the following comparison principle.

\begin{lem} \label{comparison}
Fix $\delta > 0$, $p\in \Rd$ and $\omega \in \Omega_1$. Suppose that $u\in \USC(\Rd)$ and $v\in \LSC(\Rd)$ are, respectively, a subsolution and supersolution of \eqref{aux} in $\Rd$. Assume also that
\begin{equation}\label{compsublin}
\liminf_{|y| \to \infty} \frac{v(y)}{|y|} \geq 0.
\end{equation}
Then $u \leq v$ in $\Rd$.
\end{lem}
\begin{proof}
It follows from Corollary~\ref{globest} that, for all $\omega \in \Omega_1$,
\begin{equation*}
\limsup_{|y| \to \infty} \frac{u(y)}{|y|} \leq 0.
\end{equation*}
Since $(p,\omega)$ play no further role in the argument, we omit them for the rest of the proof. Define the auxiliary function
\begin{equation} \label{linearphi}
\varphi(y): = -\left( 1 + |y|^2 \right)^{\frac12}.
\end{equation}
it is immediate that  $|D\varphi| + |D^2\varphi| \leq C$, hence, using \eqref{Alip} and \eqref{Hbound}, we have
\begin{equation} \label{phieqbnd}
\left| \tr( A(y,\omega) D^2\varphi) \right| + H(p+D\varphi,y,\omega) \leq C.
\end{equation}
Fix $\ep>0$ and define the function
\begin{equation*}
\hat u_\ep (y): = (1-\ep) u(y) + \ep (\varphi(y) - k),
\end{equation*}
where $k> 0$ is taken sufficiently large (depending on $\delta$) that $\varphi-k$ is a subsolution of \eqref{aux} in $\Rd$. Formally, using the convexity of $H$, we see that the function $\hat u_\ep$ is a subsolution of \eqref{aux}. This is made rigorous in the viscosity sense by appealing to Lemma~\ref{convtrick}, or by a more direct argument using that $\varphi-k$ is smooth. Owing to \eqref{globesteq1}, \eqref{compsublin}, and the definition of $\varphi$, we have
\begin{equation*}
\liminf_{|y| \to \infty} \frac{v(y) - \hat u_\ep(y)}{|y|} \geq \ep,
\end{equation*}
and therefore we may apply the standard comparison principle (see \cite{CIL}), yielding
\begin{equation*}
\hat u_\ep \leq v \quad \mbox{in} \ \Rd.
\end{equation*}
We obtain the result upon sending $\ep \to 0$.
\end{proof}

We next demonstrate the well-posedness of the auxiliary macroscopic problem \eqref{aux}, following the general Perron method outlined, for example, in \cite{CIL}.

\begin{prop} \label{auxsolve}
For each fixed $p\in \Rd$ and $\omega \in \Omega_1$, there exists a unique solution $v^\delta = v^\delta(\cdot,\omega; p) \in C(\Rd)$ of \eqref{aux}, which is stationary and such that, for some $C> 0$,
\begin{equation}\label{delvestbel}
\delta v^\delta(\cdot,\omega;p) \geq -C\left(1+|p|^\gamma \right).
\end{equation}
\end{prop}
\begin{proof}
Observe that by \eqref{Hbound}, the constant function $-C\delta^{-1}(1+|p|)^\gamma$ is a subsolution of \eqref{aux}. It is easy to check using \eqref{sublin} that, for sufficiently large $k$ depending on $\delta$, the function $(1+|y|^2)^{1/2} + k$ is a supersolution of \eqref{aux}. Define
\begin{equation*}
v^\delta(y,\omega) : = \sup\left\{ w(y) : w\in \USC(\Rd) \ \mbox{is a subsolution of} \ \eqref{aux} \ \mbox{in} \ \Rd \right\}.
\end{equation*}
Since $\omega\in \Omega_1$, Lemma~\ref{comparison} yields that $v^\delta \leq (1+|y|^2)^{1/2} + k$ in $\Rd$. Thus $v^\delta(y,\omega)$ is well-defined and finite. It follows (see \cite[Lemma 4.2]{CIL}) that $(v^\delta)^*$ is a viscosity subsolution of \eqref{aux}, where $(v^\delta)^*$ denotes the upper semicontinuous envelope of $v^\delta$. Since $\omega \in \Omega_1$, we deduce that $(v^\delta)^*(y)$
satisfies \eqref{compsublin}.
\begin{equation*}
\limsup_{|y| \to \infty} \frac{(v^\delta)^*(y)}{|y|} \leq 0.
\end{equation*}

If the  lower semicontinuous envelope $(v^\delta)_*$ of $v^\delta$ failed to be a supersolution of \eqref{aux}, then this would violate the definition of $v^\delta$, see \cite[Lemma 4.4]{CIL}.
Clearly \eqref{delvestbel} holds by definition, and therefore we have
Applying Lemma~\ref{comparison}, we conclude that $(v^\delta)^* \leq (v^\delta)_*$, and therefore $v^\delta = (v^\delta)^* = (v^\delta)_*$ and so $v^\delta \in C(\Rd)$ is a solution of \eqref{aux}. Uniqueness is immediate from Lemma~\ref{comparison}, and stationarity follows from uniqueness.
\end{proof}

We conclude this subsection with a continuous dependence estimate, asserting that, a.s. in $\omega$, if $p_1$ and $p_2$ are close, then $\delta v^\delta(\cdot,\omega;p_1)$ and $\delta v^\delta(\cdot,\omega;p_2)$ are close in an appropriate sense. Once we have homogenization, this is equivalent to showing that $\overline H$ is continuous. We address this point now rather than later due to technical difficulties we encounter in the homogenization proof. In particular, we must obtain a single subset of full probability on which \eqref{cplim} holds for all $p\in\Rd$. To accomplish this, we first obtain \eqref{cplim}, a.s. in $\omega$, for each rational $p$ and then intersect the respective subsets of~$\Omega$. This yields a subset of $\Omega$ of full probability on which the limit \eqref{cplim} holds for all rational $p$. To argue that, in fact, \eqref{cplim} holds for all $\omega$ in this subset and all $p \in \Rd$ requires such a continuous dependence estimate. For exactly the same reason, we also need the following result in the next subsection, where we obtain a single set of full probability on which the estimate \eqref{preosc} holds for all $p$.

\medskip

The continuous dependence estimate is Lemma~\ref{CDE-noxd}, below. It is based on the following preliminary lemma, which will also be useful to us later.

\begin{lem} \label{movep}
Fix $\lambda > 0$ and $p\in \Rd$, and define $w^\delta(y,\omega): = \lambda v^\delta(y,\omega;p)$. If $\lambda < 1$, then $w^\delta$ satisfies, for any $q\in \Rd$,
\begin{equation}\label{movepdn}
\delta w^\delta -\tr\!\left( A(y,\omega)D^2w^\delta\right) + H( q + Dw^\delta ,y,\omega) \leq (1-\lambda) H\left( \frac{q-\lambda p}{1-\lambda},y,\omega\right) \quad \mbox{in} \ \Rd.
\end{equation}
Likewise, if $\lambda > 1$, then for any $q\in \Rd$,
\begin{equation}\label{movepup}
\delta w^\delta -\tr\!\left( A(y,\omega)D^2w^\delta\right) + H( q + Dw^\delta ,y,\omega) \geq (1-\lambda) H\left( \frac{q-\lambda p}{1-\lambda},y,\omega\right) \quad \mbox{in} \ \Rd.
\end{equation}
\end{lem}
\begin{proof}
Writing $w^\delta$ in the form
\begin{equation*}
w^\delta(y) = \lambda v^\delta(y,\omega;p) = \lambda \big( (q-p)\cdot y + v^\delta(y,\omega;p) \big) + (1-\lambda)\left( \frac{\lambda(q-p)}{1-\lambda} \cdot y \right)
\end{equation*}
and using the convexity of $H$, we find that, formally,
\begin{equation*}
H(q+Dw^\delta(y),y,\omega) \leq \lambda H(p+Dv^\delta(y,\omega;p) ) + (1-\lambda) H\left( \frac{q-\lambda p}{1-\lambda},y,\omega\right).
\end{equation*}
Therefore, formally we have
\begin{multline*}
\delta w^\delta -\tr\!\left( A(y,\omega)D^2w^\delta\right) + H( q + Dw^\delta ,y,\omega) -(1-\lambda) H\left( \frac{q-\lambda p}{1-\lambda},y,\omega\right) \\
 \leq \lambda \left( \delta v^\delta(y,\omega;p) - \tr\!\left( A(y,\omega)D^2v^\delta(y,\omega;p) \right) + H( p + Dv^\delta(y,\omega;p) ,y,\omega) \right) = 0.
\end{multline*}
This inequality is easy to confirm in the viscosity sense by performing an analogous calculation with smooth test functions. We have proved \eqref{movepdn}.

The proof of \eqref{movepup} is similar. Expressing $v^\delta(\cdot; p)$ in terms of $w^\delta$ as
\begin{equation*}
v^\delta(y,\omega;p) = \lambda^{-1} \big( (p-q)\cdot y + w^\delta(y) \big) + (1-\lambda^{-1}) \left( \frac{\lambda(q-p)}{1-\lambda} \cdot y\right),
\end{equation*}
we use again the convexity of $H$ to find that, formally,
\begin{equation*}
H(p+Dv^\delta(y,\omega;p) ) \leq \lambda^{-1} H(q+Dw^\delta(y)) + (1-\lambda^{-1}) H\left( \frac{q-\lambda p}{1-\lambda},y,\omega\right).
\end{equation*}
From this we formally obtain \eqref{movepup}. The derivation is once again made rigorous with smooth test functions. 
\end{proof}

\begin{lem} \label{CDE-noxd}
There exists $C > 0$ such that, for each $\delta > 0$, $p_1,p_2\in \Rd$ and $\omega \in \Omega_1$,
\begin{equation} \label{CDE-noxdeq}
\sup_{y\in \Rd} \left| \delta v^\delta(y,\omega; p_1) - \delta v^\delta(y,\omega; p_2) \right| \leq C (1+|p_1| + |p_2|)^{\gamma-1}|p_1-p_2|.
\end{equation}
\end{lem}
\begin{proof}
Fix $\delta> 0$, $\omega \in \Omega_1$, and, for $i=1,2$, write $v_i^\delta(y,\omega) := v^\delta(y,\omega;p_i)$. For $0 < \lambda < 1$ to be selected below define $w^\delta(y) : = \lambda v_2^\delta(y,\omega)$. According to Lemma~\ref{movep}, $w^\delta$ satisfies the inequality
\begin{equation*}
\delta w^\delta -\tr\!\left( A(y,\omega)D^2w^\delta\right) + H( p_1 + Dw^\delta ,y,\omega) \leq (1-\lambda) H\left( \frac{p_1-\lambda p_2}{1-\lambda},y,\omega\right) \quad \mbox{in} \ \Rd.
\end{equation*}
Set $\lambda : = 1 - (1+|p_1|+|p_2|)^{-1} |p_1-p_2|$. It follows that
\begin{equation*}
\frac{|p_1-\lambda p_2|}{1-\lambda} = \frac{\big| (1+|p_1|+|p_2|) (p_1 -p_2) + |p_1-p_2|p_2 \big| }{|p_1-p_2|} \leq 1+|p_1|+2|p_2|.
\end{equation*}
Thus by \eqref{Hbound},
\begin{align}
(1-\lambda) H\left( \frac{p_1-\lambda p_2}{1-\lambda},y,\omega\right) & \leq C \frac{|p_1-p_2|}{1+|p_1|+|p_2|} \left(1+|p_1|+|p_2| \right)^{\gamma} \nonumber \\ & = C  \left(1+|p_1|+|p_2| \right)^{\gamma-1} |p_1-p_2|. \label{convHbar2}
\end{align}
Therefore, $\widetilde w^\delta(y) : = w^\delta(y) - \delta^{-1} C  \left(1+|p_1|+|p_2| \right)^{\gamma-1} |p_1-p_2|$ is a subsolution of
\begin{equation*}
\delta \widetilde w^\delta - \tr\!\left( A(y,\omega)D^2 \widetilde  w^\delta\right) + H( p_1 + D \widetilde w^\delta ,y,\omega) \leq 0 \quad \mbox{in} \ \Rd.
\end{equation*}
Applying Lemma~\ref{comparison}, we obtain
\begin{equation*}
\widetilde w^\delta \leq v_1^\delta \quad \mbox{in} \ \Rd.
\end{equation*}
Rearranging some terms and using the definition of $\lambda$ and \eqref{delvestbel} yields that, in $\Rd$,
\begin{align*}
v^\delta_2 - v^\delta_1 & = (\lambda - 1) v^\delta_2 + \widetilde w^\delta - v^\delta_1 + \delta^{-1} C\left(1+|p_1|+|p_2| \right)^{\gamma-1} |p_1-p_2| \\
& \leq \delta^{-1} C (1+|p_2|^\gamma) (1-\lambda) + \delta^{-1} C\left(1+|p_1|+|p_2| \right)^{\gamma-1} |p_1-p_2| \\
& \leq \delta^{-1} C\left(1+|p_1|+|p_2| \right)^{\gamma-1} |p_1-p_2|.
\end{align*}
Now we repeat the argument reversing the roles of $p_1$ and $p_2$ to obtain \eqref{CDE-noxdeq}.
\end{proof}

\subsection{Estimates on the oscillation of $\delta v^\delta$ in balls of radius $\sim 1/\delta$.}

We obtain important estimates controlling the oscillation of $\delta v^\delta$ in balls of radius $\sim1/\delta$, which plays a critical role in the proof of \eqref{cplim}. If $V$ is bounded uniformly in $\omega$, it is easy to obtain estimates which are independent of $\delta$. Indeed, Lemma~\ref{locest} and Proposition~\ref{auxsolve} provide an $L^\infty$ bound for $\delta v^\delta$ while Lemma~\ref{locestgrad} below provides uniform, global Lipschitz estimates for $v^\delta$, and there is nothing more to show. In the unbounded setting, the situation is subtle. The mixing hypothesis together with Morrey's inequality and the ergodic theorem is needed to obtain the required estimate.

\medskip

The idea is as follows. First, we use Bernstein's method to get a local Lipschitz estimate on $v^\delta$ in terms of the nearby behavior of $V$. Next, we use the mixing hypothesis to control the average of an appropriate power $V$ over large balls, thereby providing us with some control over the average of $|Dv^\delta|^q$ on large balls, for some $q> \d$. We then apply Morrey's inequality, which yields an estimate on the oscillation of $\delta v^\delta$ on balls of radius $\sim 1/\delta$, centered at the origin. A supplementary argument using Egoroff's theorem combined with the ergodic theorem upgrades the latter estimate to the one we need, which is contained in Lemma~\ref{impest}, below.

\medskip

We begin with the local Lipschitz estimate on $v^\delta$.

\begin{lem} \label{locestgrad}
Fix $\delta > 0$ and $(p,\omega) \in \Rd\times \Omega$. There exists, an independent of $\delta$, $C=C(|p|)> 0$ such that any solution $v \in C(\bar B_1)$ of \eqref{aux} in $B_1$ is Lipschitz on $B_{1/2}$, and
\begin{equation} \label{locesteq2}
\esssup_{y\in B_{1/2}} |Dv(y)|^\gamma \leq C\left( 1 + \sup_{y\in B_1} V(y,\omega)  \right).
\end{equation}
\end{lem}
\begin{proof}
The estimate follows by the Bernstein method. By performing a routine regularization, smoothing the coefficients and adding, if necessary, a small viscosity term, we may assume that $v$ is smooth. Next we adapt here the arguments of \cite{LL,LS2}.

Let $0 < \theta <1$ be chosen below and select a cutoff function $\varphi\in C^\infty(B_1)$ satisfying
\begin{equation}
0 \leq \varphi \leq 1, \ \ \varphi \equiv 1 \ \mbox{on} \ B_{\frac12}, \ \ \varphi \equiv 0 \ \mbox{in} \ B_1 \setminus B_{\frac34}, \ \ |D\varphi|^2 \leq C\varphi^{1+\theta} \ \ \mbox{and} \ \ \left| D^2 \varphi \right| \leq C\varphi^{\theta}. \label{varphiest}
\end{equation}
For example, we may take $\varphi = \xi^{2/(1-\theta)}$ where $\xi \in C^\infty(B_1)$ is any smooth function satisfying the first three condition of \eqref{varphiest}.

Define $z : = |Dv|^2$ and $\xi : = \varphi |Dv|^2 = \varphi z$, and compute
\begin{equation} \label{easyderv}
D\xi = z D\varphi + 2\varphi D^2v Dv,  \quad D^2\xi = z D^2\varphi + 2 D\varphi \! \otimes\! (D^2v Dv) + \varphi (D^3v Dv + D^2v D^2v).
\end{equation}
Differentiating \eqref{aux} with respect to $y$ and multiplying by $\varphi Dv$ yields
\begin{multline} \label{diffauxle}
\delta \xi - \varphi Dv \cdot \tr \!\left( D_yA(y,\omega) D^2v + A(y,\omega) D_yD^2v \right) + \varphi D_pH(p+Dv,y,\omega) \cdot D^2vDv \\ + \varphi Dv\cdot D_yH(p+Dv,y,\omega) = 0.
\end{multline}
Combining \eqref{easyderv} and \eqref{diffauxle} and performing some computation, we find that, at any point $y_0$ at which $\xi$ achieves a positive local maximum,
\begin{multline} \label{diffauxhell}
2\delta \xi - 2\varphi \tr \!\left( D_yA D^2v\right) \cdot Dv + \tr\!\left( A (zD^2\varphi + 2 \varphi^{-1} z D\varphi\otimes D\varphi + 2\varphi D^2vD^2v) \right) \\
- z D_pH(p+Dv,y_0) \cdot D\varphi +2 \varphi D_yH(p+Dv,y_0) \cdot Dv \leq 0.
\end{multline}
Writing $M := D^2v(y_0)$ and $q: =Dv(y_0)$ and using \eqref{Hlip}, we obtain
\begin{multline} \label{diffauxhell2}
\varphi\tr(AM^2) \leq \varphi |q| | \tr(D_yAM)|  + C|A| |q|^2 |D^2\varphi| + C |A| |D\varphi|^2 \varphi^{-1} |q|^2  \\ + C|q|^2 |p+q|^{\nu-1} |D\varphi| + C |q| \varphi V(y_0,\omega).
\end{multline}
The Cauchy-Schwarz inequality in the form
\begin{equation} \label{CSineq}
(\tr(AM))^2\leq C |A| \tr(AM^2)
\end{equation}
and \eqref{Alip} imply that
\begin{equation*}
\left( \tr(D_yA M) \right)^2 = \left( 2\tr( D_y\Sigma\Sigma^t M) \right)^2 \leq 4 |D_y\Sigma|^2 \tr( A M^2 ) \leq C \tr(AM^2).
\end{equation*}
This inequality, \eqref{varphiest}, another use of \eqref{CSineq} and some elementary inequalites yield
\begin{equation} \label{diffauxhell3}
\varphi (\tr(AM))^2\leq C \varphi \tr(AM^2) \leq C\varphi^\theta |q|^2 + C|q|^2 |p+q|^{\nu-1} \varphi^{(1+\theta)/2} + C |q| \varphi V(y_0,\omega).
\end{equation}
By squaring the equation \eqref{aux} and using \eqref{diffauxhell3}, we obtain, at $y=y_0$,
\begin{equation*}
\varphi\left( \delta v + H(p+q,y_0) \right)^2 \leq C\varphi^\theta |q|^2 + C|q|^2 |p+q|^{\nu-1} \varphi^{(1+\theta)/2} + C |q| \varphi \sup_{y\in B_{1}} V(y,\omega).
\end{equation*}
The inequality above, \eqref{Hcoer} and \eqref{locesteq1} yield
\begin{equation*}
\varphi |q|^{2\gamma} \leq C\varphi \left( 1 + \sup_{y\in B_1} \V(y,\omega) \right)^2 + C\varphi^\theta |q|^2 + C|q|^{\gamma+1} \varphi^{(1+\theta)/2},
\end{equation*}
for a constant $C> 0$ depending as well on an upper bound for $|p|$. Setting $\theta := 1 / \gamma$, we find
\begin{equation*}
(\varphi^\theta |q|^{2})^\gamma \leq C\varphi \left( 1 + \sup_{y\in B_1} \V(y,\omega) \right)^2 + C\varphi^\theta |q|^2 + C(\varphi^\theta |q|^2)^{(1+\gamma)/2}.
\end{equation*}
Therefore,
\begin{equation*}
\xi^\gamma = (\varphi |q|^2)^\gamma \leq C + C\left( 1 + \sup_{y\in B_1} \V(y,\omega) \right)^2.
\end{equation*}
The bound \eqref{locesteq2} follows.
\end{proof}

From Lemma~\ref{locestgrad} we see immediately that $v^\delta(\cdot,\omega;p) \in W^{1,\infty}_{\mathrm{loc}}(\Rd)$ for each $\omega\in \Omega_1$ and $p \in \Rd$, with $|Dv^\delta|$ controlled locally by $V$ and an upper bound for $|p|$.

\medskip

The purpose of the mixing hypothesis is to prove the following estimate which, in light of Lemma~\ref{locest}, give some control over the average of $|Dv^\delta|^q$ in large balls.

\begin{lem} \label{potestsD}
There exist a set $\Omega_2 \subseteq \Omega_1$ of full probability and constants $q> \d$ and $C_2 > 0$ such that, for every $\omega \in \Omega_2$,
\begin{equation} \label{bigmix}
\limsup_{R\to \infty} \bigg( \fint_{B_R} \sup_{B(y,1)}V(\cdot,\omega)^{q/\gamma} \, dy \bigg)^{1/q} \leq C_2.
\end{equation}
\end{lem}

Before we present the proof of Lemma~\ref{potestsD}, we briefly describe the general idea of how the mixing condition is used to prove \eqref{bigmix}. For convenience, it is better to consider a cube $Q_K$ instead of the ball $B_R$. Here $K$ is large integer and $Q_K$ is the cube of side length $2K$ centered at the origin. The idea is then to subdivide $Q_K$ into smaller cubes, which are obtained by two successive partitions of $Q_K$. The smaller cubes are collected into groups in such a way that within each group, the cubes are sufficiently separated so that we may apply the mixing condition. This provides some decay on the probability that the average of $\sup_{B(y,1)}|V(\cdot,\omega)|^{q/\gamma}$ is large on the union of each group.

\begin{proof}[{Proof of Lemma~\ref{potestsD}}]
Let $K$ and $N$ be positive integers such that $K$ is a multiple of $N$. Partition $Q_K$ into $M : = (K/N)^\d$ subcubes $Q^1,\ldots, Q^M$ of side length $2N$. Partition each of the cubes $Q^i$ into $L:=N^\d$ subcubes $Q^{i1},\ldots, Q^{iL}$ of side length 2, in such a way that $Q^{ij}$ has the same position in $Q^{i}$ as does $Q^{1j}$ relative to $Q^1$. That is, the translation that takes $Q^i$ to $Q^1$ also takes $Q^{ij}$ to $Q^{1j}$ (the partitions are illustrated in Figure~\ref{cubepart}). If $N\geq 5$, this ensures that, for each $1 \leq i,j\leq M$ and $1 \leq k \leq L$, we have
\begin{equation} \label{dAij}
\dist(Q^{ik},Q^{jk}) \geq 2( N - \sqrt2) \geq N+2.
\end{equation}

\begin{figure} \label{cubepart}
\centering

\begin{tikzpicture}[>=stealth]
\pgftransformscale{.50}

\draw[very thick] (6,6) -- (6,-6) -- (-6,-6) -- (-6,6) -- (6,6); 
\draw[thick] (-6,-2) -- (6,-2); 
\draw[thick] (-6,2) -- (6,2);
\draw[thick] (-2,-6) -- (-2,6);
\draw[thick] (2,-6) -- (2,6);
\draw (-2,5) -- (-6,5); 
\draw (-2,4) -- (-6,4);
\draw (-2,3) -- (-6,3);
\draw (-5,2) -- (-5,6);
\draw (-4,2) -- (-4,6);
\draw (-3,2) -- (-3,6);
\draw (2,1) -- (6,1); 
\draw (2,0) -- (6,0);
\draw (2,-1) -- (6,-1);
\draw (5,2) -- (5,-2);
\draw (4,2) -- (4,-2);
\draw (3,2) -- (3,-2);
\fill[lightgray] (-4,5) -- (-3,5) -- (-3,4) -- (-4,4) -- (-4,5); 
\fill[lightgray] (4,1) -- (5,1) -- (5,0) -- (4,0) -- (4,1);
\draw[thick,->]  (3.7,2.4) node[anchor=south]{$Q^{ij}$} -- (4.5,0.5); 
\draw[thick,->]  (0.9,2.8) node[anchor=east]{$Q^{i}$} -- (2.2,1.8);
\draw[thick,->]  (-1.6,4.2) node[anchor=west]{$Q^{1{j}}$} -- (-3.5,4.5); 
\draw[thick,->]  (-0.9,1.2) node[anchor=west]{$Q^{1}$} -- (-2.2,2.2);
\end{tikzpicture}
\caption{The partitioning of $Q_K$ into smaller cubes in the proof of Lemma~\ref{potestsD}}
\end{figure}
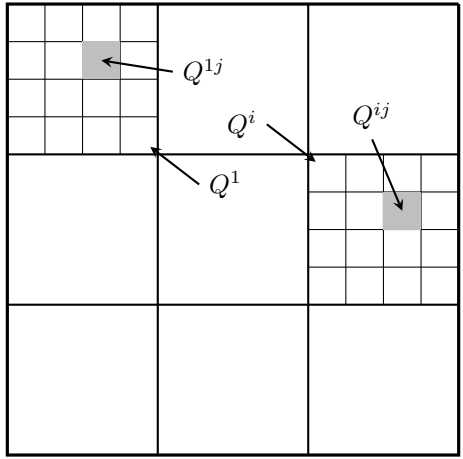

Observe that
\begin{multline} \label{breakdn}
\fint_{Q_K} \sup_{B(y,1)} |V(\cdot,\omega)|^{q/\gamma} \, dy = (2K)^{-\d} \sum_{j=1}^L \sum_{i=1}^M \int_{Q^{ij}} \sup_{B(y,1)} |V(\cdot,\omega)|^{q/\gamma} \, dy \\
\leq (2K)^{-\d} L \max_{1\leq j \leq L}  \left( \sum_{i=1}^M\sup_{\widetilde Q^{ij}} |V(\cdot,\omega)|^{q/\gamma} \right) = 2^{-\d} \max_{1\leq j \leq L} \left( \frac{1}{M} \sum_{i=1}^M \sup_{\widetilde Q^{ij}} |V(\cdot,\omega)|^{q/\gamma} \right),
\end{multline}
where $\widetilde Q^{ij} : = \{ y \in \Rd : \dist(Q^{ij},y) \leq 1\}$ is the set of points within a unit distance of $Q^{ij}$. Note that $\diam(\widetilde Q^{ij}) \leq 2+ 2\sqrt\d$, and by \eqref{dAij},
\begin{equation} \label{disteps}
\dist(\widetilde Q^{ik}, \widetilde Q^{jk} ) \geq N.
\end{equation}
Fix $1 \leq J \leq L$ and define the random variable
\begin{equation*}
g_i(\omega) : =  \sup_{\widetilde Q^{iJ}} \left( |V(\cdot,\omega)|^{q/\gamma} \wedge  K^{q\sigma/\gamma} \right),
\end{equation*}
with $\d/\alpha< \sigma < 1$ to be selected below. Lemma~\ref{mix-decay}, the mixing hypothesis \eqref{mixing} and \eqref{disteps} imply that, for each $1 \leq i,j\leq M$,
\begin{equation*}
\left| \E[g_ig_j] - \E[g_i]\,\E[g_j] \right| \leq CK^{2q\sigma/\gamma} N^{-\beta}.
\end{equation*}
Lemma~\ref{mix-EVest} yields the estimates
\begin{equation*}
\E \bigg( \frac1M \sum_{j=1}^M g_j \bigg)^2 \leq C \left(1 + K^{2q\sigma/\gamma} N^{-\beta} \right)
\end{equation*}
and
\begin{equation*}
\Var \bigg( \frac1M \sum_{j=1}^M g_j \bigg) \leq \frac1M \Var(g_j) + CK^{2q\sigma/\gamma} N^{-\beta} \leq C\left( K^{-\d} N^{\d} + K^{2q\sigma/\gamma} N^{-\beta} \right).
\end{equation*}
According to \eqref{alphabeta}, it is possible to choose $s\in \Q$ such that $0 < s < 1/2$, $q> \d$ and $\d/\alpha < \sigma < 1$ such that $s(\beta-\d) > 2\sigma q /\gamma$. Selecting $N = K^{s}$ (we show below that this is possible) and
\begin{equation*}
\ep : = \min\{ (1-s)\d, s\beta - 2\sigma q / \gamma \}> s\d > 0,
\end{equation*}
we obtain
\begin{gather*}
\E \bigg( \frac1M \sum_{j=1}^M g_j \bigg)^2 \leq C \quad \mbox{and} \quad \Var \bigg( \frac1M \sum_{j=1}^M g_j \bigg) \leq C K^{-\ep}.
\end{gather*}
It follows by Chebyshev's inequality (see Remark~\ref{remcheb}) that, for some $C> 0$,
\begin{equation*}
\Prob \bigg[  \frac1M \sum_{j=1}^M g_j > t \bigg] \leq (t^2 - C)^{-1} K^{-\ep} \leq C K^{-\ep} \quad \mbox{for all} \quad t > 2C.
\end{equation*}
Recalling \eqref{breakdn} and the choice of $N$, we see that, for each $t> 2C$,
\begin{align*}
\lefteqn{ \Prob\bigg[ \bigg(\fint_{Q_K} \sup_{B(y,1)} |V(\cdot,\omega)|^{q/\gamma} \wedge R^{q\sigma/\gamma} \, dy\bigg)^{1/q} > t \bigg]} \qquad \qquad & \\
& \leq \Prob \bigg[ \max_{1\leq j \leq L} \bigg( \frac{1}{M} \sum_{i=1}^M \sup_{\widetilde Q^{ij}} |V(\cdot,\omega)|^{q/\gamma} \wedge R^{q\sigma/\gamma} \bigg) >  Ct^q \bigg] \\
& \leq L  \Prob \left[  \frac1M \sum_{j=1}^M g_j > C t^q \right] \leq C L K^{- \ep}  \leq C K^{s\d-\ep}.
\end{align*}
Writing $h(y):= \sup_{B(y,1)} |V(\cdot,\omega)|^{q/\gamma}$, in view of \eqref{crudity}, we see that for all $t > C$,
\begin{align*}
\Prob\bigg[ \bigg( \fint_{Q_K} h(y)\, dy\bigg)^{1/q} > t \bigg] & \leq \Prob\bigg[ \bigg( \fint_{Q_K} h(y) \wedge R^{q\sigma/\gamma}\, dy \bigg)^{1/q}  > t \bigg] + \Prob\left[ \max_{y\in Q_K} h(y) > K^{q\sigma/\gamma} \right]  \\
& \leq CK^{s\d-\ep} + CK^{\d-\sigma\alpha}.
\end{align*}
Observe that the exponents $sd-\ep$ and $d-\sigma\alpha$ are negative. The above applies to any $K\in\N$ for which $K^s$ is an integer which divides $K$. In particular we can take $K = 2^m$ for any positive integer $m$ such that $ms\in\N$. Writing $s=a/b$ for $a,b \in \N$ and applying the Borel-Cantelli lemma to the estimate above, we obtain
\begin{equation*}
\limsup_{n\to \infty} \bigg( \fint_{Q_{2^{nb}}} h(y) \, dy \bigg)^{1/q} \leq C \quad \mbox{a.s. in} \ \omega.
\end{equation*}
Now \eqref{bigmix} follows, since for large $R > 0$, the positive integers $k := \lfloor \log_{2^b} (R/\sqrt\d) \rfloor $ and $\ell := \lceil \log_{2^b} R \rceil$ satisfy $Q_{2^{kb}} \subseteq B_R \subseteq Q_{2^{\ell b}}$ and $\left|Q_{2^{\ell b}} \right| \leq C \left| B_R \right| \leq C\left| Q_{2^{k b}} \right|$. Thus
\begin{equation*}
\fint_{Q_{2^{k b}}} h(y)\, dy \leq C \fint_{B_R} h(y)\,dy \leq C \fint_{Q_{2^{\ell b}}} h(y)\, dy. \qedhere
\end{equation*}
\end{proof}

In the next lemma, we prove \eqref{preosc}. The inequalities \eqref{locesteq2}, \eqref{bigmix}, and Morrey's inequality provides an estimate on the oscillation of $\delta v^\delta$ in the ball $B_{r/\delta}$. A supplementary argument combining the ergodic theorem and Egoroff's theorem extends this to the balls of the form $B(y/\delta, r/\delta)$.

\begin{lem} \label{impest}
There exist a set $\Omega_3 \subseteq \Omega_2$ of full probability and a constant $C = C(k) > 0$, such that, for each $y\in \Rd$, $r> 0$, $|p|<k$ and $\omega\in\Omega_3$,
\begin{equation} \label{oscest}
\limsup_{\delta \to 0}  \osc_{B(y/\delta,r/\delta)} \delta v^\delta(\cdot,\omega) \leq Cr.
\end{equation}
\end{lem}
\begin{proof}
Fix $p \in \Rd$. It follows from Morrey's inequality, \eqref{locesteq2} and \eqref{bigmix} that, for each $r> 0$ and $\omega \in \Omega_2$,
\begin{multline} \label{oscatz}
\limsup_{\delta \to 0} \sup_{y\in B_{r/\delta}} \left| \delta v^\delta(y,\omega) - \delta v^\delta(0,\omega) \right| \leq C\limsup_{\delta \to 0}\, r  \bigg( \fint_{B_{r/\delta}} |Dv^\delta(y,\omega)|^{q} \, dy \bigg)^{1/q} \\ \leq Cr \limsup_{\delta\to0} \bigg( \fint_{B_{r/\delta}} \sup_{B(y,1)} V(\cdot,\omega)^{q/\gamma} \, dy \bigg)^{1/q} \leq Cr, 
\end{multline}
and \eqref{oscest} holds for $y=0$. Note that  $C$ depends also on an upper bound for $|p|$.

\medskip

To obtain the full \eqref{oscest}, we combine \eqref{oscatz}, Egoroff's theorem and the ergodic theorem. Fix $r, s> 0$. According to Egoroff's theorem, for each $\eta > 0$, there exists $D_\eta \subseteq \Omega_2$ with $\Prob[D_\eta] \geq 1-\eta$ such that, for each $\omega \in D_\eta$ and $\delta=\delta(\eta) > 0$ sufficiently small,
\begin{equation} \label{egor-lip}
\osc_{B(0,(r(1+\eta)/\delta)} \delta v^\delta(\cdot,\omega) \leq 2C(1 +2\eta) r.
\end{equation}
The ergodic theorem yields $E_\eta \subseteq \Omega$ such that $\Prob[E_\eta]=1$ and, for all $\omega \in E_\eta$,
\begin{equation} \label{erglim}
\lim_{R \to \infty} \fint_{B_R} \mathds{1}_{D_\eta}(\tau_z\omega) \, dz = \Prob[D_\eta] \geq 1-\eta.
\end{equation}
Fix $y\in \Rd$ such that $|y| \leq s$. It follows from \eqref{erglim} that, if $\delta=\delta(\eta,s) > 0$ is sufficiently small and $\omega \in E_\eta$, then
\begin{equation} \label{ergod-lip}
| \{ z \in B_{2|y|} : \tau_{z/\delta}\omega \in D_\eta \}| \geq (1-2\eta) |B_{2|y|}|.
\end{equation}
Define
\begin{equation*}
\Omega_{r,s} : = \bigcap_{j=1}^\infty \bigcup_{k=j}^\infty \left( D_{2^{-k}} \cap E_{2^{-k}} \right),
\end{equation*}
and observe that $\Prob[\Omega_{r,s}] = 1$. Fix a positive integer $j$, $\omega \in \Omega_{r,s}$ and set $\ep : = 2^{-j}  > 0$. By making $j$ larger, if necessary, we may assume that $\omega \in D_{\ep} \cap E_\ep$. For all $\delta > 0$ sufficiently small, depending on $\omega$, $k$ and $j$, \eqref{ergod-lip} holds with $\eta = \ep$. This implies that we can find $z\in B_{2|y|}$ such that  $\tau_{z/\delta}\omega \in D_\ep$ and $|z-y| \leq C\ep |y|$, due to the fact that the ball $B(y,C\ep |y|)$ is too large to lie in the complement of $\{ z \in \Rd: \tau_{z/\delta} \omega \in D_\ep \}$. Then $B(y/\delta,r/\delta) \subseteq B(z/\delta, (r+C\ep|y|)/\delta)$. By shrinking $\delta> 0$, again depending on $\omega$, $s$ and $j$, we may assume that \eqref{egor-lip} holds for $\eta = C\ep|y|$. For such $\delta$, we have
\begin{multline*}
\osc_{B(y/\delta,r/\delta)} \delta v^\delta(\cdot, \omega) \leq \osc_{B(z/\delta,r(1+C\ep|y|)/\delta)} \delta v^\delta(\cdot, \omega) \\ = \osc_{B(0,r(1+C\ep |y|)/\delta)} \delta v^\delta(\cdot, \tau_{z/\delta}\omega) \leq 2C(1+2C\ep s) r.
\end{multline*}
Sending first  $\delta \to 0$ and then $j \to \infty$ so that $\ep \to 0$, we deduce that, for every $|y| \leq s$ and $\omega \in \Omega_{r,s}$,
\begin{equation} \label{finliplim}
\limsup_{\delta\to 0} \osc_{B(y/\delta,r/\delta)} \delta v^\delta (\cdot, \omega) \leq 2Cr.
\end{equation}
Now define
\begin{equation*}
\Omega_3: = \bigcap_{r \in \Q_+} \bigcap_{s\in \N} \Omega_{r,s}.
\end{equation*}
Observe that $\Prob[\Omega_3] = 1$ and \eqref{finliplim} holds for every rational $r> 0$, $y\in \Rd$ and $\omega\in \Omega_3$. By the monotonicity of the quantity on the left side of \eqref{finliplim} in $r$, the inequality \eqref{finliplim} must then hold for all $r> 0$, $y\in \Rd$ and $\omega\in \Omega_3$.

So far, we have obtained \eqref{oscest} only for a fixed $p\in \Rd$, that is, $\Omega_3$ depends on $p$. To complete the proof, we replace $\Omega_3$ by the intersection of such sets over all $p\in \Q^\d$. An appeal to Lemma~\ref{CDE-noxd} then completes the proof.
\end{proof}

\begin{remark}
Notice that in the argument above we have proved slightly more than \eqref{oscest}, namely that, for each $p\in \Rd$, $r > 0$ and $\omega \in \Omega_3$,
\begin{equation} \label{super-oscest}
\sup_{R > 0} \limsup_{\delta\to 0} \sup_{y \in B_R} \frac1r\osc_{B(y/\delta, r/ \delta)} \delta v^\delta(\cdot,\omega)  \leq C(|p|).
\end{equation}
We can do even better by observing that
\begin{equation*}
r\mapsto \sup_{y\in B_{R-r}} \frac1r \osc_{B(y/\delta,r/\delta)} \delta v^\delta(\cdot,\omega) \quad \mbox{is decreasing,}
\end{equation*}
from which it follows that, for every $\omega\in \Omega_3$ and $p\in \Rd$,
\begin{equation} \label{uber-oscest}
\sup_{s,R > 0} \limsup_{\delta\to 0} \sup_{y \in B_R} \sup_{s\leq r \leq R}  \frac1r\osc_{B(y/\delta, r/ \delta)} \delta v^\delta(\cdot,\omega)  \leq C(|p|).
\end{equation}
\end{remark}

The estimate \eqref{oscest} will typically be applied for $r>0$ very small. In certain situations we control of the oscillation of $\delta v^\delta$ in $B_{R/\delta}$ for large $R$, but require something better than $CR$. In bounded environments, of course, the quantity $\delta v^\delta$ is (essentially) bounded in $\Rd\times \Omega$. In the next lemma we prove that, roughly speaking, this is also true in our unbounded setting \emph{in balls of radius $\sim1/\delta$ as $\delta \to 0$}. The arguments strongly uses the mixing condition (again) as well as our previous oscillation bound \eqref{uber-oscest}.

\begin{lem} \label{uppbndosc}
There exists a set $\Omega_4 \subseteq\Omega_3$ of full probability and a constant $C=C(k) > 0$, such that, for each $|p| < k$ and $\omega\in \Omega_4$,
\begin{equation} \label{oscest-large}
\sup_{R>0} \limsup_{\delta \to 0} \osc_{B_{R/\delta}} \delta v^\delta(\cdot, \omega) \leq C.
\end{equation}
\end{lem}

\begin{proof}
It suffices to control $\max_{B_{R/\delta}} \delta v^\delta(\cdot,\omega)$, since by \eqref{delvestbel}, $\delta v^\delta (\cdot,\omega)$ is bounded from below uniformly in $\delta$ and a.s. in $\omega$.

We first estimate the quantity $\Prob\big[ \inf_{y\in B_{r/\delta}} \sup_{z\in B(y,1)} \V(z,\cdot) \geq t \big]$ for $r >0$, where $t>0$ is selected below. Find points $y_1,\ldots, y_{k_\delta} \in B_r$ such that $k_\delta \approx |\log \delta|$ and $|y_i - y_j| > r_\delta$ for $i\neq j$, where $r_\delta \approx k_\delta^{-1/\d} r$. For any $t> 0$,
\begin{align*}
\Prob\Big[ \inf_{y\in B_{r/\delta}} \sup_{z\in B(y,1)} \V(z,\cdot) \geq t \Big] & \leq \Prob\Big[ \min_{1 \leq i \leq k_\delta} \sup_{z\in B(y_i/\delta,1)} \V(z,\cdot) \geq t \Big]
\end{align*}
By the mixing hypothesis, the quantity on the right is at most
\begin{equation*}
\Prob\Big[ \sup_{z\in B(y,1)} \V(z,\cdot) \geq t \Big]^{k_\delta} + C(k_\delta-1) \Big( \frac{r_\delta}{\delta} - 2 \Big)^{-\beta}.
\end{equation*}
Choosing $t> 0$ so that $\Prob\Big[\sup_{z\in B(y,1)} \V(z,\cdot) \geq t \Big] \leq \tfrac 12$, we obtain, for some suitably small $\kappa > 0$,
\begin{equation} \label{oscestlg1}
\Prob\Big[ \inf_{y\in B_{r/\delta}} \sup_{z\in B(y,1)} \V(z,\cdot) \geq t \Big] \leq 2^{-k_\delta} + Ck_\delta^{1+\beta/d} r^{-\beta} \delta^{\beta} \leq C(1+r^{-\beta}) \delta^{\kappa}.
\end{equation}

\medskip

Fix $0 < r < R$ and choose $z_1,\ldots,z_\ell \in B_R$ with $\ell \approx (R/r)^d$ such that $B_R \subseteq \cup_{i=1}^\ell B(z_i,r/2)$. According to \eqref{oscestlg1}, the stationarity of $\V$ and with $t> 0$ chosen as above, we have
\begin{multline*}
\Prob \Big[ \sup_{z\in B_R} \inf_{y\in B(z/\delta, r/\delta)} \sup_{x\in B(y,1)} V(x,\cdot) \geq t\Big] \leq \Prob\Big[ \sup_{1\leq i \leq \ell}\, \inf_{y\in B(z_i/\delta,r/2\delta)} \, \sup_{x\in B(y,1)} \V(x,\cdot) \geq t \Big] \\ \leq \ell\, \Prob\Big[ \inf_{y\in B_{r/\delta}}\, \sup_{z\in B(y,1)} \V(z,\cdot) \geq t \Big] \leq C \ell (1+r^{-\beta}) \delta^\kappa.
\end{multline*}
Therefore, by the Borel-Cantelli lemma, along the diadic sequence $\delta_j : = 2^{-j}$,
\begin{equation} \label{upposc-est}
\limsup_{j\to \infty} \max_{z\in B_{R/\delta_j} } \inf_{y\in B(z,r/\delta_j)} \sup_{x\in B(y,1)} V(x,\omega) \leq t \quad \mbox{a.s. in} \ \omega.
\end{equation}
Since $t> 0$ is independent of both  $R$ nor $r$, we deduce that \eqref{upposc-est} holds along the full sequence $\delta \rightarrow 0$, i.e.,
\begin{equation} \label{upposc-estfull}
\limsup_{\delta \to 0} \max_{z\in B_{R/\delta} } \inf_{y\in B(z,r/\delta)} \sup_{x\in B(y,1)} V(x,\omega) \leq t \quad \mbox{a.s. in} \ \omega.
\end{equation}
Combining this inequality and \eqref{locesteq1} we obtain, for a constant $C$ depending on the appropriate quantities,
\begin{equation*}
\limsup_{\delta \to 0} \sup_{y\in B_{R/\delta} } \inf_{y\in B(z,r/\delta)} \delta v^\delta(y,\omega) \leq C \quad \mbox{a.s. in} \ \omega.
\end{equation*}
The last inequality and \eqref{uber-oscest} imply that
\begin{equation*}
\limsup_{\delta \to 0} \sup_{y\in B_{R/\delta} } \delta v^\delta(y,\omega) \leq C(1+ C r) \quad \mbox{a.s. in} \ \omega.
\end{equation*}
We send $r \to 0$ to obtain the result.
\end{proof}

\begin{remark}  \label{moregen}
The estimate \eqref{oscest-large} permits us to generalize our homogenization result to equations of the form \eqref{HJx}, with coefficients which depend also on the macroscopic variable $x$. The difficulty in generalizing to such equations lies in obtaining a continuous dependence result stating that $\delta v^\delta(0,\omega;p_1,x_1)$ is close to $\delta v^\delta(0,\omega;p_2,x_2)$ if $| p_1-p_2|+|x_1-x_2|$ is small, at least in the limit as $\delta \to 0$ (and this must hold a.s. in $\omega$). While we do not give details here, it is precisely \eqref{oscest-large} that permits us to obtain such a continuous dependence estimate using the classical viscosity solution-theoretic comparison machinery 
\end{remark}

\section{The effective Hamiltonian}
\label{EH}

Here we define the effective Hamiltonian $\overline H$ and explore some of its basic properties. In the process, we perform much of the work for the proof of Theorem~\ref{MAIN}.

\subsection{Construction of the effective Hamiltonian $\overline H$}

The first step in the proof of Theorem~\ref{MAIN}, formulated in the following proposition, is the identification of the effective Hamiltonian $\overline H$. The argument is based on the method recently introduced in \cite{LS3}, although substantial modifications are necessary in the unbounded setting, as discussed above.

\begin{prop} \label{mainstep}
There exists $\Omega_5 \subseteq \Omega_4$ of full probability and a continuous function $\overline H:\Rd \to \R$ such that, for every $R> 0$, $p \in \Rd$, and $\omega\in \Omega_5$,
\begin{equation} \label{mainstepeq}
\lim_{\delta \to 0} \E \Big[ \, \sup_{y\in B_{R/\delta}} \left| \delta v^\delta (y,\omega;p) + \overline H(p)\right| \Big]=0
\end{equation}
and
\begin{equation} \label{oneside}
\liminf_{\delta \to 0} \delta v^\delta(0,\omega) = -\overline H(p).
\end{equation}
Moreover, for each $p \in \Rd$, there exists a function $w: \Rd\times \Omega \to \R$ such that, a.s. in $\omega$,  $w(\cdot,\omega) \in W^{1,\alpha}_{\mathrm{loc}} (\Rd)$, $Dw$ is stationary, and for every $\omega \in \Omega_5$,
\begin{equation} \label{subcorr}
\left\{
\begin{aligned}
& -\tr (A(y,\omega) D^2w ) + H(p+Dw,y,\omega) \leq \overline H(p) \quad \mbox{in} \  \ \Rd, \\
& |y|^{-1} w(y,\omega) \rightarrow 0 \quad \mbox{as}  \quad |y| \to \infty.
\end{aligned}
\right.
\end{equation}
\end{prop}
\begin{proof}
The (local Lipschitz) continuity of $\overline H$ follows from Lemma~\ref{CDE-noxd}, once we have shown \eqref{mainstepeq}. Moreover, according to Lemma~\ref{CDE-noxd}, we may argue for a fixed $p\in \mathbb{Q}^d$, and then obtain $\Omega_5$ by intersecting the relevant subsets of $\Omega$ obtained for each rational $p$. We therefore fix $p$ and omit all dependence on $p$.

\medskip

The proof is divided into three steps.

\medskip

\emph{Step 1: Construction of a subcorrector which is strictly sublinear at infinity.}
For each $\delta > 0$, define
\begin{equation*}
w^\delta(y,\omega) := v^\delta(y,\omega) - v^\delta(0,\omega).
\end{equation*}
Owing to \eqref{Vmoment}, \eqref{locesteq1}, \eqref{delvestbel} and \eqref{locesteq2}, we can find a subsequence $\delta_j \to 0$, a random variable $\overline H = \overline H(p,\omega) \in \R$, a function $w\in L^\alpha_{\mathrm{loc}}(\Rd\times \Omega)$ and a field $\Phi \in L^\alpha_{\mathrm{loc}}(\Rd\times\Omega ; \Rd)$ such that, for every $R> 0$, as $j\to \infty$,
\begin{equation} \label{weaklims}
\left\{ \begin{aligned}
& -\delta_j v^{\delta_j}(0,\cdot ) \rightharpoonup \overline H(p,\cdot) \quad \mbox{weakly} \ \mbox{in} \ L^\alpha(\Omega), \\
& w^{\delta_j} \rightharpoonup w \quad \mbox{weakly} \ \mbox{in} \ L^\alpha(B_R\times\Omega), \\
& Dv^{\delta_j} \rightharpoonup \Phi \quad \mbox{weakly} \ \mbox{in}  \ L^{\alpha}(B_R\times\Omega;\Rd).
\end{aligned} \right.
\end{equation}
The stationarity of the functions $v^{\delta_j}$, the ergodicity hypothesis and Lemma~\ref{impest} imply that $\overline H$ is independent of $\omega$, i.e., $\overline H(p,\omega) = \overline H(p)$ a.s. in $\omega$. Indeed, it suffices to check that, for each $\mu \in \R$, the event $\{ \omega \in \Omega : \overline H(p,\omega) \geq \mu \}$ is invariant under $\tau_y$, which follows immediately from \eqref{oscest}.

It is clear that $\Phi$ is stationary, a property inherited from the sequence $\{ Dv^{\delta_j}\}$, and that $\Phi = (\Phi^1,\ldots,\Phi^d)$ is gradient-like in the sense that for every compactly-supported smooth test function $\psi=\psi(y)$,
\begin{equation*}
\int_{\Rd} \left( \Phi^i(y,\omega) \psi_{y_j} (y) - \Phi^j(y,\omega) \psi_{y_i}(y) \right) \, dy = 0 \quad \mbox{a.s. in} \ \omega.
\end{equation*}
It follows (c.f. Kozlov \cite[Proposition 7]{K}) that $\Phi = Dw$, a.s. in $\omega$ and in the sense of distributions. Since $\alpha > d$, the Sobolev imbedding theorem yields, without loss of generality, that $w(\cdot,\omega) \in C(\Rd)$ a.s. in $\omega$.

\medskip

The convexity hypothesis \eqref{Hconv} and the equivalence of distributional and viscosity solutions for linear inequalities (c.f. Ishii~\cite{I}) allow us to pass to weak limits in \eqref{aux}, obtaining that $w(\cdot,\omega)$ is a viscosity solution, a.s. in $\omega$, of
\begin{equation} \label{subcorreq}
 - \tr\!\left( A(y,\omega) D^2w \right) + H(p+Dw, y, \omega) \leq \overline H(p) \quad \mbox{in} \ \Rd.
\end{equation}
Since
\begin{equation*}
\E\Phi(0,\cdot) = \lim_{j\to \infty} \E Dv^{\delta_j}(0,\cdot) = 0,
\end{equation*}
it follows from Lemma~\ref{koz} that
\begin{equation} \label{sublininfty}
\lim_{|y|\to\infty} |y|^{-1}w(y,\omega) = 0 \quad \mbox{a.s. in} \ \omega.
\end{equation}
Therefore $w$ is a subcorrector which is strictly sublinear at infinity.

\medskip

\emph{Step 2: $\overline H$ characterizes the full limit of $\delta v^\delta(0,\omega)$ in $L^1(\Omega,\mathbb P)$.}
The key step consists in showing that
\begin{equation} \label{cmpineq}
-\overline H\leq \liminf_{\delta \to 0} \delta v^\delta(0,\omega) \quad \mbox{a.s. in} \ \omega,
\end{equation}
which we prove by comparing the subcorrector $w$ to $v^\delta$. Let $\widetilde\Omega$ be a subset of $\Omega$ of full probability such that $\widetilde\Omega \subseteq \Omega_4$ and, for every $\omega \in \widetilde\Omega$, we have $\overline H(p,\omega) = \overline H(p)$ as well as \eqref{subcorreq} and \eqref{sublininfty}.

\medskip

Fix $\omega\in \widetilde\Omega$. We remark that the constants we introduce immediately below depend on $\omega$. Let $\varphi$ be defined by \eqref{linearphi}, and recall from \eqref{phieqbnd} that
\begin{equation} \label{msphibnd}
\left| \tr( A(y,\omega) D^2\varphi) \right| + H(p+D\varphi,y,\omega) \leq C.
\end{equation}
For each $\delta > 0$, define the function
\begin{equation*}
w^\delta(y) : = (1-\ep) \!\left( w(y,\omega) - ( \overline H + \eta ) \delta^{-1} \right) + \ep \varphi(y),
\end{equation*}
where $\eta > 0$ is a given small constant and $\ep > 0$ will be chosen below in terms of $\eta$. The strategy for obtaining \eqref{cmpineq} lies in comparing $w^\delta$ and $v^\delta$ in the limit as $\delta \to 0$. Assuming that $w$ is smooth, in view of \eqref{subcorreq}, \eqref{msphibnd} and the convexity of $H$, we have
\begin{equation}
\delta w^\delta -\tr\!\left( A(y,\omega) D^2w^\delta \right) +H(p+Dw^\delta,y)  \leq \delta w^\delta + (1-\ep) \overline H + C \ep.
\label{delwdel}
\end{equation}
In the case $w$ is not smooth, one can verify \eqref{delwdel} in the viscosity sense by using either that $\varphi$ is smooth, or by appealing to Lemma~\ref{convtrick}. According to \eqref{sublininfty},
\begin{equation*}
\sup_{B_R} w \leq C_\eta + \eta^3 R,
\end{equation*}
and so, by choosing $\ep = \min\{ 1/4, \eta/4C\}$ with $C$ is as in \eqref{delwdel}, we estimate the right side of \eqref{delwdel} by
\begin{equation*}
\delta w^\delta + (1-\ep) \overline H + C \ep = (1-\ep) (\delta w-\eta) + C\ep \leq \delta C_{\eta} + \delta \eta^3 R - \frac12 \eta \quad \mbox{in} \ B_R.
\end{equation*}
Observe that \eqref{delvestbel} implies
\begin{equation*}
w^\delta - v^\delta \leq (1-\ep) w +C \delta^{-1} - c\eta R \quad \mbox{on} \ \partial B_R.
\end{equation*}
By selecting $R : = C(\delta\eta)^{-1}$ for a large constant $C> 0$ and taking $\delta$ to be sufficiently small, depending on both $\omega$ and $\eta$, we have
\begin{equation}
\left\{ \begin{aligned}
& \delta w^\delta -\tr \! \left( A(x,y,\omega) D^2 w^\delta \right) + H(p+Dw^\delta,x,y,\omega) \leq 0 & \mbox{in} & \ B_R, \\
& w^\delta \leq v^\delta & \mbox{on} & \ \partial B_R.
\end{aligned} \right.
\end{equation}
We may now apply the comparison principle to deduce that $w^\delta(\cdot) \leq v^\delta(\cdot,\omega)$ in $B_R$, and in particular, $w^\delta(0) \leq v^\delta(0,\omega)$. Multiplying this inequality by $\delta$ and sending $\delta \to 0$ yields
\begin{equation*}
-\overline H -\eta \leq ( 1 - C \eta)^{-1} \liminf_{\delta \to 0} \delta v^\delta (0,\omega).
\end{equation*}
Recalling that $\omega \in \widetilde\Omega$ was arbitrary with $\Prob\big[\widetilde\Omega\big] =1$ and disposing of $\eta > 0$ yields \eqref{cmpineq}.

\medskip

Since $-\overline H$ is the weak limit of the sequence $\delta_j v^{\delta_j}(0,\cdot)$, the reverse of \eqref{cmpineq} is immediate and we obtain \eqref{oneside}. It follows that the full sequence $\delta v^\delta(0,\cdot)$ converges weakly to $-\overline H$, that is, as $\delta \to 0$,
\begin{equation*}
\delta v^\delta (0,\cdot) \rightharpoonup -\overline H = \liminf_{\delta\to 0} \delta v^\delta(0,\cdot) \quad \mbox{weakly in} \ L^\alpha(\Omega).
\end{equation*}
Lemma~\ref{meastheor} yields
\begin{equation}\label{convprob}
\E \big| \delta v^\delta(0,\cdot) + \overline H \big| \rightarrow 0 \quad \mbox{as} \quad \delta \to 0.
\end{equation}

\medskip

\emph{Step 3: Improvement of \eqref{convprob} to balls of radius $\sim 1/\delta$.} We claim that, for each $R> 0$,
\begin{equation} \label{convprob1del}
\lim_{\delta\to0} \E \bigg[ \sup_{y\in B_{R/\delta}} \big| \delta v^\delta(0,\cdot) + \overline H \big| \bigg] = 0.
\end{equation}
Fix $R > 0$. Let $\rho > 0$ and select points $y_1,\ldots,y_k \in B_R$ such that
\begin{equation*}
B_R \subseteq \bigcup_{i=1}^k B(y_i,\rho) \quad \mbox{and} \quad k \leq C \bigg( \frac{R}{\rho}\bigg)^\d.
\end{equation*}
Using \eqref{uber-oscest}, we find
\begin{align*}
\lefteqn{ \limsup_{\delta \to 0} \E \bigg[ \sup_{y\in B_{R/\delta}} \big| \delta v^\delta(0,\cdot) + \overline H \big| \bigg]} \qquad & \\
& \leq \sum_{i=1}^k \limsup_{\delta \to 0} \E \big| \delta v^\delta(y_i/\delta ,\cdot) + \overline H \big| + \limsup_{\delta\to 0} \E\bigg[  \max_{1\leq i \leq k} \osc_{z \in B(y_i/\delta, \rho/\delta)} \delta v^\delta(z,\cdot) \bigg] \\
& \leq k\limsup_{\delta\to 0}\E \big| \delta v^\delta(0,\cdot) + \overline H \big| + C\rho \\
& =  C\rho.
\end{align*}
Disposing of $\rho>0$ yields \eqref{convprob1del}.
\end{proof}


\begin{remark}
It is clear that the subcorrector $w$ is locally Lipschitz in $\Rd$ with a constant controlled by $V$. Indeed, recall from \eqref{weaklims} that since $\Phi=Dw$ is the weak limit of $Dv^{\delta_j}$, we deduce from \eqref{locesteq2} that for a.e. $z\in B(y,1/2)$,
\begin{equation}
|Dw(z,\omega)|^\gamma \leq \limsup_{j\to \infty} |Dv^{\delta_j}(z,\omega)|^\gamma \leq C\Big( 1+ \sup_{B(y,1)} V(\cdot,\omega) \Big) \quad \mbox{a.s. in} \ \omega.
\end{equation}
Hence
\begin{equation} \label{wLIP}
\esssup_{B(y,1/2)} |Dw(\cdot,\omega)|^\gamma \leq C\Big( 1+ \sup_{B(y,1)} V(\cdot,\omega) \Big) \quad \mbox{a.s. in} \ \omega.
\end{equation}
\end{remark}

\begin{remark} \label{assubseq}
From \eqref{mainstepeq} we deduce the existence of a subsequence $\delta_j \to 0$ and a subset $\Omega_6 \subseteq \Omega_5$ of full probability such that, for every $\omega\in \Omega_6$,
\begin{equation}\label{assubseqeq}
\lim_{j \to \infty} \sup_{y\in B_{R/\delta_j}} \left| \delta_j v^{\delta_j} (y,\omega;p) + \overline H(p)\right| =0.
\end{equation}
By a diagonalization procedure, and by intersecting the relevant subsets of $\Omega$, we may assume that \eqref{assubseqeq} holds for every $R> 0$, rational $p$ and $\omega\in \Omega_6$. Then by Lemma~\ref{CDE-noxd} we obtain \eqref{assubseqeq} for all $p\in \Rd$ and $\omega \in \Omega_6$.
\end{remark}

\begin{remark} \label{WDlimsup}
Proposition~\ref{mainstep} is a partial result in the direction of our ultimate goal, which is \eqref{cplim}. In light of \eqref{oneside}, the key step that remains is to show that $\limsup_{\delta\to 0} \delta v^\delta(0,\omega; p) = -\overline H(p)$. While we must postpone the proof of this fact until we have studied the metric problem, let us note here the quantity $\limsup_{\delta\to 0} \delta v^\delta(0,\omega; p)$ is at least deterministic, a.s. in $\omega$. Indeed, observe that, according to Lemma~\ref{impest} and the stationarity of the $v^\delta$'s, for each $\mu \in \R$, the set
\begin{equation*}
\Big\{ \omega\in \Omega_3: \limsup_{\delta \to 0} \delta v^\delta(0,\omega) \geq \mu \Big\} = \Big\{ \tau_{y} \omega\in \Omega_3: \limsup_{\delta \to 0} \delta v^\delta(y,\omega) \geq \mu \Big\}
\end{equation*}
is invariant under $\tau_y$ for every $y\in \Rd$. Therefore, the ergodic hypothesis implies that each of these sets has probability 0 or 1. Letting $-\hat H=-\hat H(p)$ denote the supremum over $\mu$ for which the set above has full probability, we obtain
\begin{equation}\label{hatHdefeq}
\limsup_{\delta\to0} \delta v^\delta(0,\omega;p) = -\hat H(p) \quad \mbox{a.s. in}  \ \omega.
\end{equation}
Select $\Omega_7 \subseteq \Omega_6$ such that \eqref{hatHdefeq} holds for all $\omega \in \Omega_7$.
\end{remark}

\subsection{Some properties of $\overline H$} \label{Hbarprop}

The effective nonlinearity $\overline H$ inherits the convexity, coercivity and continuity of $H$. We record these elementary observations in the next proposition. We conclude this subsection with a discussion of an interesting property $\overline H$ possesses in a particular case in the random setting, which strongly contrasts with the situation in periodic and almost periodic media.

\begin{prop} \label{effHam}
The effective nonlinearity $\overline H$ has the following properties:
\begin{enumerate}
\item[(i)] $p\mapsto \overline H(p)$ is convex,
\item[(ii)] with $c_0 > 0$ as in \eqref{Hcoer} and for all $p \in \Rd$ 
\begin{equation} \label{Hbarcoer}
\frac12 c_0 |p|^\gamma - C \leq \overline H(p) \leq C\left(1+|p|^\gamma\right),
\end{equation}
\item[(iii)] for all $p_1,p_2\in\Rd$,
\begin{equation} \label{contpeq}
\left| \overline H(p_1) - \overline H(p_2) \right| \leq C \! \left(1 + |p_1|+|p_2| \right)^{\gamma-1} |p_1-p_2|.
\end{equation}
\end{enumerate}
\end{prop}
\begin{proof}
Select $p_1,p_2,x\in \Rd$ and, for each $\delta > 0$, write $v_i^\delta (y,\omega) : = v^\delta(y,\omega; p_i)$. Set $p : = \frac12(p_1+p_2)$ and $\tilde v^\delta : = \frac12( v^\delta_1 + v^\delta_2)$. Applying Lemma~\ref{convtrick}, we observe that $\tilde v^\delta$ satisfies
\begin{equation*}
\delta \tilde v^\delta  - \tr \! \left( A(y,\omega) D^2 \tilde v \right) + H( p+ D\tilde v, y, \omega) \leq 0 \quad \mbox{in} \ \Rd.
\end{equation*}
Lemma~\ref{comparison} yields $\tilde v^\delta(y,\omega) \leq v^\delta(y,\omega;p)$ a.s. in $\omega$. Multiplying this inequality by $-\delta$ and sending $\delta \to 0$, we obtain
\begin{equation*}
\overline H(p) \leq \tfrac12 \overline H(p_1) + \tfrac12 \overline H(p_2).
\end{equation*}
and (i) follows.

\medskip

The upper bound in \eqref{Hbarcoer} follows immediately from \eqref{delvestbel}. The coercivity of $\overline H$ is more subtle, requiring an integration by parts. To this end, we may assume, by regularizing the coefficients and making the diffusion matrix $A$ uniformly elliptic, that $v^\delta$ is smooth. Integrating \eqref{aux} over $B_1$ and taking expectations yields
\begin{equation} \label{coerint}
\E \fint_{B_1} -\delta v^\delta(y) \, dy = \E \fint_{B_1}\left( - \tr\!\left( A(y,\omega) D^2v^\delta \right) + H(p+Dv^\delta, y, \omega)  \right) \, dy,
\end{equation}
while Proposition~\ref{mainstep} gives
\begin{equation} \label{Hcoerc1}
\lim_{\delta \to 0} \E \fint_{B_1} -\delta v^\delta \, dy = \overline H(p).
\end{equation}
Integrating by parts and using \eqref{Alip} and \eqref{locesteq2}, we can estimate the first term on the right side of \eqref{coerint} pointwise in $\omega$ by
\begin{equation*}
\left| \int_{B_1} \tr \!\left( A(y,\omega) D^2v^\delta \right) \, dy \right| \leq C \int_{B_1} |Dv^\delta(y)|\, dy + C\sup_{y\in \partial B_1} |Dv^\delta(y)| \leq C \sup_{y\in B_2} V(y,\omega).
\end{equation*}
Taking expectations yields
\begin{equation} \label{diffest}
\E \left| \fint_{B_1} \tr \!\left( A(y,\omega) D^2v^\delta \right) \, dy \right| \leq C \, \E \sup_{B_2} V(y,\omega) \leq C.
\end{equation}
The second term on the right side of \eqref{coerint} is estimated from below using Jensen's inequality, \eqref{Hcoer} and \eqref{locesteq2}, as follows:
\begin{align}
\E \fint_{B_1} H(p+Dv^\delta ,y,\cdot) \, dy & \geq \E \fint_{B_1} \left( c_0 |p+Dv^\delta|^\gamma - V(y,\cdot) - C_0 \right)  dy \nonumber \\
& \geq c_0 \left| \, p + \E \fint_{B_1} Dv^\delta(y) \, dy\,\right|^\gamma - C \nonumber \\
& \geq \frac{1}{2} c_0 |p|^\gamma - \E \sup_{y \in B_1} |Dv^\delta(y)|^\gamma \, dy - C \nonumber \\
& \geq \frac12 c_0 |p|^{\gamma} - C. \label{Hcoerc2}
\end{align}
Sending $\delta \to 0$ in \eqref{coerint} using \eqref{Hcoerc1}, \eqref{diffest} and \eqref{Hcoerc2} yields the lower bound in \eqref{Hbarcoer}.

\medskip

Finally, we note that (iii) follows from Lemma~\ref{CDE-noxd} by simply sending $\delta \to 0$ in \eqref{CDE-noxdeq}.
\end{proof}

\begin{remark} \label{effHcmp}
The properties of $\overline H$ enumerated in Proposition~\ref{effHam} imply, in view of standard viscosity solution theory, that the problem \eqref{HJ-eff} has a unique solution.
\end{remark}

We conclude this section with some observations about the effective Hamiltonian in the particular case that, for all $p,y \in \Rd$,
\begin{equation} \label{parcas}
H(p,y,\omega) \geq H(0,y,\omega) \quad \mbox{a.s. in} \ \omega.
\end{equation}
It is essentially known (at least in the bounded case, see for example \cite{LPV,LS2}) that for separated Hamiltonians, i.e., if
\begin{equation}\label{sepcase}
H(p,y,\omega) = \widetilde H(p) - V(y,\omega),
\end{equation}
where $\widetilde H(p) \geq \widetilde H(0)$ and $A\equiv 0$, we have
\begin{equation} \label{Hbarinv}
\overline H(p) \geq \overline H(0) = -\essinf_{\Omega} V(0,\cdot).
\end{equation}
Here we prove the following more general fact.

\begin{prop} \label{invispot-easy-prop}
Assume that $A\equiv 0$ and \eqref{parcas}. Then
\begin{equation} \label{aspacsas}
\overline H(p) \geq \overline H(0) = \esssup_{\omega \in \Omega} H(0,0,\omega).
\end{equation}
\end{prop}
\begin{proof}
We may assume that $\esssup_\Omega H(0,0,\cdot) = 0$. Fix $p\in \Rd$ and $\ep > 0$, and define
\begin{equation*}
\widetilde\Omega_\ep : = \{ \omega \in \Omega : H(0,0,\omega) > - \ep \}.
\end{equation*}
Observe that $\Prob\big[\widetilde\Omega_\ep\big] > 0$. We claim that, for any $\delta > 0$ and $\omega \in \widetilde\Omega_\ep \cap \Omega_1$, where $\Omega_1$ is as in Lemma~\ref{comparison}, we have
\begin{equation*}
\delta v^\delta(0,\omega; p) \leq 2\ep.
\end{equation*}
To prove this, we construct a very simple barrier. In fact we may take any smooth function $w$ satisfying $w \geq w(0) = 2\delta^{-1} \ep$ in $B_r$ and $w(y)\to \infty$ as $|y| \to r$. If $r > 0$ is small enough, depending on $\omega$, then, by continuity,
\begin{equation*}
\max_{B_r} H(0,0,\omega) > - 2\ep.
\end{equation*}
It follows from this, $A\equiv 0$ and \eqref{parcas} that $w$ is a supersolution of \eqref{aux} in $B_r$. Thus the comparison principle gives $\delta v^\delta(0,\omega;p) \leq \delta w(0) = 2\ep$.

Since $\Prob\big[\widetilde\Omega_\ep\cap \Omega_1 \big] > 0$, by \eqref{mainstepeq} we obtain $\overline H(p) \geq -2\ep$.
Sending $\ep \to 0$ yields $\overline H(p) \geq 0$. It follows that $\overline H(0) =  0$, since it is clear that
\begin{equation*}
\overline H(0) \leq \esssup_\Omega H(0,0,\cdot) = 0. \qedhere
\end{equation*}
\end{proof}

\begin{remark} \label{invispot-easy}
One observes from \eqref{Hbarinv} the reason we must require $V$ to be bounded from below, even as our assumptions allow it to be unbounded from above.
\end{remark}

The situation is different in the presence of diffusion, i.e., if $A\not\equiv 0$. For example, an easy calculation shows in the periodic setting, with $A = \iden$ and $H$ separated as in \eqref{sepcase}, that \eqref{Hbarinv} holds if and only if $V$ is constant. We nonetheless observe in the next proposition that, in the latter case, \eqref{Hbarinv} holds if $V$ has arbitrarily large ``bare spots" (defined implicitly below) near its essential infimum. This phenomenon is special to the random setting. It is true for a Poissonian potential, for instance, while in contrast, any periodic or almost periodic potential having arbitrarily large bare spots is necessarily constant.

\begin{prop} \label{invispot}
Suppose that $H$ is of the form \eqref{sepcase} and, for all $p\in \Rd$,
\begin{equation} \label{tH}
\widetilde H(p) \geq \widetilde H(0) = 0,
\end{equation}
and, for each $\mu,R> 0$,
\begin{equation} \label{barespot}
\Prob\left[ \sup_{y\in B_R} V(y,\cdot) \leq \mu \right] > 0.
\end{equation}
Then $\overline H \geq 0$, and, if, in addition, we assume that $V\geq 0$, then $\overline H(0) = 0$.
\end{prop}
\begin{proof}
Fix $\mu> 0$ and $R> 1$, and set $E: = \{ \omega \in \Omega : \sup_{B_R} V(y,\omega) \leq \mu \}$. The ergodic theorem yields an event $\widetilde\Omega\subseteq \Omega$ of full probability such that, for each $\omega \in \widetilde\Omega$,
\begin{equation} \label{eprtuoa}
\lim_{r\to \infty} \fint_{B_r} \mathds{1}_E (\tau_y\omega) \, dy = \Prob[E] > 0.
\end{equation}
Fix $\omega \in \widetilde\Omega$. According to \eqref{eprtuoa}, for every sufficiently small $\delta$, depending on $\omega$, we can find $z \in B_{1/\delta}$ for which $\tau_z\omega \in E$. This implies that
\begin{equation} \label{barespotz}
\sup_{B(z,R)} V(\cdot ,\omega) = \sup_{B(0,R)} V(\cdot, \tau_z\omega) \leq \mu.
\end{equation}
Using a barrier function similar to the one in the proof of Lemma~\ref{locest}, we estimate $\delta v^\delta (z,\omega; p)$ from above. Translating the equation, we may assume that $z=0$. Notice that the hypotheses on $H$ imply that
\begin{equation} \label{Htform}
H(p,y,\omega) \geq \left( c_0 |p|^\gamma - C_0\right)_+ - V(y,\omega).
\end{equation}
Since the calculations below rely only on \eqref{Htform}, we assume that $\gamma < 2$. With positive constants $a,b> 0$ to be chosen, define
\begin{equation*}
w(y) : = a + b \left( R^2 - |y|^2 \right)^{-\eta} -bR^{-2\eta},
\end{equation*}
where $\eta : = (2-\gamma) / (\gamma-1)$, so that $\gamma(\eta+1) = \eta+2$. Routine calculations give
\begin{equation*}
|Dw(y)|^\gamma \geq c b^\gamma |y|^\gamma \left( R^2- |y|^2 \right)^{-(\eta+2)} \quad \mbox{and} \quad \tr\!\left(AD^2w(y)) \leq C b R^2 (R^2-|y|^2 \right)^{-(\eta+2)}.
\end{equation*}
Inserting $w$ into \eqref{aux} and using \eqref{barespotz} and \eqref{Htform}, we find 
\begin{multline}
\delta w - \tr(AD^2w) + H(p+Dw,y,\omega)
\\ \geq \delta a - CbR^2 (R^2-|y|^2)^{-(\eta+2)} + \left( cb^\gamma |y|^\gamma  \left( R^2- |y|^2 \right)^{-(\eta+2)} - C_0\right)_+ - \mu \quad \mbox{in} \ B_R.
\end{multline}
Taking $b:= CR^\eta$ for a sufficiently large constant $C$, we get
\begin{equation*}
\left( cb^\gamma |y|^\gamma  \left( R^2- |y|^2 \right)^{-(\eta+2)} - C\right)_+ \geq  CbR^2  \left( R^2- |y|^2 \right)^{-(\eta+2)} \quad \mbox{for all} \quad  \frac R2 \leq |y| < R.
\end{equation*}
Of course,
\begin{equation*}
CbR^2  \left( R^2- |y|^2 \right)^{-(\eta+2)} \leq CbR^{-2(\eta+1)} \leq CR^{-2-\eta} \quad \mbox{for all} \quad |y | \leq \frac R2.
\end{equation*}
Therefore in $B_R$, we have
\begin{equation*}
\delta w - \tr(AD^2w) + H(p+Dw,y,\omega)  \geq  \delta a -\mu - CR^{-2-\eta}.
\end{equation*}
If $a := \delta^{-1}(\mu + CR^{-2-\eta})$, then $w$ is a supersolution of \eqref{aux}. Since $w(y) \rightarrow +\infty$ as $|y| \to R$, the comparison principle implies that
\begin{equation*}
\delta v^\delta(0,\omega;p) \leq \delta w(0) = \delta a = \mu + CR^{-2-\eta}.
\end{equation*}
Taking expectations and passing to the limit $\delta \to 0$, using \eqref{mainstepeq}, yields that
\begin{equation*}
\overline H(p) \geq -\mu - CR^{-2-\eta}.
\end{equation*}
Sending $\mu \to 0$ and $R\to \infty$ gives the conclusion.
\end{proof}

In the case that $\gamma=2$, $\overline H(0)$ is related to the bottom of the spectrum of random Schrodinger operators, and the previous proposition is known. See for instance Carmona and Lacroix \cite{CL} as well as Section 8 of \cite{LS2}.

Sznitman \cite{Sz2} proved more than Proposition~\ref{invispot} in the case that $A=\iden$, $\widetilde H(p) = |p|^2$ and $V$ is a Poissonian potential. In particular, he showed that $\overline H$ has a ``flat spot" near $p=0$, that is, $\overline H(p) = 0$ for small $|p|$, a result which in his language he calls the ``nondegeneracy of the quenched Lyapunov exponents" (see \cite[Proposition 5.2.9]{Szb}). See also Section~\ref{P} below for a guide to translating between our notation and that of \cite{Szb}.

\medskip

\begin{remark}\label{approximation}
It is sometimes convenient to approximate $\overline H$ by effective Hamiltonians corresponding to bounded environments. For $c_0$ as in \eqref{Hcoer}, define
\begin{equation} \label{Hk}
H_k(p,y,\omega) : = \max\big\{ H(p,y,\omega), c_0|p|^\gamma - k \big\}.
\end{equation}
It is clear that $H_k$ satisfies the same hypotheses as $H$, and, in addition, $H_k$ is uniformly coercive, i.e., $H_k(p,y,\omega) \geq c_0|p|^\gamma - k$. Obviously $H_k \geq H$ and $H_k \downarrow H$ as $k\to \infty$. In fact, by Lemma~\ref{potests}, there is a constant $\sigma < 1$ such that
\begin{equation}
0 \leq H_k(p,y,\omega) - H(p,y,\omega) \leq C_0 + V(y,\omega) - k \leq C\big(1+|y|^\sigma\big) - k \quad \mbox{a.s. in} \ \omega.
\end{equation}
To each $H_k$ corresponds an ${\overline H}_k$ and, in view of the monotonicity of the $H_k$'s with respect to $k$, it is obvious that ${\overline H}_k \downarrow \hat H$ as $k\to \infty$. For bounded environments it is easy to check using standard viscosity arguments that ${\overline H} = \hat H$. 
The argument is more complicated in the unbounded setting. We leave it up to the interested reader to fill in the details.
\end{remark}

We conclude this section by writing down the usual inf-sup formula for the effective Hamiltonian:
\begin{equation}\label{infsup}
\overline H(p) = \inf_{\Phi \in \mathcal S} \sup_{y\in \Rd} \big( -\tr A(y,\omega)D^2\Phi + H(p+D\Phi,y,\omega) \big).
\end{equation}
Here $\mathcal S$ is the set of $\Phi : \Rd \times \Omega \rightarrow \R$ such that $\Phi(\cdot,\omega)$ is locally Lipschitz a.s. in $\omega$ and $D\Phi$ is stationary with $\E [D\Phi ]= 0$. The supremum is to be understood in the viscosity sense, i.e., it is defined to be the smallest constant $k$ for which the inequality
\begin{equation*}
-\tr A(y,\omega)D^2\Phi + H(p+D\Phi,y,\omega) \leq k \quad \mbox{in} \ \Rd
\end{equation*}
holds in the viscosity sense.

The proof of \eqref{infsup} is simple. The inequality ``$\leq$" is clear from the existence of the subcorrector in Proposition~\ref{mainstep}. If this inequality were strict, however, we could repeat the argument used to obtain \eqref{cmpineq} in the proof of Proposition~\ref{mainstep}, with an appropriate choice of $\Phi$ in place of $w$, to obtain an improvement of \eqref{cmpineq}. But this would contradict \eqref{oneside}.

\section{The metric problem: another characterization of $\overline H$}

 \label{Q}

\nocite{R,KMZ}

In this section we consider the special stationary equation
\begin{equation} \label{MP}
- \tr\! \left( A(y,\omega) D^2u\right) + H(p+Du, y, \omega)  = \mu \quad \mbox{in} \ \Rd\!\setminus \! D
\end{equation}
where $D$ is a bounded, closed subset of $\Rd$ and $\mu \in \R$.  Subject to appropriate boundary and growth conditions, a solution $u$ of \eqref{MP} is related to the ``metric" (distance function) associated with the effective Hamiltonian.

\medskip

Our motivation for studying \eqref{MP} is threefold. First, as we will see, solutions of \eqref{MP} possess some subadditive structure. An application of the subadditive ergodic theorem therefore permits us to essentially homogenize a rescaled version of \eqref{MP}. Working backwards we are able to improve the convergence in $L^1$ obtained in \eqref{mainstepeq} to almost sure convergence, which is done later in Section~\ref{BS}. This is a critical step in the proof of Theorem~\ref{MAIN}. Second, we provide a characterization of $\overline H$ in terms of the solvability of \eqref{MP}, which is new even in the bounded setting. We hope that this formula will yield new information about the structure of the effective Hamiltonian, and we intend to return to this point in future work. Finally, the metric problem is natural from the probability point of view, and allows us to make a precise connection to the results of Sznitman in the case of the model equation \eqref{PCP} with $\gamma=2$ and a Poissonian potential $V$.

\medskip

We remark that, in this section, the subsets of $\Omega$ of full probability on which our ``almost sure" assertions hold depend on $p$. This is because, for the sake of brevity and because it is not required for the proof of our main theorems, we do not wish to trouble ourselves with proving a separate continuous dependence result.

\subsection{Well-posedness}

We begin by showing that \eqref{MP} is well-posed for each $\mu > \overline H(p)$, a.s. in $\omega$, subject to appropriate growth conditions at infinity and the boundary condition $u=0$ on $\partial D_1$, where $D=D_1$ is defined by
\begin{equation*}
D_1 : = \begin{cases} \{ 0 \} & \mbox{if} \ \gamma > 2 \ \mbox{or} \ A \equiv 0, \\ B_1 & \mbox{if} \ \gamma \leq 2 \ \mbox{and} \ A \not\equiv 0.\end{cases}
\end{equation*}
The reason for defining $D_1$ in this way, from the probability point of view, is that, in contrast to the case $\gamma\leq 2$, for $\gamma > 2$, it is possible to have Brownian bridges. That is, we may connect two points via a diffusion if $\gamma > 2$, while if $\gamma=2$ we may only connect a point to a small ball. From the pde point of view, this manifests itself in the kind of barrier functions we are able to build.

\medskip

In the next proposition, we prove a comparison result. Existence then follows from the Perron method once suitable barriers have been constructed. 

\medskip

We prove the next proposition with a new argument that, as far as we know, has no analogue in the literature. What makes this results different from typical comparison results on unbounded domains is that we do \emph{not} assume that the subsolution $u$ and supersolution $v$ separate only sublinearly at infinity; see \eqref{meteqcmpgc}. In the proof, we ``lower" $u$ until it has strictly sublinear separation from $v$, and then argue that, if we needed to lower $u$ at all, we could have lowered it a little less. To accomplish this, first we perturb $u$ by subtracting a term $\varphi_R$ which is negligible in balls of radius $\sim R$ but grows linearly at infinity. Then we  compare the resulting
function with $v$. We conclude sending $R \to \infty$. The fact that the parameter $\mu$ is strictly larger than $\overline H$ allows us to compensate for this perturbation by using the subcorrector.

\begin{prop} \label{metcomparison}
Fix $p\in \Rd$, $\mu > \overline H(p)$ and $\omega \in\Omega_7$. Suppose that $D\subseteq \Rd$ is closed and bounded. Assume that $u\in\USC(\overline{ \Rd\!\setminus \!D})$ and $v\in \LSC(\overline{ \Rd\!\setminus\! D})$ are, respectively, a subsolution and supersolution of \eqref{MP} such that $u \leq v$ on $\partial D$, and
\begin{equation} \label{meteqcmpgc}
\limsup_{|y| \to \infty} \frac{u(y)}{|y|} < \infty \quad \mbox{and} \quad \liminf_{|y| \to \infty} \frac{v(y)}{|y|} \geq 0.
\end{equation}
Then $u \leq v$ in $\Rd \!\setminus \! D$.
\end{prop}
\begin{proof}
We may assume that $\limsup_{|y| \to \infty} u(y)/|y| \geq 0$, since otherwise the result is immediate from the usual comparison principle (c.f. \cite{CIL}). We omit all dependence on $\omega$, since it plays no role in the argument. We may also assume that $p=0$. Define
\begin{equation*}
\Lambda : = \Big\{ 0 \leq \lambda \leq 1 : \liminf_{|y| \to \infty} \frac{v(y) - \lambda u(y)}{|y|} \geq 0 \Big\} \qquad \mbox{and} \qquad \overline \lambda: = \sup \Lambda.
\end{equation*}
We begin with the observation that $\Lambda = \big[0,\overline \lambda \big]$. This follows from \eqref{meteqcmpgc} and a continuity argument. The assumption \eqref{meteqcmpgc} implies $0 \in \Lambda$. To see that $\overline \lambda \in \Lambda$, select $\ep > 0$ and $\lambda \in \Lambda$ such that $\overline \lambda \leq \lambda + \ep$, and observe that, by \eqref{meteqcmpgc},
\begin{multline*}
\liminf_{|y| \to \infty} \frac{v(y) - \overline \lambda u(y)}{|y|}  \geq \frac{\overline \lambda}{\lambda+\ep} \liminf_{|y| \to \infty} \frac{v(y) - (\lambda+\ep)u(y)}{|y|} \\ \geq \frac{\overline \lambda}{\lambda+\ep} \left( -\ep \limsup_{|y| \to\infty} \frac{u(y)}{|y|}\right) \geq -C\overline \lambda \ep(\lambda+\ep)^{-1}.
\end{multline*}
Sending $\ep \to 0$ yields $\overline \lambda \in \Lambda$. If $\lambda \in \big(0,\overline \lambda\big)$, then using again \eqref{meteqcmpgc}, we have
\begin{equation*}
\liminf_{|y| \to \infty} \frac{v(y) - \lambda u(y)}{|y|}  \geq \frac{\lambda}{\overline \lambda} \liminf_{|y| \to \infty} \frac{v(y) - \overline \lambda u(y)}{|y|} \geq 0,
\end{equation*}
and, hence, the claim.

\medskip

Next we show that $\overline \lambda = 1$. Select $\lambda < \overline \lambda \leq 1$. For each $R> 1$, define the auxiliary function
\begin{equation} \label{varphiR}
\varphi_R(y): = R- (R^2+|y|^2)^{1/2},
\end{equation}
and observe that, for a constant $C > 0$ independent of $R> 1$,
\begin{equation} \label{phidcn}
|D\varphi_R(y)| + |D^2\varphi_R(y)| \leq C.
\end{equation}
We have defined $\varphi_R$ in such a way that $-\varphi_R$ grows at a linear rate at infinity, which is independent of $R$, while $\varphi_R \to 0$ as $R \to \infty$. Indeed, it is easy to check that
\begin{equation} \label{phideath}
|\varphi_R(y)| \leq |y|^2 \left( R^2 + |y|^2 \right)^{-1/2}.
\end{equation}
Fix constants $0 < \eta < 1$ and $\theta > 1$ to be selected below. By \eqref{Hcoer} and \eqref{phidcn},
\begin{equation} \label{Ephi}
\left|\theta \tr\!\left(A(y) D^2\varphi_R\right)\right| + H(-\theta D\varphi_R,y) \leq C\theta^\gamma.
\end{equation}
Define
\begin{equation*}
\hat u : = \lambda(1+4\eta) u+ (1-\lambda(1+4\eta)) w
\end{equation*}
and
\begin{align*}
\hat u_R  : = &\ (1-\eta) \hat u + \eta \theta \varphi_R \\  = &\ \lambda(1+4\eta)(1-\eta)u + (1-\eta)(1 - \lambda(1+4\eta))w + \eta \theta \varphi_R,
\end{align*}
where $w$ is the subcorrector which satisfies \eqref{subcorreq} and \eqref{sublininfty}, constructed in the proof of Proposition~\ref{mainstep}. By subtracting a constant from $w$, we may assume that $\sup_{D} w = 0$.
Since $\lambda < 1$, we may shrink $\eta$, if necessary, to ensure that
\begin{equation} \label{etachoice}
\lambda(1+4\eta) < 1 < (1+4\eta)(1-\eta).
\end{equation}
Select $\theta : = 1+4 \limsup_{|y| \to \infty} u(y) / |y|$, and observe that, by the previous inequality, \eqref{sublininfty}, $\lambda \in \Lambda$ and the definition of $\varphi_R$, we have, for every $R> 1$,
\begin{equation}
\liminf_{|y| \to \infty} \frac{v(y) - \hat u_R(y)}{|y|} \geq \liminf_{|y| \to \infty} \frac{v(y) - \lambda u(y)}{|y|} + \liminf_{|y| \to \infty} \frac{(3\eta-4\eta^2) u(y) + \eta\theta \varphi_R(y)}{|y|} > 0.
\end{equation}

\medskip

To get a differential inequality for $\hat u_R$, we apply Lemma~\ref{convtrick} twice. The first application, using \eqref{subcorreq} and that $u$ is a subsolution of \eqref{MP}, yields that
\begin{equation*}
-\tr\!\left(A(y) D^2 \hat u \right) + H(D\hat u,y) \leq \lambda(1+4\eta) \mu + (1-\lambda(1+4\eta)) \overline H(0) \quad \mbox{in} \ \Rd\! \setminus \! D.
\end{equation*}
Combining this with \eqref{Ephi}, we deduce that
\begin{equation*}
-\tr\!\left(A(y) D^2 \hat u_R \right) + H(D\hat u_R,y) \leq \widetilde\mu(\eta) \quad \mbox{in} \ \Rd\! \setminus \! D,
\end{equation*}
where the constant $\widetilde\mu(\eta)$ is given by
\begin{equation*}
\widetilde \mu(\eta) := \lambda(1+4\eta)(1-\eta)\mu + (1-\eta)(1-\lambda(1+4\eta))\overline H(0) + C\theta^\gamma \eta.
\end{equation*}
Since $\mu > \overline H(0)$, by selecting $\eta> 0$ sufficiently small we have $\widetilde\mu(\eta) < \mu$. We may therefore apply the comparison principle to $\hat u_R$ and $v$, obtaining that
\begin{equation*}
\hat u_R - v \leq \max_{\partial D} (\hat u_R - v) \quad \mbox{in} \ \Rd\!\setminus \! D.
\end{equation*}
By sending $R\to \infty$ and using the fact that $\varphi_R\rightarrow 0$ locally uniformly, we deduce that
\begin{equation} \label{hatuvinq}
(1-\eta) \hat u - v  \leq \max_{\partial D} \big( (1-\eta) \hat u - v\big) \quad \mbox{in} \ \Rd\! \setminus\! D.
\end{equation}
Since $w$ is strictly sublinear at infinity, it follows that
\begin{equation*}
\liminf_{|y| \to \infty} \frac{v(y) - \lambda(1+4\eta)(1-\eta)u(y)}{|y|} \geq 0,
\end{equation*}
and, hence, $\overline \lambda \geq \lambda(1+4\eta)(1-\eta)$. If $\overline \lambda < 1$, then we may send $\lambda \to \overline \lambda$ to obtain that $\overline \lambda \geq \overline \lambda(1+4\eta)(1-\eta)$, a contradiction to \eqref{etachoice}. It follows that $\overline \lambda =1$.

\medskip

The preceding analysis therefore applies to any $0 <\lambda < 1$. Sending $\eta \to 0$ and then $\lambda \to 1$ in \eqref{hatuvinq} completes the proof.
\end{proof}

With a comparison principle in hand, the unique solvability of the metric problem \eqref{MP} for $\mu > \overline H(p)$ (subject to the boundary and growth conditions) follows from Perron's method. The proof in the case of bounded $V$ is accomplished more smoothly. In the unbounded case we approximate by bounded potentials.

\begin{prop} \label{metexistence}
For each fixed $p,z \in\Rd$ and $\mu > \overline H(p)$, there exists a unique solution $\ms_\mu=\ms_\mu(\cdot,z,\omega) = \ms_\mu(\cdot,z,\omega;p) \in C(\overline{ \Rd \!\setminus\! D_1(z)})$ of \eqref{MP}, a.s. in $\omega$, with $D=D_1(z) : = z+D_1$, subject to the conditions
\begin{equation} \label{bngcond}
\ms_\mu(\cdot,z,\omega) = 0 \quad \mbox{on} \  \partial D_1(z) \quad \mbox{and} \quad 0 \leq \liminf_{|y|\to\infty} \frac{\ms_\mu(y,z,\omega)}{|y|}  \ \quad  \mbox{a.s. in} \ \omega.
\end{equation}
\end{prop}

\begin{proof}
Assume, without loss of generality, that $p=0$ and $z=0$. We first argue in the case that $V$ is bounded, i.e., $V(0,\omega) \leq C$ a.s. in $\omega$. According to the first step in the proof of Proposition~\ref{mainstep}, there exists a subcorrector $w$ which satisfies, a.s. in $\omega$, the inequality
\begin{equation*}
-\tr\!\left( A(y,\omega) D^2w \right) + H(Dw,y,\omega) \leq \overline H(0) \quad \mbox{in} \ \Rd,
\end{equation*}
as well as
\begin{equation} \label{wsublinpmp}
\lim_{|y|\to\infty} \frac{w(y)}{|y|} = 0.
\end{equation}
By subtracting a constant from $w$, we may assume that $\inf_{D_1} w(\cdot,\omega) = 0$. Define
\begin{multline*}
m_\mu(y) : = \inf\Big\{ v(y) : v\in \LSC(\Rd\!\setminus \!D_1), \ v \ \mbox{is a supersolution of} \ \eqref{MP} \ \mbox{in} \ \Rd\!\setminus \! D_1,\\ v \geq 0 \ \mbox{on} \ \partial D_1, \  \mbox{and} \ \liminf_{|y| \to \infty} v(y) / |y| \geq 0 \Big\}.
\end{multline*}
To demonstrate that $m_\mu$ is well-defined, we must show that the admissible set is nonempty. Since $V$ is bounded, it is easy to check that, for $a> 0$ sufficiently large,
\begin{equation*}
v(y): =
\begin{cases}
a|y| & \mbox{if} \ A \equiv 0, \\
a(|y| + |y|^\theta) & \mbox{if} \ A \not\equiv 0 \ \ \mbox{and} \ \ \gamma > 2,  \ \ \mbox{where} \ \ \theta: = (\gamma-2) / (\gamma-1),\\
a(|y|-1) & \mbox{if} \ A \not\equiv 0 \ \ \mbox{and} \ \  \gamma \leq 2,
\end{cases}
\end{equation*}
is an admissible supersolution, and moreover, $v = 0$ on $\partial D_1$ and $\limsup_{|y|\to \infty} |y|^{-1} v(y) = a < \infty$. Thus $m_\mu$ is well-defined, and we have $w \leq m_\mu \leq v$ in $\Rd \setminus D_1$. It follows from standard viscosity theory 
that $(m_\mu)^*$ is a subsolution of \eqref{MP} and $(m_\mu)_*$ is a supersolution. Since these functions satisfy the hypotheses of Proposition~\ref{metcomparison}, we obtain that $(m_\mu)^*\leq (m_\mu)_*$, and therefore $(m_\mu)^* = (m_\mu)_*=m_\mu \in C(\Rd \!\setminus \! D_1)$ is a solution of \eqref{MP}-\eqref{bngcond}.

\medskip

By repeating the argument of Lemma~\ref{locest} we find that, for any $B(y_0,r+1/2) \subseteq \Rd\!\setminus \!D_1$,
\begin{equation} \label{MPLB}
\esssup_{B(y,r)} |Dm_\mu(y)|^\gamma \leq C \Big( 1 + \sup_{B(y,r+1/2)} V(y,\omega)\Big).
\end{equation}
Next, we mimic the proof of Lemma~\ref{impest}, using Morrey's inequality and Lemma~\ref{potests}, to obtain that, a.s in $\omega$, the oscillation of any solution of \eqref{MP} in a large ball grows at most like the radius of the ball. In particular,
\begin{equation} \label{metosc}
\limsup_{r\to \infty} \frac{1}{r} \osc_{B_r \! \setminus \! D_1} m_\mu \leq C,
\end{equation}
and $m_\mu$ grows at most like $C|y|$, where $C$ depends only the constants in our assumptions on $V$.

In the unbounded case, we approximate $H$ by bounded Hamiltonians as explained in Remark~\ref{approximation}. If $\mu > \hat H$, then for sufficiently large $k$ we have $\mu > \overline H_k$. For such $k$, there is a solution $m_{\mu,k}$ of the problem \eqref{MP}--\eqref{bngcond} for $H=H_k$, which satisfies $m_{\mu,k}(y) \leq C_1|y|$ for a constant $C_1$ independent of $k$. The comparison principle (Proposition~\ref{metcomparison}) implies that $m_{\mu,k}$ is monotone in $k$, so that $m_{\mu,k} \uparrow m_\mu$ as $k\to \infty$ to some function $m_\mu$. By the local Lipschitz estimate \eqref{MPLB}, $m_{\mu,k}$ converges locally uniformly (and a.s. in $\omega$) to $m_\mu$ which is therefore a solution of \eqref{MP}. We have $m_\mu(y) \leq C|y|$ a.s. in $\omega$, and by monotonicity we have that $\liminf_{|y| \to \infty} |y|^{-1}m_\mu(y) \geq 0$. 
If $\hat H \geq \mu > \overline H$, we observe that the solution constructed above for any $\hat \mu >\hat H$ belongs to the admissible set used for the definition of $m_\mu$ at the beginning of the ongoing proof. Then we can repeat verbatim the arguments used for the bounded case to find a solution. In either case, uniqueness follows from Proposition~\ref{metcomparison}, and the proof is complete.
\end{proof}

\begin{remark} \label{MPstat}
The functions $m_\mu$ are jointly stationary in the sense that
\begin{equation*}
m_\mu(y,x,\tau_z\omega) = m_\mu(y+z,x+z,\omega).
\end{equation*}
This follows from uniqueness and the stationarity of the coefficients in the equation.
\end{remark}

\begin{remark} \label{lingrow}
It is easy to show that, for each $\mu > \overline H$,
\begin{equation*}
\essinf_{\omega\in \Omega}\liminf_{|y| \to \infty} \frac{\ms_\mu(y,0,\omega)}{|y|} > 0.
\end{equation*}
In brief, the idea is to choose $\overline H < \widetilde \mu < \mu$, add a small multiple of $\varphi_R$ (defined in \eqref{varphiR}) to $\ms_{\widetilde \mu}$, and compare it to $\ms_\mu$ using Proposition~\ref{metcomparison}. The room provided by taking $\widetilde \mu$ smaller than $\mu$ allows to compensate for the perturbation, following along the lines of the calculation in the proof of Proposition~\ref{metcomparison}. We will see in the next section that $t^{-1} \ms_\mu(ty,0,\omega)$ has a deterministic limit (a.s. in $\omega$) as $t\to \infty$, which is equivalent to the homogenization of the metric problem.

The argument above also implies that that $\ms_\mu$ is strictly increasing in $\mu$. This permits us to define
\begin{equation*}
\ms (y,x,\omega) : = \lim_{\mu \ssearrow \overline H} \ms_\mu(y,x,\omega).
\end{equation*}
\end{remark}

\begin{remark} \label{tildemetp}
By repeating the argument in the proof of Proposition~\ref{metexistence}, we obtain for each fixed $p,z \in \Rd$ and $\mu > \overline H(p)$ a unique solution $\widetilde m_\mu=\widetilde m_\mu(y,z,\omega;p)$ of the modified metric problem
\begin{equation} \label{meteqmod}
\left\{ \begin{aligned}
& -\tr\!\left( A(y , \omega) D^2\widetilde m_\mu \right) + H(p+D\widetilde m_\mu, y, \omega) = \mu \quad \mbox{in} \ \ \Rd\!\setminus \! D_1, \\
& \widetilde m_\mu(y,z,\omega;p) = w(y,\omega;p) - w(z,\omega;p) \ \ \mbox{on}  \ \partial D_1, \\
& \liminf_{|y| \to \infty} |y|^{-1} \widetilde m_\mu(y,z,\omega) \geq 0 \ \ \mbox{a.s. in} \ \omega.
\end{aligned} \right.
\end{equation}
The only difference between $\widetilde m_\mu$ and the solution $m_\mu$ constructed in Proposition~\ref{metexistence} is the boundary conditions we impose on $\partial D_1$, which differ only in the case $A\not\equiv 0$ and $\gamma \leq 2$. To prove the existence of $\widetilde m_\mu$, there is an additional difficulty in the construction of barriers, but this is straightforward since $w$ is Lipschitz and $\partial D_1$ is the unit ball. We leave this detail to the reader. The function $\widetilde m_\mu$ is sometimes more convenient to work with than $m_\mu$, as we see below.

In the case $D_1=B_1$, in which $\widetilde m_\mu(y,z,\omega)$ is not defined for $|y-z| <1$, we extend $\widetilde m_\mu(\cdot,\omega)$ to $\Rd\times\Rd$ by defining $\widetilde m_\mu (y,z,\omega) : = w(y,\omega) - w(z,\omega)$ for $|y-z| < 1$.
\end{remark}

\subsection{The subadditivity of $\ms_\mu$}
\label{sadd}

An important property of $\ms_\mu$ is that, up to a modification in the case $\gamma \leq 2$ and $A\not\equiv 0$, it is \emph{subadditive}. This is easier to see in the case that $\gamma > 2$ or $A\equiv 0$, i.e., when $D_1 = \{ 0 \}$. With $w$ denoting the subcorrector, Proposition~\ref{metcomparison} implies that, for all $x,y\in \Rd$ and a.s. in $\omega$,
\begin{equation*}
w(y,\omega) - w(x,\omega) \leq \ms_\mu (y,x,\omega)
\end{equation*}
Reversing $x$ and $y$ and adding the two inequalities together, we obtain
\begin{equation} \label{subaddpr}
0 \leq \ms_\mu (y,x,\omega) + \ms_\mu (x,y,\omega).
\end{equation}
We next claim that, for every $x,y,z\in \Rd$ and a.s. in $\omega$,
\begin{equation} \label{msubadd}
\ms_\mu(y,x,\omega) \leq \ms_\mu (y,z,\omega) + \ms_\mu (z,x,\omega).
\end{equation}
Indeed, thinking of both sides of \eqref{msubadd} as a function of $y$ with $x$ and $z$ fixed, noting that the inequality holds at both $y=x$ and $y=z$ (the former is \eqref{subaddpr} and the latter is obvious), and applying Proposition~\ref{metcomparison} with $D = \{ x,z\}$, we obtain \eqref{msubadd}. The subadditivity property \eqref{msubadd} must be modified in the case that $D_1 = B_1$, which is recorded in the next lemma.

It is convenient to state the subadditivity in terms of $\widetilde m_\mu$ from Remark~\ref{tildemetp}. To this end, observe that there exists a random variable $c(y,\omega)$ satisfying $|y|^{-1} c(y,\omega) \to 0$ a.s. in $\omega$ and
\begin{equation} \label{cthingy}
\sup_{z\in B(y,1)} \big( |w(y,\omega) - w(z,\omega)| + |\widetilde m_\mu(y,\omega) - \widetilde m_\mu(z,\omega) | \big) \leq c(y,\omega).
\end{equation}
Indeed, according to \eqref{sublin}, \eqref{wLIP} and \eqref{MPLB}, it suffices to take
\begin{equation*}
c(y,\omega): = C\Big(1 + \sup_{B(y,2)} V(\cdot,\omega) \Big)^{1/\gamma}. 
\end{equation*}
For future reference, we observe that $c$ is stationary and
\begin{equation} \label{cint}
\E c(0,\cdot) < \infty.
\end{equation}
Define the quantity
\begin{equation*}
\hat m_\mu(y,x,\omega) : = \widetilde m_\mu(y,x,\omega) + \frac12 ( c(x,\omega) + c(y,\omega) )
\end{equation*}
We claim that $\hat m_\mu$ is subadditive.

\begin{lem}\label{subadditivity}
For each $p\in \Rd$ and $\mu > \overline H(p)$, the function $\widetilde m_\mu$ satisfies, for all $x,y,z\in \Rd$, and a.s. in $\omega$,
\begin{equation} \label{msubadd2}
\hat m_\mu(y,x,\omega) \leq \hat m_\mu(y,z,\omega) + \hat m_\mu(z,x,\omega)
\end{equation}
\end{lem}
\begin{proof}
The proof in the case  $\gamma > 2$ or $A\equiv 0$ is given in the discussion preceding the lemma. The case $\gamma \leq 2$ and $A\not\equiv 0$, i.e., $D_1=B_1$, is handled by a modification of this argument. We fix $x$ and $z$, write \eqref{msubadd2} in the form
\begin{equation} \label{msubadd3}
\widetilde m_\mu(y,x,\omega) \leq \widetilde m_\mu(y,z,\omega) +  \widetilde m_\mu(z,x,\omega) + 2c(z,\omega),
\end{equation}
and, thinking of both sides of \eqref{msubadd3} as functions of $y$, we observe that the desired inequality follows from Proposition~\ref{metcomparison} provided we can show that it holds in $B(x,1) \cup B(z,1)$.

According Proposition~\ref{metcomparison}, for every $y,z\in \Rd$, we have
\begin{equation} \label{subadds1}
w(y,\omega) - w(z,\omega) \leq \widetilde m_\mu (y,z,\omega).
\end{equation}
Using this inequality, we obtain, for all $|y-x| \leq 1$, $z\in \Rd$ and a.s. in $\omega$,
\begin{equation*}
\widetilde m_\mu(y,x,\omega) = w(y,\omega) - w(x,\omega) \leq \widetilde m_\mu(y,z,\omega) + \widetilde m_\mu(z,x,\omega).
\end{equation*}
For $|y-z| \leq 1$, we use \eqref{cthingy} to get
\begin{align*}
\widetilde m_\mu(y,x,\omega) & \leq c(z,\omega) + \widetilde m_\mu (z,x,\omega) + w(y,\omega) - w(z,\omega) \\&  = \widetilde m_\mu (y,z,\omega) + \widetilde m_\mu (z,x,\omega) + c(z,\omega).
\end{align*}
This completes the proof.
\end{proof}

\subsection{The homogenization of the metric problem}
\label{mho}

We consider, for $\mu > \overline H(p)$ and 
\begin{equation*}
D_\ep : = \begin{cases} \{ 0 \} & \mbox{if} \ \gamma > 2 \ \mbox{or} \ A \equiv 0, \\ B_\ep & \mbox{if} \ \gamma \leq 2 \ \mbox{and} \ A \not\equiv 0,\end{cases}
\end{equation*}
the boundary value problem
\begin{equation} \label{meteq}
\left\{ \begin{aligned}
& -\ep \tr\!\left( A(\tfrac x\ep , \omega) D^2m^\ep_\mu \right) + H(p+Dm^\ep_\mu, \tfrac x\ep, \omega) = \mu \quad \mbox{in} \ \ \Rd\!\setminus \! D_\ep, \\
& m^\ep_\mu = 0 \ \ \mbox{on}  \ \partial D_\ep, \quad \liminf_{|y| \to \infty} |y|^{-1} m^\ep_\mu(y,\omega) \geq 0 \ \ \mbox{a.s. in} \ \omega,
\end{aligned} \right.
\end{equation}
which is nothing but a rescaling of \eqref{MP}--\eqref{bngcond}. In particular, \eqref{meteq} has a unique solution given by
\begin{equation} \label{rescalemeteq}
\ms^\ep_\mu(x,\omega) = \ep \ms_\mu\left(\frac x\ep,0,\omega\right).
\end{equation}

\medskip

We show that the $m^\ep_\mu$'s converge, as $\ep\to 0$, a.s. in $\omega$ and in $C(\Rd \setminus \{0\})$, to the unique solution $\overline m_\mu=\overline m_\mu(y) = \overline m_\mu(y;p)$ of 
\begin{equation} \label{meteq-bar}
\left\{ \begin{aligned}
& \overline H(p+D\overline m_\mu ) = \mu \quad \mbox{in} \ \ \Rd\!\setminus \! \{ 0 \}, \\
& \overline m_\mu(0) = 0, \quad \liminf_{|y| \to \infty} |y|^{-1}\, \overline m_\mu(y) \geq 0,
\end{aligned} \right.
\end{equation}
which, in view of Proposition~\ref{metexistence}, is well-posed for every $\mu > \overline H(p)$. 

\begin{lem}
For each $\mu > \overline H(p)$, 
\begin{equation}\label{mbarform}
\overline m_\mu(y;p) = \sup \big\{ y\cdot q :  \overline H(p+q) \leq \mu \big\}.
\end{equation}
\end{lem}
\begin{proof}
Since the right hand side of \eqref{mbarform} is a subsolution of \eqref{meteq-bar}, the inequality ``$\geq$" in \eqref{mbarform} follows at once from Proposition~\ref{metcomparison}.

To obtain the reverse inequality, notice first that $\overline m_\mu$ is positively 1-homogeneous in $y$, i.e., for every $t \geq 0$,
\begin{equation*}
\overline m_\mu(ty;p) = t \overline m_\mu(y;p).
\end{equation*}
The scaling invariance of \eqref{meteq-bar} and the uniqueness of $\overline m_\mu$. Likewise, the convexity of $\overline H$, Lemma~\ref{convtrick} and Proposition~\ref{metcomparison} are easily combined to yield that $\overline m_\mu(y;p)$ is convex in $y$. If \eqref{mbarform} fails, we can find a point $y\neq 0$ such that $\overline m_\mu(\cdot;p)$ is differentiable at $y$ and $\overline  m_\mu(y;p) > \sup\{ y\cdot q: \overline H(p+q) \leq \mu \}$. Since $\overline m_\mu(y;p) = y \cdot D\overline m_\mu(y;p)$, it follows that $H(p+D\overline m_\mu(y;p)) > \mu$. This contradicts the elementary viscosity-theoretic fact that $\overline m_\mu$ must satisfy $H(p+D\overline m_\mu) = \mu$ at any point of differentiability.
\end{proof}

\begin{remark} \label{touch-mbar}
We can see from \eqref{meteq-bar} and the convexity of $\overline H$ that
\begin{equation*}
\overline m_\mu(y;p) = \sup\big\{ y\cdot q: \overline H(p+q) = \mu\big\}.
\end{equation*}
Moreover, if $\overline H(p+q_0) = \mu$, then $q_0$ belongs to the subdifferential $\partial \overline m_\mu(x_0;p)$ of $\overline m_\mu$ at some point $x_0\neq 0$. That is, the plane $x \cdot q_0$ must touch $\overline m_\mu$ from below at some nonzero $x_0 \in \Rd$. To see this, consider the largest value of $\lambda$ for which
\begin{equation*}
\overline m_\mu(y;p) \geq \lambda y\cdot q_0.
\end{equation*}
By the positive homogeneity of both sides, we need only check the above inequality on the unit sphere, and for the same reason, our claim follows if we can show that $\lambda =1$. But if $\lambda > 1$, then we would obtain a contradiction as follows: first, by convexity of $\overline H$ and $\lambda > 1$, we have $\overline H(p+\lambda q_0 ) > \mu$. Then, by the convexity of the sublevel sets of $\overline H$, we can choose a vector $z\in \Rd\setminus \{ 0 \}$ and $\ep > 0$ such that $z\cdot (\lambda q_0) \geq z\cdot q+\ep$ for every $\overline H(p+q)\leq \mu$. But this implies $z\cdot (\lambda q_0) > \overline m_\mu(z;p)$, a contradiction.
\end{remark}

From \eqref{rescalemeteq}, we see that the limit, as $\ep\to 0$, of the $m_\mu^\ep$'s is equivalent to studying the limit of $t^{-1} \ms_\mu(ty,0,\omega)$ as $t \to \infty$. Therefore we may state the result on the homogenization of \eqref{meteq} as follows.

\begin{prop} \label{homogMP}
For each $p,y\in \Rd$ and $\mu > \overline H(p)$,
\begin{equation} \label{homogMPeq}
\lim_{t\to \infty} \frac1t \ms_\mu(ty,0,\omega;p) = \overline m_\mu(y;p) \quad \mbox{a.s. in} \ \omega.
\end{equation}
\end{prop}
\begin{proof}
By Proposition~\ref{metcomparison}, the functions $\widetilde m(\cdot,0,\omega)$ and $m_\mu(\cdot, 0,\omega)$ differ by no more than $c(0,\omega)$. Likewise, $\hat m_\mu(y,0,\omega)$ and $\widetilde m_\mu(y,0,\omega)$ differ by at most $\frac12( c(0,\omega) + c(y,\omega))$, which is strictly sublinear in $y$, a.s. in $\omega$. In light of Lemma~\ref{subadditivity} and \eqref{cint}, and using the semigroup $\sigma_t:= \tau_{ty}$, we may apply the subadditive ergodic theorem (Proposition~\ref{SAergthm}) to obtain that, for each $y\in \Rd$ and a.s. in $\omega$,
\begin{equation} \label{preliMP}
\lim_{t\to \infty} t^{-1} m_\mu(ty,0,\omega) = \lim_{t\to \infty} t^{-1} \hat m_\mu(ty,0,\omega) = M_\mu(y,\omega),
\end{equation}
for some $M_\mu(y,\omega)$. At this point, we must allow for the possibility that $M_\mu$ is random because we do not assume $\{ \sigma_t \}$ is ergodic.

It remains to show that $M_\mu = \overline m_\mu$. We first show that $M_\mu$ is constant a.s. in $\omega$. This follows from the stationary and ergodic assumptions in the usual way. According to Lemma~\ref{potests}, \eqref{MPLB} and Remark~\ref{MPstat}, for each $\omega \in \Omega_1$ and $z\in \Rd$, we have
\begin{equation*}
\limsup_{t \to \infty} t^{-1} m_\mu(ty,0,\tau_z\omega) = \limsup_{t \to \infty} t^{-1} m_\mu(ty+z,z,\omega) = \limsup_{t \to \infty} t^{-1} m_\mu(ty,0,\omega),
\end{equation*}
i.e., the set $\{ \omega \in \Omega_1: \limsup_{t\to \infty} t^{-1} m_\mu(ty,0,\omega) \leq k \}$ is invariant under $\tau_z$, for each $k\in \R$. Since $\Omega_1$ has full probability, the ergodic hypothesis implies that 
\begin{equation*}
\limsup_{t \to \infty} t^{-1} m_\mu(ty,0,\tau_z\omega) \equiv C \quad \mbox{a.s. in} \  \omega.
\end{equation*}
Hence $M_\mu(y,\omega) = M_\mu(y)$ a.s. in $\omega$.

It is clear that $M_\mu(y)$ is positively homogeneous. According to \eqref{bngcond}, it is thus nonnegative. The estimates \eqref{oscest} and \eqref{MPLB} imply that $M_\mu(y)$ is Lipschitz in $y$.

To complete the proof that $M_\mu = \overline m_\mu$, we show that $M_\mu$ is the solution of \eqref{meteq-bar}. Suppose that $\varphi$ is a smooth function and $x_0 \neq 0$ are such that
\begin{equation} \label{Mmuslm}
x \mapsto M_\mu(x) - \varphi(x) \quad \mbox{has a strict local maximum at} \quad x=x_0.
\end{equation}
We show by a perturbed test function argument that
\begin{equation} \label{Mmuwts}
\overline H(p+D\varphi(x_0)) \leq \mu.
\end{equation}
 Arguing by contradiction, we assume that $\theta: = \overline H(p+D\varphi(x_0)) - \mu > 0$. Take $\{ \delta_j \}$ to be the subsequence described in Remark~\ref{assubseq}, along which we have \eqref{assubseqeq}. Set $p_1 : = p+D\varphi(x_0)$, let $\lambda > 1$ to be selected below, and define the perturbed test function
\begin{equation*}
\varphi_j(x) : = \varphi(x) + \lambda\Big( \delta_j v^{\delta_j} \big(\frac x{\delta_j}, \omega; p_1 \big) + \overline H(p_1)\Big).
\end{equation*}
We claim that, for all sufficiently large $j$ and sufficiently small $r> 0$, $\varphi_j$ satisfies
\begin{equation} \label{vjclmsdf}
-\delta_j \tr(A(x/\delta_j,\omega) D^2\varphi_j) + H(p+D\varphi_j,x/\delta_j,\omega) \geq \mu + \frac12 \theta \quad \mbox{in} \ B(x_0,r).
\end{equation}
Since $\varphi_j$ is not smooth in general, we verify the inequality in the viscosity sense. To this end, select a smooth function $\psi$ and a point $x_1 \in B(x_0,r)$ at which $\varphi_j - \psi$ has a local minimum. It follows that
\begin{equation*}
y\mapsto \lambda v^{\delta_j}(y,\omega;p_1) - \delta_j^{-1}\big( \psi(\delta_jy) - \varphi(\delta_jy)  \big) \quad \mbox{has a local minimum at} \quad y= x_1/\delta_j.
\end{equation*}
Using \eqref{movepup} with $q=p+ D\varphi(x_1)$, we obtain
\begin{multline} \label{ptfm1p}
\lambda \delta_j v^{\delta_j}(x_1/\delta_j,\omega;p_1) - \delta_j\tr\big( A(x_1/\delta_j,\omega) (D^2\psi(x_1) - D^2\varphi(x_1) \big) \\ + H(p+D\psi(x_1), x_1/\delta_j , \omega) \geq (1-\lambda) H\left( \frac{D\varphi(x_1) - \lambda D\varphi(x_0)}{1-\lambda},y,\omega \right).
\end{multline}
Fix $\lambda>1$ to be sufficiently close to $1$ so that, by \eqref{assubseqeq}, for all large $j$ we have
\begin{equation*}
-\lambda \delta_j v^{\delta_j} (x_1/\delta_j,\omega;p_1) \geq \overline H(p_1) - \frac{1}{8}\theta.
\end{equation*}
We may allow $r$ and $j$ to depend on $\varphi$, and so by taking $j$ larger still we may assume that
\begin{equation*}
\delta_j \big|\tr(A(x_1/\delta_j,\omega) D^2\varphi(x_1))\big| \leq \frac18 \theta.
\end{equation*}
Next, observe that by taking $r> 0$ to be small, depending on $\lambda$ and $\varphi$, then we obtain
\begin{equation*}
|\lambda-1|^{-1} \big| D\varphi(x_1) - \lambda D\varphi(x_0)\big| \leq |\lambda-1|^{-1} \big|D\varphi(x_1) - D\varphi(x_0) \big| + |D\varphi(x_0)| \leq 2 |D\varphi(x_0)|.
\end{equation*}
By shrinking $\lambda$, depending only on $\varphi$, it follows from \eqref{Hbound} that 
\begin{equation*}
(1-\lambda) H\left( \frac{D\varphi(x_1) - \lambda D\varphi(x_0)}{1-\lambda},y,\omega \right) \geq -C(1-\lambda) \big( 1+ |D\varphi(x_0)|^\gamma \big) \geq -\frac14 \theta.
\end{equation*}
The observations above yield, for small $r>0$ and large $j$,
\begin{equation*}
-\delta_j \tr(A(x_1/\delta_j,\omega) D^2\psi)  + H(p+D\psi,x_1/\delta_j,\omega) \geq \overline H(p_1) - \frac12 \theta = \mu + \frac12 \theta.
\end{equation*}
This confirms the claim \eqref{vjclmsdf} in the viscosity sense.

The comparison principle implies that $m^{\delta_j}_\mu - \varphi_j$ cannot have a local maximum in $B(x_0,r)$. Sending $j\to \infty$, we obtain a contradiction to \eqref{Mmuslm}. This completes the proof that $M_\mu$ is a subsolution of \eqref{meteq-bar}. The argument that $M_\mu$ is a supersolution of \eqref{meteq-bar} is very similar. The perturbed test function $\varphi_j$ is defined in the same way, but with $\lambda < 1$, and we use \eqref{movepdn} instead of \eqref{movepup}. We leave the details to the reader. We conclude that $M_\mu$ is a solution of \eqref{meteq-bar}. By uniqueness, $M_\mu = \overline m_\mu$.
\end{proof}

\begin{remark} \label{nomuassump}
In the proof above, the assumption that $\mu > \overline H(p)$ is only needed for the existence of $m^\ep_\mu$, i.e., it is not used in the perturbed test function argument which identifies $M_\mu$ with $\overline m_\mu$. 
\end{remark}

\begin{remark}
In the setting of Hamilton-Jacobi equations, the classical perturbed test function argument (see Evans~\cite{E}) typically requires uniform Lipschitz estimates so that the error induced in the perturbation inside the coercive $H$ can be controlled. In our setting, the unboundedness of the potential prevents us from having such Lipschitz estimates. To get around this technical glitch we have introduced a new, modified version of the perturbed test function argument which utilizes the convexity of $H$. This is the purpose of introducing the parameter $\lambda$ in the proof of Proposition~\ref{homogMP}, and this technical device will be repeated several times in this paper. 
\end{remark}

\subsection{A characterization of $\min\overline{H}$}

Using the methods developed in the previous subsections, we notice that the minimum of $\overline H$ can be characterized by the solvability of the inequality
\begin{equation} \label{charmHeqcl}
-\tr \!\left( A(y, \omega) D^2 v \right) + H(Dv,y, \omega) \leq \mu \quad \mbox{in} \ \Rd
\end{equation}
where, instead of asking for sublinear growth at infinity, we impose the growth condition
\begin{equation} \label{charmHgc}
\limsup_{|y| \to \infty} |y|^{-1} v(y) < \infty. 
\end{equation}
The observation plays an important role in the proof of Proposition~\ref{bigstep}.

\begin{prop} \label{charmH}
The following formula holds a.s. in $\omega$:
\begin{equation} 
\min_{p\in \Rd} \overline H(p) = \inf \big\{ \mu \, : \mbox{there exists} \ \ v \in C(\Rd) \ \ \mbox{satisfying} \ \ \eqref{charmHeqcl}  \ \ \mbox{and} \ \ \eqref{charmHgc} \big\}.\label{charmHeq}
\end{equation}
\end{prop}
\begin{proof}
Let $\widetilde H(\omega)$ denote the left side of \eqref{charmHeq}. It is easy to see, by the ergodic hypothesis and the stationarity of the coefficients, that $\widetilde H$ is constant; i.e., $\widetilde H(\omega) = \widetilde H$ a.s. in $\omega$. According to the existence of the subcorrector in Proposition~\ref{mainstep}, we have that $\widetilde H \leq \min_{p\in \Rd} \overline H(p)$. It remains to show the reverse inequality. 

Select $\mu > \widetilde H$ and define
\begin{equation*}
s_\mu(y,x,\omega) : = \sup\big\{ v(y-x) : v\in C(\Rd) \ \mbox{satisfies} \ v \leq 0 \ \ \mbox{on} \ D_1,  \ \eqref{charmHeqcl} \ \ \mbox{and} \ \ \eqref{charmHgc} \big\}. 
\end{equation*}
Then using Proposition~\ref{metcomparison} and barrier functions built in Proposition~\ref{metexistence} (we may take for example $m_\mu(y,0,\omega)$ with any $p$), $s_\mu$ is well-defined, jointly stationary in the sense of Remark~\ref{MPstat}, and growing at most linearly at infinity. By similar arguments as in the proof of Proposition~\ref{metexistence}, we see that, for each fixed $z\in \Rd$, $s_\mu(y,z,\omega)$ is a solution of the equation
\begin{equation*}
-\tr \!\left( A(y, \omega) D^2 v \right) + H(Dv,y, \omega) = \mu \quad \mbox{in} \ \Rd \setminus D_1(z).
\end{equation*}
By repeating the arguments of Sections~\ref{sadd} and~\ref{mho}, we obtain the homogenization of $s_\mu$. We first observe that $s_\mu$ has sufficient properties to apply the subadditive ergodic theorem, and we are able to obtain, a.s. in $\omega$,
\begin{equation*}
\frac1t s_\mu(ty,0,\omega) \rightarrow \overline s_\mu(y) \quad \mbox{as} \ t\to \infty.
\end{equation*}
Then, by the proof of Proposition~\ref{homogMP} (see also Remark~\ref{nomuassump}), the function $\overline s_\mu$ is a solution of the equation
\begin{equation*}
\overline H(D\overline s_\mu) = \mu \quad \mbox{in} \ \Rd\setminus \{ 0 \}. 
\end{equation*}
Therefore, $\min \overline H \leq \mu$. This holds for any $\mu > \widetilde H$, and so we conclude that $\min \overline H \leq \widetilde H$. 
\end{proof}

\section{The proof of Homogenization}



\label{PH}

In this section we complete the proof of Theorem~\ref{MAIN}.

\subsection{The almost sure homogenization of the macroscopic problem} \label{BS}

We now use Propositions~\ref{homogMP} and~\ref{charmH} to obtain an improvement of Proposition~\ref{mainstep}. Informally, the limit \eqref{bigstepeq} says that the functions $\hat v^\delta(y,\omega): = v^\delta(\cdot,\omega) - v^\delta(0,\omega)$ are ``approximate correctors on balls of radius~$\sim1/\delta$." This is what we need to prove Theorem~\ref{MAIN}.

\begin{prop} \label{bigstep}
There is a subset $\Omega_8 \subseteq \Omega_7$ of full probability such that, for each $R> 0$ and $\omega\in \Omega_8$,
\begin{equation} \label{bigstepeq}
\lim_{\delta \to 0} \sup_{B_{R/\delta}} \big| \delta v^\delta(\cdot,\omega;p) + \overline H(p) \big| = 0.
\end{equation}
\end{prop}
\begin{proof}
Lemma~\ref{CDE-noxd} reduces the work to showing that \eqref{bigstepeq} holds, a.s. in $\omega$, for each fixed $p\in \Rd$. The general result is then obtained and $\Omega_8$ is defined by intersecting the relevant subsets, say $\Omega_p$, over a countable dense subset of $p$, and appealing to \eqref{CDE-noxdeq}. Therefore we fix $p\in \Rd$.

\medskip

We split the proof into two steps. The first is to use the homogenization of the metric problem to obtain almost sure convergence of $\delta v^\delta $ to $-\overline H$ at the origin. To accomplish this we use a perturbed test function argument \emph{in reverse}, arguing in effect that, if we improperly perturb our test function, then we could not have homogenization. In the second step, we combine the ergodic and Egoroff's theorems with Lemma~\ref{impest} to improve the almost sure convergence at the origin to balls of radius $\sim1/\delta$. Unfortunately, we must perform the two steps in reverse order, since we need the result of second step to perform the first.

\medskip

\emph{Step 1. Reducing the question to convergence at the origin.}
According to Proposition~\ref{mainstep} and Remark~\ref{WDlimsup},
\begin{equation} \label{limsorg}
\liminf_{\delta \to 0} \delta v^\delta (0,\omega) = -\overline H(p) \leq -\hat H(p) : = \limsup_{\delta \to 0} \delta v^\delta(0,\omega) \quad \mbox{a.s. in} \ \omega.
\end{equation}
We claim that, for each $R> 0$,
\begin{equation} \label{toballs1del}
\liminf_{\delta \to 0} \inf_{B_{R/\delta}} \delta v^\delta (\cdot,\omega) = -\overline H(p) \quad \mbox{and} \quad \limsup_{\delta \to 0} \sup_{B_{R/\delta}} \delta v^\delta(\cdot,\omega) =  -\hat H(p) \quad \mbox{a.s. in} \ \omega.
\end{equation}
Since the arguments are nearly identical, we only prove the first identity of \eqref{toballs1del}. From \eqref{limsorg} and Egoroff's theorem, for every $\rho > 0$, there exists $\bar\delta(\rho)> 0$ and a set $E_\rho \subseteq \Omega$ such that $\Prob[E_\rho] \geq 1-\rho$ and, for every $0 < \delta \leq \bar \delta(\rho)$,
\begin{equation*}
\inf_{\omega\in E_\rho} \delta v^\delta(0,\omega) + \overline H \geq -\rho.
\end{equation*}
The ergodic theorem provides, for each $\rho > 0$, a subset $F_\rho \subseteq \Omega$ of full probability such that, for every $\omega \in F_\rho$,
\begin{equation*}
\lim_{R \to \infty} \fint_{B_R} \mathds{1}_{E_\rho}(\tau_y\omega) \, dy = \Prob[E_\rho] \geq 1-\rho.
\end{equation*}
Define $F_0 : = \cap_{j=1}^\infty (F_{2^{-j}}\cap  \Omega_3)$ with $\Omega_3$ as in Lemma~\ref{impest}. Then $\Prob[F_0] =1$. Fix $\omega \in F_0$ and $R, \rho > 0$, with $\rho = 2^{-j}$ for some $j\in \N$. It follows that, if $\delta > 0$ is sufficiently small (depending on $\omega$, $R$ and $\rho$), then
\begin{equation} \label{fillingup}
| \{ y\in B_{R/\delta} : \tau_y \omega \in E_\rho \} | \geq (1-2\rho) |B_{R/\delta}|.
\end{equation}
Using \eqref{uber-oscest} and shrinking $\delta$, depending again on $\omega$, $\rho$ and $R$, we have, for $r\in (\rho R, R)$,
\begin{equation} \label{oscestapp}
\sup_{y\in B_R} \osc_{B(y/\delta, r/\delta)} \delta v^\delta(\cdot,\omega) \leq 2C r.
\end{equation}
Select any $z\in B_{R/\delta}$. In a similar way as in the proof of Lemma~\ref{impest}, we may use \eqref{fillingup} to find a point $y\in B_{R/\delta}$ with $|y-z|\leq C\rho R\delta^{-1}$ and $\tau_y\omega \in E_\rho$. In view of  \eqref{oscestapp} and the stationarity of the $v^\delta$'s, we deduce that, for each $\delta$ sufficiently small, depending on $\omega$, $\rho$, and $R$,
\begin{align*}
\delta v^\delta(z,\omega) + \overline H & \geq -| \delta v^\delta(z,\omega) - \delta v^\delta(y,\omega)| + \delta v^\delta(y,\omega) + \overline H
\\ & \geq -2C_1(C\rho R) + \delta v^\delta(0,\tau_y\omega) + \overline H \\ & \geq -C\rho R - \rho,
\end{align*}
and, hence, for each $\omega \in F_0$ and $R>0$,
\begin{equation*}
\liminf_{\delta\to 0} \inf_{z\in B_{R/\delta}} \big( \delta v^\delta(z,\omega) + \overline H\big) \geq 0,
\end{equation*}
from which \eqref{toballs1del} follows.

\medskip

\emph{Step 2. The convergence at the origin in the case $\overline H(p) = \min \overline H$.} 
We show that $\hat H(p) \geq \widetilde H$, the latter introduced in the proof of Proposition~\ref{charmH}. For each fixed $\omega \in \Omega_7$, so that \eqref{hatHdefeq} holds, observe that for we can pass to limit as $\delta\to 0$ along a subsequence of $\hat v^\delta(y,\omega): = v^\delta(\cdot,\omega) - v^\delta(0,\omega)$ to find a solution of 
\begin{equation*}
-\tr \!\left( A(y, \omega) D^2 v \right) + H(p+Dv,y, \omega) = \hat H(p) \quad \mbox{in} \ \Rd.
\end{equation*}
It then follows from the definition of $\widetilde H$ that $\widetilde H \leq \hat H(p)$. According to our assumption and Proposition~\ref{charmH}, we have $\overline H(p) = \min_{\Rd} \overline H = \widetilde H$, and so we conclude that $\overline H(p) = \hat H(p)$. This completes the proof in the case that $p$ lies on the level set $\{ q\in \Rd : \overline H(q) = \min \overline H \}$. 

\medskip

\emph{Step 3. The convergence at the origin in the case $\overline H(p) > \min \overline H$.} To complete the proof of \eqref{bigstepeq}, it remains to show that $\overline H(p) = \hat H(p)$. We may assume that
\begin{equation*}
\min_{\Rd} \overline H  = \overline H(0) < \overline H(p).
\end{equation*}

Arguing by contradiction, we suppose that $\rho : = \overline H(p) - \hat H(p) > 0$. Fix $\omega \in\Omega_{7}$ and select a subsequence $\delta_j \to 0$, which may depend on $\omega$, such that
\begin{equation*}
\lim_{j \to \infty} {\delta_j} v^{\delta_j} (0,\omega;p) = -\hat H(p)= -\overline H(p) + \rho.
\end{equation*}
Put $\mu := \overline H(p)$. According to Remark~\ref{touch-mbar}, we can pick $x_0 \neq 0$ such that $\overline m_\mu (x_0; 0 ) = x_0 \cdot p$. With $\eta > 0$ to be selected below, define
\begin{equation*}
w_j(x):=  (x+x_0)\cdot p + \delta_j v^{\delta_j} \Big(\frac x{\delta_j},\omega;p\Big) - \eta |x|^2.
\end{equation*}
It is clear that, for $r,\eta>0$ small enough and $j$ large enough, $w_j$ is a subsolution of
\begin{equation*}
-\delta_j \tr \!\Big( A\Big(\frac x{\delta_j}, \omega\Big) D^2w_j \Big) + H\Big(Dw_j,\frac x{\delta_j}, \omega\Big) \leq \mu - \frac12\rho  \quad \mbox{in} \ B(0,r).
\end{equation*}
Denoting by $m^\ep_\mu$ the unique solution of \eqref{meteq}, we see, from the comparison principle and the above inequality, that
\begin{equation} \label{nolocmin}
\min_{x\in B(0,r)} \big( m^{\delta_j}_\mu(x,-x_0,\omega;0) - w_j(x) \big) = \min_{x\in \partial B(0,r)}\big( m^{\delta_j}_\mu(x,-x_0,\omega;0) - w_j(x) \big).
\end{equation}
Letting $j\to \infty$, we deduce a contradiction as follows. Grouping the terms as
\begin{multline*}
m^{\delta_j}_\mu(x,-x_0,\omega;0) - w_j(x) = \big( m^{\delta_j}_\mu(x,-x_0,\omega;0) - \overline m_\mu(x+x_0;0) \big) + \big( \overline m_\mu(x+x_0;0) - p\cdot (x+x_0) \big)  \\ + \big( -\delta_j v^{\delta_j}(x/\delta_j,\omega;p) + \eta|x|^2 \big),
\end{multline*}
we see that the first term converges uniformly to 0 in $B(0,r)$ as $j\to \infty$ by Proposition~\ref{homogMP}, the second term is nonnegative and vanishes at $x=0$, and the last term satisfies, by \eqref{toballs1del},
\begin{equation*}
\lim_{j\to \infty} \inf_{x \in B(0,r)} \big( -\delta_j v^{\delta_j}(x/\delta_j,\omega;p) + \eta|x|^2 \big) = \hat H(p) + \eta r^2 > \hat H(p) \geq \lim_{j\to \infty} \big( -\delta_j v^{\delta_j}(0,\omega;p) \big).
\end{equation*}
Therefore \eqref{nolocmin} is impossible for large $j$, and so we are forced to conclude that $\overline H(p) = \hat H(p)$. Now \eqref{toballs1del} yields \eqref{bigstepeq}.
\end{proof}

\subsection{The unique solvability of \eqref{HJ}}

We pause our march toward a proof of Theorem~\ref{MAIN} for a brief word regarding the well-posedness of \eqref{HJ}. It is not the main purpose of this paper to deal with such issues, and providing a complete proof of the unique solvability of \eqref{HJ} requires little more than repeating the arguments in Section~\ref{AMP-wp} or Section~\ref{AMP} and rearranging them to follow Section 6 of \cite{LS2}. For this reason, we omit the proof of the following proposition.

\begin{prop}\label{epWP}
There exists $\Omega_9 \subseteq \Omega_8$ such that, for every $u_0 \in \BUC(\Rd)$, there exists a unique solution $u^\ep=u^\ep(\cdot,\omega) \in C(\Rd\times\R_+)$ of \eqref{HJ} which is bounded below on $\Rd \times[0,T]$ for each $T>0$ and $\omega\in \Omega_9$. Moreover, there exists a constant $C >0$ such that, for each $(x_0,t_0)\in \Rd\times \R_+$, $r> 0$ and $\omega\in \Omega_9$,
\begin{equation} \label{eposcbd}
\limsup_{\ep\to 0} \osc_{B(x_0,r) \times (t_0,t_0+r)} u^\ep(\cdot,\omega) \leq C r.
\end{equation}
\end{prop}

\subsection{The proof of the main result}

We now assemble the ingredients developed in the previous sections into a proof of the homogenization result. The argument is based on the classical perturbed test function method 
and is also very similar to the argument in the proof of Proposition~\ref{homogMP}.

\begin{proof}[{Proof of Theorem~\ref{MAIN}}]
Fix an $\omega \in \Omega_9$ so that both \eqref{bigstepeq} and \eqref{eposcbd} hold. According to \eqref{eposcbd}, we may extract a subsequence $\ep_j \to 0$ and $u\in (C(\Rd \times\R_+) \cap \BUC(\Rd \times [0,T]))$, for every $T> 0$, such that, as $j\to \infty$,
\begin{equation*}
u^{\ep_j} \rightarrow u \quad \mbox{locally uniformly in} \ \Rd\times\R_+.
\end{equation*}
We claim that $u$ is the unique solution of the problem \eqref{HJ-eff}.

\medskip

Here we verify only that $u$ is a supersolution of \eqref{HJ-eff}, since the argument for showing $u$ is a subsolution is similar (and both are also similar to the proof Proposition~\ref{homogMP}. To this end we select a smooth test function $\varphi$ and a point $(x_0,t_0) \in \Rd\times\R_+$ such that
\begin{equation} \label{touch1}
(x,t) \mapsto (u - \varphi)(x,t) \quad \mbox{has a strict local minimum at} \ (x_0,t_0).
\end{equation}
We must show that
\begin{equation*}
\varphi_t(x_0,t_0) + \overline H(D\varphi(x_0,t_0)) \geq 0.
\end{equation*}
Suppose on the contrary that
\begin{equation*}
-\theta : = \varphi_t(x_0,t_0) + \overline H(D\varphi(x_0,t_0)) < 0.
\end{equation*}
Set $p_0 : = D\varphi(x_0,t_0)$, let $0< \lambda < 1$ be a parameter to be selected below, and define the perturbed test function
\begin{equation*}
\varphi^\ep(x,t) : = \varphi(x,t) + \lambda\ep v^\ep\!\left(\tfrac x\ep, \omega; p_0 \right)\!,
\end{equation*}
where $v^\ep$ is the solution of the auxiliary problem \eqref{aux} with $\delta = \ep$ and $p=p_0$. While we cannot expect that $\varphi^\ep$ is smooth, we claim that, for $r,\ep> 0$ sufficiently small, $\varphi^\ep$ is a viscosity solution of
\begin{equation} \label{corrclm}
\varphi^\ep_t - \ep \tr \!\left( A(\tfrac x\ep, \omega) D^2\varphi^\ep \right) + H(D\varphi^\ep,\tfrac x\ep, \omega) \leq -\frac12 \theta \quad \mbox{in} \ B(x_0,r) \times (t_0-r,t_0+r).
\end{equation}
To verify \eqref{corrclm}, select a smooth $\psi$ and a point $(x_1,t_1) \in B(x_0,r)\times (t_0-r,t_0+r)$ such that
\begin{equation*}
(x,t) \mapsto (\varphi^\ep - \psi)(x,t) \quad \mbox{has a local maximum at} \ (x_1,t_1).
\end{equation*}
We may rewrite this as
\begin{equation*}
(y,t) \mapsto \lambda v^\ep(y,\omega; p_0) - \frac1\ep \left( \psi(\ep y,t) - \varphi(\ep y,t) \right) \quad \mbox{has a local maximum at} \ \left(\tfrac{x_1}\ep,t_1\right).
\end{equation*}
In particular, $\varphi_t(x_1,t_1) = \psi_t(x_1,t_1)$. Using that $v^\ep$ is a viscosity supersolution of \eqref{movepdn} with $q= D\varphi(x_1,t_1)$, we deduce that
\begin{multline}
\lambda \ep v^\ep \left( \tfrac {x_1}\ep, \omega;p_0\right) - \ep \tr \!\left( A\left(\tfrac {x_1}\ep,\omega\right) \left(D^2\psi(x_1,t_1) - D^2\varphi(x_1,t_1) \right)\right) \\
+ H\!\left(D\psi(x_1,t_1), \tfrac {x_1}{\ep}, \omega\right) \leq (1-\lambda) H\left( \frac{D\varphi(x_1,t_1) - \lambda D\varphi(x_0,t_0)}{1-\lambda}, y, \omega \right).
\end{multline}
In a similar way as in the proof of Proposition~\ref{homogMP}, we can take $r> 0$ small, $j$ large, and $\lambda$ close to $1$, all depending on $\varphi$ but not on $\psi$, so that
\begin{equation*}
\lambda \ep v^\ep(x_1/\ep,\omega;p_0) \geq -\overline H(p_0) - \frac18 \theta,
\end{equation*}
\begin{equation*}
\big|\varphi_t(x_1,t_1) - \varphi_t(x_0,t_0)\big| + \ep \big| \tr \!\left(A\left(x_1/\ep,\omega\right) D^2\varphi(x_1,t_1) \right)\big| \leq \frac18\theta,
\end{equation*}
and
\begin{equation*}
(1-\lambda) H\left( \frac{D\varphi(x_1,t_1) - \lambda D\varphi(x_0,t_0)}{1-\lambda}, y, \omega \right) \leq \frac14 \theta
\end{equation*}
Combining the above inequalities yields 
\begin{equation*}
\psi_t(x_1,t_1) - \ep \tr \!\left( A\left(\tfrac {x_1}\ep,\omega\right) D^2\psi(x_1,t_1) \right) + H\!\left(D\psi(x_1,t_1), \tfrac {x_1}{\ep}, \omega\right) \leq - \frac12 \theta,
\end{equation*}
and hence \eqref{corrclm}.

\medskip

According to the comparison principle, the map $(x,t) \mapsto (u^\ep - \varphi^\ep)$ cannot have a local maximum in the set $B(x_0,r) \times (t_0-r,t_0+r)$. This is the contradiction to \eqref{touch1}. 

\medskip

The solution of \eqref{HJ-eff} is unique, according to Remark~\ref{effHcmp}. This implies the local uniform convergence of the full sequence $u^\ep(\cdot,\omega)$ to $u$ for each $\omega$ in a set of full probability, and completes the proof of the theorem.
\end{proof}

\section{Further properties of the metric problem} \label{P}

\subsection{A probabilistic interpretation}
The analysis of the metric problem allows us to make a precise connection to some of the results of Sznitman \cite{Sz1,Sz3, Sz2}. What we describe here is found in Chapter 5 of \cite{Szb}. For the reader's convenience, we deviate a bit from the notation in the rest of this paper in an effort to be more faithful to that of \cite{Szb}.

In what follows we denote by $Z=Z_t$ the canonical process on $\Rd$, $P_0$ is a Wiener measure, and $E_0$ is the expectation with respect to $P_0$. Sznitman studies the behavior of the Brownian motion $Z$ moving in a random environment composed of soft Poissonian obstacles. The latter are represented by a potential $V$ given by \eqref{Vexam}.
The obstacles have the interpretation of ``killing" the Brownian motion, that is, the medium is absorbing the diffusion at a rate $V$.

\medskip

Specifically, Sznitman studies the quantity
\begin{equation*}
e_\mu(y,\omega) : = E_0 \left[ \exp\left( -\int_0^{h(y)} (V(y+Z_s,\omega) + \mu) \, ds \right) : \ h(y) < \infty \right],
\end{equation*}
where $\mu \geq 0$ is a parameter and $h(y): = \inf\{ s> 0 : y+ Z_s \in B_1 \}$ is the first hitting time of the unit ball for the Brownian motion translated by $y$. The goal is to understand how $e_\mu$ behaves for very large $|y|$. According to the Feynman-Kac formula, $e_\mu$ is a solution, $\Prob$-a.s., of
\begin{equation*}
-\Delta e_\mu +\left( V(y,\omega) +\mu \right) e_\lambda= 0 \quad \mbox{in} \ \Rd\!\setminus\! B_1,
\end{equation*}
and $0 < e_\mu \leq 1$ in $\Rd \!\setminus \! B_1$ and $e_\mu = 1$ on $\partial B_1$.

\medskip

The change of variables
\begin{equation*}
a_\mu(y,\omega): = -\log e_\mu(y,\omega),
\end{equation*}
yields that $a_\mu \geq 0$ solves
\begin{equation*}
\left\{ \begin{aligned}
& -\Delta a_\mu + |Da_\mu|^2 - V(y,\omega) = \mu & \mbox{in} & \ \Rd\!\setminus \! B_1,\\
& a_\mu = 0 & \mbox{on} & \ \partial B_1, \\
& a_\mu \geq 0 & \mbox{in} & \ \Rd\!\setminus \! B_1.
\end{aligned} \right.
\end{equation*}
Recalling from Proposition~\ref{invispot} that in this case $\overline H(0) = 0$, we see that, for $\mu > 0$, $a_\mu$ is the unique solution $m_\mu$ of the metric problem given by Proposition~\ref{metexistence} and, according to Proposition~\ref{homogMP},
\begin{equation*}
\lim_{|y| \to \infty} \frac{|a_\mu(y,\omega) - \alpha_\mu(y)|}{|y|} = 0 \quad \mbox{a.s. in} \ \omega,
\end{equation*}
where $\alpha_\mu(y): = \overline m_\mu(y)$ is the solution of \eqref{meteq-bar}. Using the fact that, according to Proposition~\ref{metcomparison}, the functions $\alpha_\mu$ and $a_\mu$ are monotone in $\mu$, we deduce that, for each $\Lambda > 0$,
\begin{equation} \label{shapethm}
\lim_{|y| \to \infty} \sup_{0 < \mu \leq \Lambda} \frac{|a_\mu(y,\omega) - \alpha_\mu(y)|}{|y|} = 0 \quad \mbox{a.s. in} \ \omega.
\end{equation}
The homogenization assertion \eqref{shapethm} is what Sznitman calls the ``shape theorem," which he proves using entirely probabilitstic methods (see \cite[Theorem 2.5 in Chapter 5]{Szb}). The functions $\alpha_\mu = \overline m_\mu$ are called the \emph{quenched Lyapunov exponents}.

\medskip

The homogenization result in Proposition~\ref{homogMP} is more general than Sznitman's shape theorem, as we consider more general Hamiltonians and diffusions, and even in the special case $H = |p|^2 + V(y,\omega)$, more general potentials $V$.

\subsection{New formulae for $\overline H$}

Since $\overline m_\mu$ is positively 1-homogeneous, $D\overline m_\mu$ is constant along rays starting from the origin, and the function $\overline m_\mu$ is determined by its values on the unit sphere. It follows that $\overline m_\mu \geq 0$. Moreover, since $\mu > \overline H(p)$, it is clear that $\overline m_\mu$ cannot be touched from below by a constant function at any point of $\Rd \!\setminus \!\{0\}$, and, therefore, $\overline m_\mu > 0$ in $\Rd\! \setminus \! \{0\}$. Then by homogeneity, we have $\overline m_\mu(y) \geq c|y|$ for some $c> 0$.

\medskip

The convexity of $\overline H$ and $\mu > \overline H(p)$ imply that the level set $\{ q \in \Rd: \overline H(p+q) = \mu \}$ has empty interior. It follow that the level set above equals the closure of the image of $D\overline m_\mu$. In this way, one may recover the effective Hamiltonian $\overline H$ from the functions $\overline m_\mu$, since we may fix $p_0 \in \Rd$ with $\overline H (p_0) = \min \overline H$ and write
\begin{equation} \label{barHbarm}
\overline H(p)  = \inf\big\{ \mu > \min \overline H \, : \, \overline m_\mu(y;p_0) > (p-p_0)\cdot y\, \mbox{ for all } y\in \Rd \big\}.
\end{equation}

From the results and arguments above, we can characterize the effective Hamiltonian in terms of the solvability of the metric problem.

\begin{prop} The effective Hamiltonian is given by
\begin{equation}
\overline H(p) = \inf\!\left\{ \mu \in \R : \mbox{there exists} \ u=u(y,\omega) \ \mbox{satisfying} \ \eqref{MP}\mbox{--}\eqref{bngcond}, \ \mbox{a.s. in} \ \omega \right\}. \label{Hformula}
\end{equation}
\end{prop}
\begin{proof}
We first argue that, for $\mu < \overline H(p)$, \eqref{meteq-bar} has no subsolution $u$ satisfying \eqref{meteq-bar}. Assuming to the contrary that, for some $\mu < \overline H(p)$,  \eqref{MP} has a subsolution $u$, we argue that $u$ cannot satisfy the growth condition at infinity. To this end, for each $R> 1$, consider the smooth function $\phi_R:(0,\infty) \to \R$
\begin{equation*}
\phi_R(t): = \left( 1 - (t-R)^2 \right)^{1/2} - (R^2+1)^{1/2},
\end{equation*}
which satisfies 
\begin{equation*}
\phi_R(0) = 0, \quad \phi_R(t) \leq -ct \quad \mbox{on} \ [0,R/2] \quad \mbox{and} \quad |\phi'_R(t)| \leq C \quad \mbox{on} \ [0,\infty),
\end{equation*}
for constants $C,c> 0$ independent of $R> 1$, and, finally, $t^{-1}\phi_R(t) \rightarrow 1$ as $t\to \infty$. Define $\varphi_R(y): = \ep \phi_R(|y|)$, and observe that $\varphi_R$ is smooth and that, by the continuity of $\overline H$, for sufficiently small $\ep > 0$,
\begin{equation*}
\overline H(p+D\varphi_R) > \mu.
\end{equation*}
Since $\varphi_R(0) = 0$ and $\varphi_R(y) \sim \ep |y|$ for large $|y|$, we deduce from Proposition~\ref{metcomparison} that $u < \varphi_R$ in $\Rd\setminus \{ 0 \}$. In particular, we obtain
\begin{equation*}
u(y) \leq -c\ep |y| \quad \mbox{for all} \ |y| \leq R/2.
\end{equation*}
Sending $R \to \infty$, we see that $\limsup_{|y|\to \infty} u(y) / |y| \leq -c \ep$.

Next, suppose there exists a function $u=u(y,\omega)$ satisfying \eqref{MP}--\eqref{bngcond}, a.s. in $\omega$. Then by repeating the arguments in the proof of Theorem~\ref{MAIN}, we can find an $\omega$ and a subsequence $\ep_j\to 0$ for which
\begin{equation*}
\ep_j u(\ep_j^{-1} y, \omega) \rightarrow \widetilde u \quad \mbox{locally uniformly in} \ \Rd\!\setminus \! D_1,
\end{equation*}
and such that $\widetilde u$ satisfies \eqref{MP}. Hence $\mu \geq \overline H(p)$.
\end{proof}

We conclude by noting that the statements in Theorem~\ref{MAIN1} follow from the results in Section~\ref{Q} and the proof of the above proposition.

\appendix
\section{Auxiliary lemmata} \label{A}

Recorded here are supplemental facts needed in various arguments in the main body of the paper. The first item we present is a viscosity solution theoretic lemma, and the ones that follow are elementary observations from measure theory.

\medskip

The following lemma is easy to prove if one of $u_1$ or $u_2$ is smooth (which, with an exception in the proof of Proposition~\ref{metcomparison}, is how we use it in this paper). The general case is well-known to experts. Since we cannot find a reference, we indicate a proof for the convenience of the reader.

\begin{lem} \label{convtrick}
Assume that $A=\Sigma\Sigma^t$, where $\Sigma=\Sigma(y)$ is Lipschitz, and $H=H(p,x)$ is convex in $p$ and satisfies, for all $p,x,y\in \Rd$,
\begin{equation} \label{Hcthyp}
H(p,x) - H(p,y) \leq C(1+|p|) |x-y|.
\end{equation}
Suppose that $U\subseteq \Rd$ is open and $u_1,u_2,f_1,f_2 \in \USC(U)$ satisfy, for $i=1,2$, the inequalities
\begin{equation} \label{convtrickeq}
-\tr\!\left(A (x) D^2u_i\right) + H(Du_i,x) \leq f_i \quad \mbox{in} \ U
\end{equation}
in the viscosity sense. Then, for each $0 < \lambda < 1$, $u: = \lambda u_1 + (1-\lambda)u_2$ is a viscosity solution of
\begin{equation} \label{convtrickconc}
-\tr\!\left(A (x) D^2u\right) + H(Du,x) \leq \lambda f_1 + (1-\lambda) f_2 \quad \mbox{in} \ U.
\end{equation}
\end{lem}
\begin{proof}
Suppose that $\phi$ is a smooth function and $x_0\in U$ such that
\begin{equation*}
x\mapsto (u-\phi)(x) \quad \mbox{has a strict local maximum at} \quad x_0.
\end{equation*}
By restricting to a neighborhood of $x_0$, we may assume that $U$ is bounded, $u_1$ and $u_2$ are bounded from above on $U$, and $u-\phi$ achieves its global maximum at $x_0$. Define, for each $\ep > 0$, the auxiliary function
\begin{equation*}
\Psi_\ep(x,y) : = \lambda u_1(x) + (1-\lambda) u_2(y) - \phi(x) - \frac{1}{\ep}|x-y|^2.
\end{equation*}
It follows that, for sufficiently small $\ep > 0$, there exists $(x_\ep,y_\ep) \in U\times U$ such that
\begin{equation*}
\Psi_\ep (x_\ep,y_\ep) = \sup_{U\times U} \Psi.
\end{equation*}
Using \cite[Lemma 3.1]{CIL}, we get
\begin{equation} \label{UGct1}
\lim_{\ep \to 0} \frac{|x_\ep-y_\ep|^2}{\ep} = 0 \quad \mbox{and} \quad \lim_{\ep\to 0}(x_\ep,y_\ep) = (x_0,x_0).
\end{equation}
According to the maximum principle for semicontinuous functions \cite[Theorem 3.2]{CIL}, there exist symmetric matrices $X_\ep$ and $Y_\ep$ such that
\begin{equation} \label{matmadd}
\left( X_\ep, \frac{x_\ep-y_\ep}{\ep} \right) \in \overline{\mathcal{J}}^{2,+} (\lambda u_1 -\phi)(x_\ep) \quad \mbox{and} \quad \left( Y_\ep, \frac{y_\ep-x_\ep}{\ep} \right) \in \overline{\mathcal{J}}^{2,+} ((1-\lambda) u_2)(y_\ep)
\end{equation}
and
\begin{equation*}
\begin{pmatrix} X_\ep & 0 \\ 0 & Y_\ep \end{pmatrix} \leq \frac{3}{\ep} \begin{pmatrix} \iden & -\iden \\ -\iden & \iden \end{pmatrix}.
\end{equation*}
We refer to \cite{CIL} for the definition of the jets $\overline{\mathcal{J}}^{2,\pm}$ and their role in viscosity solution theory. Using the inequalities \eqref{convtrickeq}, we obtain
\begin{equation} \label{plugXep}
-\tr\!\left( A(x_\ep) (X_\ep+D^2\phi(x_\ep) )\right) + \lambda H\!\left(\frac{1}{\lambda}\left(\frac{x_\ep-y_\ep}{\ep}+ D\phi(x_\ep)\right),x_\ep\right) \leq \lambda f_1(x_\ep),
\end{equation}
and
\begin{equation} \label{plugYep}
-\tr\!\left( A(y_\ep) Y_\ep \right) + (1-\lambda) H\!\left(\frac{1}{1-\lambda} \frac{y_\ep-x_\ep}{\ep},y_\ep\right) \leq (1-\lambda) f_2(x_\ep).
\end{equation}
By \eqref{Hcthyp} and \eqref{UGct1},
\begin{equation} \label{Hxeye}
H\!\left(\frac{1}{1-\lambda} \frac{y_\ep-x_\ep}{\ep},x_\ep\right) \leq H\!\left(\frac{1}{1-\lambda} \frac{y_\ep-x_\ep}{\ep},y_\ep\right) + o(1) \quad \mbox{as} \ \ep \to 0.
\end{equation}
Multiplying the matrix inequality \eqref{matmadd} on the left by the nonnegative definite matrix
\begin{equation*}
\begin{pmatrix} \Sigma(x_\ep) \\ \Sigma(y_\ep)  \end{pmatrix}\begin{pmatrix} \Sigma(x_\ep) \\ \Sigma(y_\ep)  \end{pmatrix}^T = \begin{pmatrix} \Sigma(x_\ep) \Sigma(x_\ep)^T &  \Sigma(x_\ep)\Sigma(y_\ep)^T \\ \Sigma(y_\ep)\Sigma(x_\ep)^T & \Sigma(y_\ep)\Sigma(y_\ep)^T\end{pmatrix},
\end{equation*}
taking the trace of both sides, using that $\Sigma$ is Lipschitz and \eqref{UGct1}, we find that
\begin{equation} \label{amtrineqd}
\tr\!\left(A(x_\ep)X_\ep - A(y_\ep)Y_\ep \right) \leq \frac3\ep\!\left| \Sigma(x_\ep) - \Sigma(y_\ep) \right|^2 \leq \frac{3|x_\ep-y_\ep|^2}{\ep} \rightarrow 0 \quad \mbox{as} \ \ep \to 0.
\end{equation}
Combining \eqref{plugXep}, \eqref{plugYep}, \eqref{Hxeye}, \eqref{amtrineqd} and the convexity of $H$ yields
\begin{equation*}
-\tr\!\left( A(x_\ep)D^2\phi(x_\ep) \right) + H(D\phi(x_\ep),x_\ep) \leq \lambda f_1(x_\ep) + (1-\lambda) f_2(y_\ep) + o(1) \quad \mbox{as} \ \ep \to 0.
\end{equation*}
Sending $\ep \to 0$, we obtain
\begin{equation*}
-\tr\!\left( A(x_0)D^2\phi(x_0) \right) + H(D\phi(x_0),x_0) \leq \lambda f_1(x_0) + (1-\lambda) f_2(x_0). \qedhere
\end{equation*}
\end{proof}

The following two elementary facts were used in \cite{CS} as a convenient way of using mixing properties to control averages. We put them to similar use in Section~\ref{AMP} to estimate the average of a power of $V$ in a large balls; see Lemma~\ref{potests}. The proofs are elementary and short, and so we include them for the convenience of the reader.

\begin{lem}\label{mix-decay}
Assume the mixing condition \eqref{mixing} holds, and suppose that $f$ and $g$ are random variables taking values in $[0,\mu]$, $\mu > 0$, which are measurable with respect to $\mathcal{G}(U)$ and $\mathcal{G}(V)$, respectively, where $\dist(U,V) \geq r > 0$. Then
\begin{equation}
\left| \E[fg] - \E[f] \E[g] \right| \leq C\mu^2 r^{-\beta}.
\end{equation}
\end{lem}
\begin{proof}
We have
\begin{align*}
\E[fg] - \E[f] \E[g] = \int_{[0,\mu]\times[0,\mu]} \left( \Prob\!\left[ f > t \ \mbox{and} \ g > s \right] - \Prob[f> t] \Prob[g > s]\right) \, ds\, dt.
\end{align*}
Observe that the absolute value of the integrand on the right side is at most $Cr^{-\beta}$.
\end{proof}

\begin{lem} \label{mix-EVest}
Let $h, h_1,\ldots, h_k$ be a sequence of identically distributed random variables such that, for some $a > 0$ and each $i, j\in\{ 1,\ldots, k\}$, we have
\begin{equation*}
 \E[h_ih_j] - (\E[h])^2 \leq a.
\end{equation*}
Then we have
\begin{equation} \label{Emixcontrol}
\E\left( \frac1k \sum_{i=1}^k h_i \right)^2 \leq (\E[h])^2 + \frac{1}{k} \Var(h) + \frac{k-1}{k}a ,
\end{equation}
and
\begin{equation} \label{Vmixcontrol}
\Var\left( \frac1k \sum_{i=1}^k h_i \right)^2 \leq \frac{1}{k} \Var(h) + \frac{k-1}{k}a.
\end{equation}
\end{lem}
\begin{proof}
First, observe that \eqref{Vmixcontrol} is merely a rearrangement of \eqref{Emixcontrol}. To obtain the latter, we write
\begin{align*}
\E\left( \frac1k \sum_{i=1}^k h_i \right)^2 & = \frac{1}{k^2} \sum_{i,j=1}^k \E[h_ih_j] \\
& = \frac1k E[h^2] + \frac{1}{k^2} \sum_{i\neq j}  \E[h_ih_j] \\
& \leq \frac{1}{k} (\E h)^2 + \frac{1}{k} \Var(h) + \frac{k-1}{k} \left( (\E h)^2 + a \right) \\
& = (\E[h])^2 + \frac{1}{k} \Var(h) + \frac{k-1}{k}a. \qedhere
\end{align*}
\end{proof}

\begin{remark}\label{remcheb}
Recall that Chebyshev's inequality provides control of $\Prob[ f > t]$ in terms of $\E f$ and $\Var(f)$. Indeed, for any $t> \mu : = \E f$,
\begin{equation*}
\Prob[f > t] \leq (t-\mu)^{-2} \, \E (f - \mu)^2 = (t-\mu)^{-2} \Var(f).
\end{equation*}
This observation allows us to put \eqref{Emixcontrol} and \eqref{Vmixcontrol} to good use in the proof of Lemma~\ref{potests}.
\end{remark}

Arguing as in the proof of Theorem 9 in Kozlov \cite{K}, we show that a process $w$ with stationary, mean zero gradient must be strictly sublinear at infinity. This is a critical ingredient in the proof of Proposition~\ref{mainstep}.

\begin{lem} \label{koz}
Suppose that $w:\Rd\times\Omega\to \R$ and $\Phi = Dw$ in the sense of distributions, a.s. in $\omega$. Assume $\Phi$ is stationary, $\E \Phi(0,\cdot) = 0$, and $\Phi(0,\cdot) \in L^\alpha(\Omega)$ for some $\alpha > \d$. Then
\begin{equation}  \label{sublininftyA}
\lim_{|y|\to\infty} |y|^{-1}w(y,\omega) = 0 \quad \mbox{a.s. in} \ \omega.
\end{equation}
\end{lem}
\begin{proof}
The ergodic theorem implies that there exists a set $\widetilde \Omega \subseteq \Omega$ of full probability, such that, for every bounded $V\subseteq \Rd$ and every $\omega \in \widetilde \Omega$,
\begin{equation} \label{Phizrmn}
\lim_{\ep \to 0} \fint_{V} \Phi(\tfrac x\ep, \omega) \, dx = \E \Phi(0,\cdot) = 0,
\end{equation}
and
\begin{equation}\label{Phibdalm}
\lim_{\ep \to 0} \fint_{V} |\Phi(\tfrac x\ep, \omega)|^\alpha \, dx = \E |\Phi(0,\cdot)|^\alpha  \leq C_V.
\end{equation}

Fix $\omega \in \widetilde \Omega$ and let
\begin{equation*}
w_\ep (x) : = \ep w(\tfrac x\ep, \omega).
\end{equation*}
By the Sobolev embedding theorem, we may suppose that $w_\ep \to w_0$, as $\ep \to 0$, locally uniformly in $\Rd$, and that $Dw_\ep \rightharpoonup Dw_0$ weakly in $L^\alpha(B_R;\Rd)$, for every $R> 0$. It follows from \eqref{Phizrmn} that $Dw_0 \equiv 0$, so that $w_0$ is constant. Since $w_\ep (0) = \ep w(0) \rightarrow 0$ as $\ep \to 0$, we must have $w_0 \equiv0$. It follows that the full sequence $w_\ep$ converges locally uniformly to $0$ as $\ep \to 0$. By reverting to our original scaling, we obtain \eqref{sublininftyA} for each $\omega \in \widetilde \Omega$.
\end{proof}

We conclude with a simple measure theoretic lemma needed in the proof of Proposition~\ref{mainstep}. It is nearly the same as \cite[Lemma 1]{LS3}, and plays an identical role as in that paper.

\begin{lem} \label{meastheor}
Suppose that $(X,\mathcal{G}, \mu)$ is a finite measure space, and $\{ f_\ep \}_{\ep > 0} \subseteq L^1(X,\mu)$ is a family of $L^1(X,\mu)$ functions such that $\liminf_{\ep \to 0} f_\ep \in L^1(X,\mu)$, and
\begin{equation} \label{mtass}
f_\ep \rightharpoonup \liminf_{\ep \to 0} f_\ep \quad \mbox{weakly in} \ L^1(X,\mu) \quad \mbox{as} \ \delta \to 0.
\end{equation}
Then
\begin{equation*}
f_\ep \rightarrow \liminf_{\ep \to 0} f_\ep \quad \mbox{in} \ L^1(X,\mu) \quad \mbox{as} \ \delta \to 0.
\end{equation*}
In particular, $f_\ep \rightarrow \liminf_{\ep \to 0} f_\ep$ in measure.
\end{lem}
\begin{proof}
Define $g_\ep : = \inf_{0 < \delta \leq \ep} f_\delta$ and $f_0 := \liminf_{\ep \to 0}f_\ep$. Observing that $f_\ep - g_\ep \geq 0$ and $g_\ep \uparrow f_0 \in L^1(X,\mu)$, we use \eqref{mtass} and the monotone convergence theorem to obtain
\begin{align*}
\limsup_{\ep \to 0} \| f_\ep - f_0 \|_{L^1(X,\mu)} & \leq \limsup_{\ep \to 0} \| f_\ep - g_\ep \|_{L^1(X,\mu)} + \limsup_{\ep \to 0} \| g_\ep - f_0\|_{L^1(X,\mu)} \\
& = \limsup_{\ep \to 0} \int_{X} (f_\ep - g_\ep) \, d\mu + \limsup_{\ep \to 0} \int_{X} (f_0-g_\ep) \, d\mu \\
& \leq  \limsup_{\ep \to 0} \int_{X} (f_\ep - f_0) \, d\mu + 2\limsup_{\ep \to 0} \int_{X} (f_0 -g_\ep) \, d\mu \\
& = 0. \qedhere
\end{align*}
\end{proof}

\subsection*{Acknowledgements}
The first author was partially supported by NSF Grant DMS-1004645 and the second author by NSF Grant DMS-0901802.

\bibliographystyle{plain}
\bibliography{potentials}

\end{document}